\documentclass[12pt]{article}

\title{4-Manifold Invariants From Hopf Algebras}
\author{Julian Chaidez$^{\,1}$, Jordan Cotler$^{\,2}$, Shawn X. Cui$^{\,3,4}$}
\date{%
    \footnotesize{$^1$\textit{Department of Mathematics, University of California Berkeley, Berkeley, CA 94720, USA}}\\%
\vskip.2cm
    \footnotesize{$^2$\textit{Stanford Institute for Theoretical Physics, Stanford University, Stanford, CA 94305, USA}}\\%
\vskip.2cm
    \footnotesize{$^3$\textit{Department of Mathematics, Purdue University, West Lafayette, IN 47907, USA}}\\%
\vskip.2cm
    \footnotesize{$^4$\textit{Department of Physics and Astronomy, Purdue University, West Lafayette, IN 47907, USA}}\\%
    \blfootnote{\texttt{jchaidez@berkeley.edu, jcotler@stanford.edu, cui177@purdue.edu}}
}

\usepackage{eucal}
\usepackage{amssymb}
\usepackage{latexsym}
\usepackage{tikz-cd}
\usepackage{amsmath}
\usepackage{amsthm}
\usepackage{amscd}
\usepackage{overpic}
\usepackage{enumitem}
\usepackage{hyperref}

\usepackage{tikz}
\usetikzlibrary{decorations.markings}
\usetikzlibrary{intersections,decorations.pathmorphing,decorations.fractals}

\usepackage{comment}

\usepackage{url}

\usepackage{color}

\makeatletter
\def\blfootnote{\xdef\@thefnmark{}\@footnotetext}
\makeatother

\newcommand{\mc}[1]{{\mathcal #1}}

\numberwithin{equation}{section}

\newtheorem{theorem}{Theorem}[section]
\newtheorem{proposition}[theorem]{Proposition}
\newtheorem{corollary}[theorem]{Corollary}
\newtheorem{lemma}[theorem]{Lemma}
\newtheorem{lemma-definition}[theorem]{Lemma-Definition}

\newtheorem{conjecture}[theorem]{Conjecture}

\theoremstyle{definition}

\newtheorem{definition}[theorem]{Definition}
\newtheorem{remark}[theorem]{Remark}
\newtheorem{procedure}[theorem]{Procedure}

\newtheorem{example}[theorem]{Example}
\newtheorem{convention}[theorem]{Convention}
\newtheorem{notation}[theorem]{Notation}

\newcommand{\C}{{\mathbb C}}

\newcommand{\R}{{\mathbb R}}

\newcommand{\Z}{{\mathbb Z}}
\newcommand{\Sp}{{S}}

\newcommand{\op}{\operatorname}

\newcommand{\End}{\op{End}}

\newcommand{\Ker}{\op{Ker}}

\newcommand{\Tr}{\op{Tr}}

\newcommand{\vu}{\nu}

\newcommand{\bpm}{\begin{pmatrix}}
\newcommand{\epm}{\end{pmatrix}}

\newcommand{\pair}{\langle - \rangle}
\newcommand{\Rep}{\text{Rep}}
\newcommand{\cC}{\mathcal{C}}
\newcommand{\cD}{\mathcal{D}}
\newcommand{\Hopf}[1]{\mathcal{H}_{#1}}

\newcommand{\Ima}{\op{Im}}

\newcommand{\Op}{\text{op}}
\newcommand{\Cop}{\text{cop}}
\tikzset{
    myarrow/.style n args = {2}{
    postaction = decorate,
    decoration={
    markings,
    mark=at position {#1} with {\arrow{#2}}}
    }
}

\tikzset{
    ncbar angle/.initial=90,
    ncbar/.style={
        to path=(\tikztostart)
        -- ($(\tikztostart)!#1!\pgfkeysvalueof{/tikz/ncbar angle}:(\tikztotarget)$)
        -- ($(\tikztotarget)!($(\tikztostart)!#1!\pgfkeysvalueof{/tikz/ncbar angle}:(\tikztotarget)$)!\pgfkeysvalueof{/tikz/ncbar angle}:(\tikztostart)$)
        -- (\tikztotarget)
    },
    ncbar/.default=0.5cm,
}
\tikzset{square left brace/.style={ncbar=0.5cm}}
\tikzset{square right brace/.style={ncbar=-0.5cm}}

\tikzset{round left paren/.style={ncbar=0.5cm,out=120,in=-120}}
\tikzset{round right paren/.style={ncbar=0.5cm,out=60,in=-60}}

\begin{document}

\setcounter{tocdepth}{2}

\maketitle

\begin{abstract}
The Kuperberg invariant is a topological invariant of closed 3-manifolds based on finite-dimensional Hopf algebras.  In this paper, we initiate the program of constructing 4-manifold invariants in the spirit of Kuperberg's 3-manifold invariant. We utilize a structure called a Hopf triplet, which consists of three Hopf algebras and a bilinear form on each pair subject to certain compatibility conditions. In our construction, we present 4-manifolds by their trisection diagrams, a four-dimensional analog of Heegaard diagrams. The main result is that every Hopf triplet yields a diffeomorphism invariant of closed 4-manifolds.  In special cases, our invariant reduces to Crane-Yetter invariants and generalized dichromatic invariants, and conjecturally Kashaev's invariant.  As a starting point, we assume that the Hopf algebras involved in the Hopf triplets are semisimple. We speculate that relaxing semisimplicity will lead to even richer invariants.
\end{abstract}

% 
%The Kuperberg invariant is a topological invariant of closed 3-manifolds based on finite-dimensional Hopf algebras. The invariant takes as input a single Hopf algebra, and when that Hopf algebra is semisimple, the invariant reduces to the Turaev-Viro-Barrett-Westbury invariant defined from the representation category of the Hopf algebra. In this paper, we initiate the program of  constructing 4-manifold invariants in the spirit of Kuperberg’s 3-manifold invariant. The algebraic data used in our construction is a structure called a Hopf triplet, which consists of three Hopf algebras and a bilinear form on each pair of them subject to certain compatibility conditions. In the construction, we present 4-manifolds by their trisection diagrams, a four-dimensional analog of Heegaard diagrams. The main result is that for each Hopf triplet, there is a scalar associated with any closed 4-manifold, which is an invariant of diffeomorphism classes. Our invariant recovers the Crane-Yetter invariant as well as some cases of the generalized dichromatic invariant. Conjecturally it also specializes to Kashaev's invariant. As a starting point in the current paper, we have assumed that the Hopf algebras involved in the Hopf triplets are semisimple. We speculate that a relaxation on the semisimplicity would lead to more interesting invariants, which was the original motivation for this program.

\newpage
\tableofcontents
\newpage

\section{Introduction}
Since the discovery of the Jones polynomial of knots \cite{jones1997polynomial}  and its interpretation in terms of a topological quantum field theory (TQFT) \cite{witten1989quantum, witten1988topological} in the 1980s, the field of quantum topology and TQFTs has seen substantial progress. Deep connections have been discovered between quantum topology and such disparate areas as knot theory, low dimensional topology, quantum groups, tensor categories, conformal field theory, and topological quantum computing. 

Roughly, for a non-negative  integer $d$, a TQFT in dimension $d+1$, or a $(d+1)$-TQFT,  is an assignment of vector spaces to $d$-manifolds and vectors/scalars to $(d+1)$-manifolds subject to certain compatibility conditions. All manifolds involved are assumed to be smooth. In particular, each $(d+1)$-TQFT provides an invariant, called a quantum invariant, for smooth $(d+1)$-manifolds. Quantum invariants have important applications in smooth topology as they can be used to distinguish different manifolds. A fundamental and well-established family of TQFTs in dimension $(2+1)$ is the Reshetikhin-Turaev/Witten-Chern-Simons theory \cite{turaev1994quantum, turaev1994quantum, witten1989quantum}.  $(3+1)$-TQFTs are especially interesting from the perspective of 4-manifolds. Below, the notions of TQFTs and quantum invariants are referred to interchangeably.

The construction of TQFTs is closely related to algebraic structures such as quantum groups, tensor categories, and more generally higher categories. In dimension $(3+1)$, perhaps the simplest example of a TQFT is the Dijkgraaf-Witten theory based on finite groups \cite{dijkgraaf1990topological}. This was generalized in one direction to the Yetter TQFT based on finite categorical groups or 2-groups \cite{yetter1993tqft}, and was generalized in another direction to the Crane-Yetter/Walker-Wang TQFT based on ribbon fusion categories \cite{crane1993categorical, crane1997state, walker2012top}. More recently, the third author of the current paper proposed a construction \cite{cui2016higher} based on crossed braided fusion categories  which simultaneously generalized the Yetter and Crane-Yetter/Walker-Wang TQFTs. Finally, Douglas and Reutter \cite{douglas2018fusion} pinned down the notion of spherical fusion 2-categories and used it to define invariants of 4-manifolds, further generalizing the invariants from crossed braided fusion categories. There are also a few other invariants of 4-manifolds such as the dichromatic invariant \cite{petit2008dichromatic, barenz2016dichromatic} based on pivotal functors and the Kashaev invariant indexed by a finite cyclic group \cite{kashaev2014asimple}. These are speculated to be special cases of the invariants mentioned above, but a proof of this is not known.

The Douglas-Reutter invariant from spherical fusion 2-categories is believed to be the most general state-sum type invariant. However, it has several known limitations. To the authors' best knowledge, there are not many examples of spherical fusion 2-categories apart from the ones constructed from crossed braided fusion categories (plus some cohomology twistings) and the ones arising as the module categories of braided fusion categories which should correspond to the Crane-Yetter theory. Moreover, from a practical point of view, both the data encoding a spherical fusion 2-category and the state-sum formulation of the invariant have very large complexity, which makes calculations intractable beyond a few simple examples. An alternative formulation of the invariant in terms of handlebody decomposition may help with the calculations. But most importantly, it is speculated (at a non-rigorous level) that all invariants of 4-manifolds of state-sum type or those from fully extended TQFTs that are based on semisimple algebraic data are not sensitive to smooth structures.   None of the invariants mentioned above are known to distinguish smooth structures. Thus, it may be necessary to construct invariants from non-semisimple data.

In one dimension lower, the Kuperberg invariant \cite{kuperberg1996noninvolutory}, which is constructed from a finite-dimensional Hopf algebra, is a fundamental invariant of 3-manifolds. When the Hopf algebra is semisimple, the invariant recovers the Turaev-Viro-Barrett-Westbury theory \cite{turaev1992state, barrett1996invariants}. However, the invariant is more powerful when the Hopf algebra is non-semisimple. In this case, the invariant contains information about some additional structures of the 3-manifold, such as combings and framings. A generalization of the invariant from Hopf algebras in the category of vector spaces to those in a symmetric fusion category is also possible \cite{kashaev2018generalized}.  L{\'o}pez-Neumann \cite{lopez2019kuperberg} studied the invariant associated with involutory (possibly non-semisimple) Hopf algebras in the category of super vector spaces and showed that the invariant specializes to the Reidemeister torsion invariant (cf.~\cite{turaev2001introduction}) which is closely related to the Seiberg-Witten theory. 

\subsection{Main Results}

In this paper, we initiate a program of constructing 4-manifold invariants in the spirit of Kuperberg's 3-manifold invariant.

The algebraic data used in our construction is a structure called a {\it Hopf triplet}, which consists of three finite-dimensional Hopf algebras and a bilinear form on each pair of them satisfying certain compatibility conditions (see Definition \ref{def:hopf_triplet}). As a starting point in this paper, we assume that all the Hopf algebras involved are semisimple. There are several ways of producing examples of Hopf triplets. For instance, any quasi-triangular Hopf algebra gives rise to a Hopf triplet (see Example \ref{ex:basic_examples_of_triplets}).

The topological data used in the construction is a presentation of a 4-manifold, in the form of a trisection diagram \cite{gk2016}. A trisection diagram is a 4-dimensional analog of a Heegaard diagram, consisting of three families of circles on a closed surface.

\vspace{3pt} 

The first main result of this paper addresses the construction and well-definedness of our invariant. Informally, it can be stated as follows.
\begin{theorem}[see \S \ref{subsec:main_definition_and_properties}] Given a Hopf triplet $\mathcal{H}$ over a field $k$ of characteristic zero, we may associate a scalar $\tau_{\mathcal{H}}(X) \in k$ to any closed smooth $4$-manifold $X$, and this scalar is a diffeomorphism invariant.
\end{theorem}
\noindent The invariant is constructed by associating, to each circle in a trisection diagram, a generalized comultiplication tensor chosen from an appropriate Hopf algebra and contracting the tensors via the bilinear forms in a similar manner as the Kuperberg invariant. 

\vspace{3pt}

The second main result of this paper addresses the relationship between our invariant and existing invariants. In particular, we prove that (again, informally stated here)
\begin{theorem}[see \S \ref{sec:CY_dichro}] \label{thm:trisection_vs_crane_yetter_informal} Let $\mathcal{H}$ be the Hopf triplet associated to a quasi-triangular Hopf algebra $H$, and let $\mathcal{C}$ be the ribbon fusion category of representations of $H$.  Then the trisection invariant $\tau_{\mathcal{H}}(X)$ equals the Crane-Yetter invariant $CY_{\mathcal{C}}(X)$ up to a factor depending on $\op{dim}(H)$ and the Euler characteristics of $X$.
\end{theorem}
\noindent We show more generally that for certain triplets, our invariants recover some cases of dichromatic invariants.

\begin{remark}[Kashaev Invariant] For each positive integer $N$, the group algebra $\C[\Z/N]$ of the cyclic group $\Z/N$ has a quasi-triangular structure. The corresponding trisection invariant agrees with the Kashaev invariant associated with $\Z/N$ for some examples of 4-manifolds (up to Euler characteristics). 

Combined with Theorem \ref{thm:trisection_vs_crane_yetter_informal}, this provides supporting evidence for the conjecture that the Kashaev invariant is a special case of the Crane-Yetter invariant, which was first proposed in \cite{williamson2016hamiltonian} when studying Hamiltonian models of the two theories. \end{remark}

In sequel projects, we plan to generalize the trisection invariant in several directions, e.g.: 1) to 4-manifolds endowed with a Spin$^c$ structure,  2) to 4-manifolds with an embedded closed surface (a 2-knot for instance), and 3) to Hopf triplets in a general symmetric fusion category.  We also aim to produce invariants from non-semisimple Hopf triplets. 

\vspace{10pt}

\noindent \textbf{Organization.} The rest of the paper is organized as follows.
\begin{itemize}
\item In \textbf{\S\ref{sec:preliminaries}}, we review tensor algebras and Hopf algebras, emphasizing a diagrammatic point of view.  We define Hopf triplets, and derive their essential structural properties.  Then we review trisections of 4-manifolds, and their corresponding diagrammatics.
\item In \textbf{\S\ref{sec:trisection_kuperberg_invt}}, we define the input data to our invariant, culminating in the definition of the trisection invariant itself.  We prove that it is a diffeomorphism  invariant of smooth closed 4-manifolds, along with other structural properties.  The section concludes with a generalized formulation of the involutory Kuperberg invariant for 3-manifolds analogous to that of the $4$-manifold invariant.
\item In \textbf{\S\ref{sec:examples_and_calculations}}, we give examples of trisection diagrams and their corresponding tensor diagrams which evaluate to the trisection invariant.  We detail computational methods we devised to evaluate trisection invariants, and provide examples.  Two examples of particular interest are cyclic triplets which we conjecture give rise to Kashaev's invariants, and triplets of 8-dimensional Hopf algebras.  Some of the latter appear inequivalent to any known 4-manifold invariant.
\item In \textbf{\S\ref{sec:CY_dichro}}, we prove a relationship between special cases of the trisection invariant, and both the Crane-Yetter and generalized dichromatic invariants.
\end{itemize}

\vspace{10pt}

\noindent \textbf{Acknowledgements.} We would like to thank Peter Lambert-Cole for valuable discussion, and the anonymous referee for comments leading to Proposition \ref{prop:generalized_Kup_is_Kup}.  JCh was supported by the NSF Graduate Research Fellowship under Grant No.~1752814.  JCo is supported by the Fannie and John Hertz Foundation and the Stanford Graduate Fellowship program.  SCui is partially supported  by the startup fund from Purdue University. SCui also acknowledges the support from the Simons Foundation and Virginia Tech.

\section{Preliminaries} \label{sec:preliminaries} In this section, we review the background required to construct and compute our invariant. In \S \ref{subsec:tensor_algebra}, we discuss the algebraic objects involved in the construction, namely Hopf algebras and certain assemblages of Hopf algebras called Hopf triplets. In \S \ref{subsec:trisections}, we discuss trisections of 4-manifolds and trisection diagrams.

\subsection{Tensor Algebra} \label{subsec:tensor_algebra} Here we review the algebraic preliminaries for our invariant. We begin by discussing tensor diagrams, which provide a diagrammatic notation for tensor calculations (\S \ref{subsubsec:tensor_diagrams}). We then review the basic theory of Hopf algebras (\S \ref{subsubsec:hopf_algebra}) using the diagrammatic notation. Finally, we introduce the fundamental notion of a Hopf triplet (\S \ref{subsubsec:doublets_and_triplets}), which will serve as the algebraic input for our invariant.

\subsubsection{Tensor Diagrams} \label{subsubsec:tensor_diagrams} Let us begin by reviewing tensor diagram notation. This will be our main tool for defining all tensorial quantities in the present section \S \ref{subsec:tensor_algebra}, for defining our invariant and proving its properties in \S \ref{sec:trisection_kuperberg_invt}, and for performing calculations in \S \ref{sec:examples_and_calculations}.

\begin{notation} \label{not:general_tensor_diagrams} (Tensor Diagrams) Fix a collection $\mc{V}$ of vector spaces over a field $k$ of characteristic zero and let $\mc{F}$ be a collection of linear maps
\begin{equation} \label{eqn:linear_map_f} f:U_1 \otimes \cdots \otimes U_a \to V_1 \otimes \cdots \otimes V_b \qquad \text{where}\qquad U_i,V_j \in \mc{V} \end{equation}
Any linear map $g$ that can be obtained from the maps in $\mc{F}$ by a sequence of compositions, tensor products and traces can be represented as a \emph{tensor diagram}, i.e.~a decorated, directed graph immersed in the plane, in the following manner.

\begin{itemize}
\item[(a)] (Base Maps) Any linear map $f \in \mathcal{F}$ as in (\ref{eqn:linear_map_f}) is denoted by a node $f$ with $a$ incoming edges and $b$ outgoing edges ordered cyclically as follows.
\[
\begin{tikzpicture}
  %\node at (-1.8,0) (lbrack) {$a$ inputs $\Big\{$};
  %\node at (3.9,0) (rbrack) {$\Big\}$ $b$ outputs};
  \node at (1,0) [draw,rectangle,minimum size=30pt] (f) {$f$};

  \node at (-.8,.5) {\scriptsize $1$};
  \node at (-.8,.2) {\scriptsize $2$};
  \node at (-.6,-.2) {\scriptsize $(a-1)$};
  \node at (-.8,-.5) {\scriptsize $a$};

  \node at (2.8,.5) {\scriptsize $1$};
  \node at (2.8,.2) {\scriptsize $2$};
  \node at (2.6,-.2) {\scriptsize $(b-1)$};
  \node at (2.8,-.5) {\scriptsize $b$};

  \draw[->] (-.6,.5)--(.4,.5);
  \draw[->] (-.6,.2)--(.4,.2);
  \draw (-.2,0) node (dts1) {$\cdots$};
  \draw[->] (-.1,-.2)--(.4,-.2);
  \draw[->] (-.6,-.5)--(.4,-.5);

  \draw[->] (1.6,.5)--(2.6,.5);
  \draw[->] (1.6,.2)--(2.6,.2);
  \draw (2,0) node (dts2) {$\cdots$};
  \draw[->] (1.6,-.2)--(2.1,-.2);
  \draw[->] (1.6,-.5)--(2.6,-.5);
  \end{tikzpicture}
\]
The edges $i \to$ and $\to j$ correspond to the $U_i$ and $V_j$ factors, respectively. We strictly adhere to this edge ordering, and typically omit the edge labels $i,j$.

\item[(b)] (Tensor Products) Fix two tensors $g$ and $g'$. The tensor product $g \otimes g'$ is denoted by taking the disjoint union of the corresponding graphs for $g$ and $g'$.
\[
\begin{tikzpicture}
  \node at (1,0) [draw,rectangle,minimum size=30pt] (gg') {$g \otimes g'$};
  \node at (3.8,0)  (=) {$=$};
  \node at (6,.4) [draw,rectangle,minimum size=20pt] (g) {$g$};
  \node at (6,-.4) [draw,rectangle,minimum size=20pt] (g') {$g'$};

  \draw[->] (-.9,.4)--(.2,.4);
  \draw[->] (-.9,.2)--(.2,.2);

  \draw[->] (-.9,-.2)--(.2,-.2);
  \draw[->] (-.9,-.4)--(.2,-.4);

  \draw[->] (1.8,.4)--(2.9,.4);
  \draw[->] (1.8,.2)--(2.9,.2);

  \draw[->] (1.8,-.2)--(2.9,-.2);
  \draw[->] (1.8,-.4)--(2.9,-.4);

  \draw[->] (4.8,.5)--(5.6,.5);
  \draw[->] (4.8,.3)--(5.6,.3);

  \draw[->] (4.8,-.3)--(5.6,-.3);
  \draw[->] (4.8,-.5)--(5.6,-.5);

  \draw[->] (6.4,.5)--(7.2,.5);
  \draw[->] (6.4,.3)--(7.2,.3);

  \draw[->] (6.4,-.3)--(7.2,-.3);
  \draw[->] (6.4,-.5)--(7.2,-.5);
  \end{tikzpicture}
\]

\item[(c)] (Trace) Fix a tensor $g$, an incoming edge $c$ and outgoing edge $d$ corresponding to the same $V \in \mathcal{V}$. The trace $\op{Tr}_c^d(g)$ is denoted by connecting $c$ and $d$.
\[\hspace{30pt}\begin{tikzpicture}
  \node at (.8,0) [draw,rectangle,minimum size=30pt] (Trg) {$\op{Tr}_c^d(g)$};

  \draw[->] (-1.1,.1)--(0,.1);
  \draw[->] (-1.1,-.1)--(0,-.1);
  \draw[->] (-1.1,-.3)--(0,-.3);
  \draw[->] (-1.1,-.5)--(0,-.5);

  \draw[->] (1.6,.1)--(2.7,.1);
  \draw[->] (1.6,-.1)--(2.7,-.1);
  \draw[->] (1.6,-.3)--(2.7,-.3);
  \draw[->] (1.6,-.5)--(2.7,-.5);

  \node at (3.4,-.2) (=) {$=$};

  \node at (6,0) [draw,rectangle,minimum size=30pt] (g) {$g$};

  \draw[->] (6.7,.3) to [out=30,in=150,looseness=2] (5.3,.3);
  \draw[->] (4.2,.1)--(5.3,.1);
  \draw[->] (4.2,-.1)--(5.3,-.1);
  \draw[->] (4.2,-.3)--(5.3,-.3);
  \draw[->] (4.2,-.5)--(5.3,-.5);

  \draw[->] (6.7,.1)--(7.8,.1);
  \draw[->] (6.7,-.1)--(7.8,-.1);
  \draw[->] (6.7,-.3)--(7.8,-.3);
  \draw[->] (6.7,-.5)--(7.8,-.5);
\end{tikzpicture}\]
\end{itemize}

\noindent Note that the composition of linear maps $g$ and $h$ can be written as a trace of $g \otimes h$.\end{notation}

\begin{remark}[Ordering Ambiguity] We typically omit the box around the node $f \in \mathcal{F}$ in any setting where the cyclic order of the edges of $f$ unambiguously determines the correspondence between tensor factors and edges. 

\vspace{3pt}

The exception is when $f$ is a linear map with only input edges or only output edges (corresponding to a single vector space) that is \emph{not} invariant under cyclic permutation of the input/output tensor factors. In this case, we will \emph{always} draw a box about the node of $f$. If $f$ has only inputs, we draw the edges entering the left side of the box ordered from top to bottom. Likewise, if $f$ has only output edges, we draw the edges exiting the right side of the box ordered from top to bottom. \end{remark}

\subsubsection{Hopf Algebra} \label{subsubsec:hopf_algebra} A Hopf algebra is a bialgebra (i.e.,~a simultaneous algebra and coalgebra where the two structures interact nicely) equipped with a canonical antipode. More precisely, we have the following.

\begin{definition}[Hopf Algebra] \label{def:hopf_algebra} A \emph{Hopf algebra} $H = (H,M,\eta,\Delta,\epsilon,S)$ over a ring $k$ is a module $H$ over $k$ equipped with structure tensors of the form
\[\begin{tikzpicture}
  \draw (-1.5,0) node (label_product) {(Product)};
  \draw (1,0) node (M) {$M$};
  \draw[->] (0,-.25)--(M);
  \draw[->] (0,.25)--(M);
  \draw[->] (M)--(2,0);

  \draw (5.5,0) node (label_unit) {(Unit)};
  \draw (7.5,0) node (N) {$\eta$};
  \draw[->] (N)--(6.5,0);
\end{tikzpicture}\]
\[\begin{tikzpicture}
  \draw (-1.5,0) node (label_coproduct) {(Coproduct)};
  \draw (1,0) node (D) {$\Delta$};
  \draw[->] (D)--(2,-.25);
  \draw[->] (D)--(2,.25);
  \draw[->] (0,0)--(D);

  \draw (5.5,0) node (label_counit) {(Counit)};
  \draw (7.5,0) node (E) {$\epsilon$};
  \draw[->] (6.5,0)--(E);
  \end{tikzpicture}\]
\[\begin{tikzpicture}
  \draw (-1.5,0) node (label_antipode) {(Antipode)};
  \draw (1,0) node (S) {$S$};
  \draw[->] (0,0)--(S);
  \draw[->] (S)--(2,0);
\end{tikzpicture}\]
These structure tensors must satisfy a series of compatibility properties, which we now specify using tensor diagram notation.
\begin{itemize} 

\item[(a)] (Algebra) $(H,M,\eta)$ must define a unital $k$-algebra. In tensor diagrams, we have the following identities.

\[
\hspace{-20pt}\begin{tikzpicture}
  \draw (0,.25) node (M1) {$M$};
  \draw (1,0) node (M2) {$M$};
  \draw (2.4,-1) node (label_assoc) {(Associativity)};

  \draw[->] (-1,.5)--(M1);
  \draw[->] (-1,0)--(M1);
  \draw[->] (M1)--(M2);
  \draw[->] (0,-.25)--(M2); 
  \draw[->] (M2)--(1.9,0);

  \draw (2.4,0) node (=) {$=$};

  \draw (3.9,-.25) node (M1) {$M$};
  \draw (4.9,0) node (M2) {$M$};

  \draw[->] (3.1,-.5)--(M1);
  \draw[->] (3.1,-0)--(M1);
  \draw[->] (M1)--(M2);
  \draw[->] (3.9,.25)--(M2); 
  \draw[->] (M2)--(5.8,0);

  \draw (7.4,.25) node (n) {$\eta$};
  \draw (8.4,0) node (M) {$M$};
  \draw[->] (7.4,-.25)--(M);
  \draw[->] (n)--(M);
  \draw[->] (M)--(9.3,0);

  \draw (9.7,0) node (=) {$=$};

  \draw (10.1,-.25) node (n) {$\eta$};
  \draw (11.1,0) node (M) {$M$};
  \draw[->] (10.1,.25)--(M);
  \draw[->] (n)--(M);
  \draw[->] (M)--(12,0);

  \draw (12.5,0) node (=) {$=$};

  \draw[->] (12.9,0)--(13.5,0);

  \draw (10.4,-1) node (label_unit) {(Unitality)};
\end{tikzpicture}
\]
\item[(b)] (Coalgebra) $(H,\Delta,\epsilon)$ must define a counital $k$-coalgebra. In tensor diagrams, we have the following identities.
\[
\hspace{-20pt}\begin{tikzpicture}
  \draw (0,.25) node (D1) {$\Delta$};
  \draw (1,0) node (D2) {$\Delta$};
  \draw (2.4,-1) node (label_assoc) {(Co-associativity)};

  \draw[->] (D1)--(-1,.5);
  \draw[->] (D1)--(-1,0);
  \draw[->] (D2)--(D1);
  \draw[->] (D2)--(0,-.25); 
  \draw[->] (1.9,0)--(D2);

  \draw (2.4,0) node (=) {$=$};

  \draw (3.9,-.25) node (D1) {$\Delta$};
  \draw (4.9,0) node (D2) {$\Delta$};

  \draw[->] (D1)--(3.1,-.5);
  \draw[->] (D1)--(3.1,-0);
  \draw[->] (D2)--(D1);
  \draw[->] (D2)--(3.9,.25); 
  \draw[->] (5.8,0)--(D2);

  \draw (7.4,.25) node (e) {$\epsilon$};
  \draw (8.4,0) node (D) {$\Delta$};
  \draw[->] (D)--(7.4,-.25);
  \draw[->] (D)--(e);
  \draw[->] (9.3,0)--(D);

  \draw (9.7,0) node (=) {$=$};

  \draw (10.1,-.25) node (e) {$\epsilon$};
  \draw (11.1,0) node (D) {$\Delta$};
  \draw[->] (D)--(10.1,.25);
  \draw[->] (D)--(e);
  \draw[->] (12,0)--(D);

  \draw (12.5,0) node (=) {$=$};

  \draw[->] (13.5,0)--(12.9,0);

  \draw (10.4,-1) node (label_unit) {(Co-unitality)};
\end{tikzpicture}
\]
\item[(c)] (Bialgebra/Antipode) The coalgebra and algebra structures must be compatible, in the sense that they define a bialgebra, and the antipode must satisfy a standard antipode identity involving the product and coproduct.
\[
\hspace{-20pt}\begin{tikzpicture}
\draw (0,0) node (D1) {$\Delta$};
\draw (0,1) node (D2) {$\Delta$};
\draw (1,0) node (M1) {$M$};
\draw (1,1) node (M2) {$M$};

 \draw[->] (-.6,0)--(D1);
 \draw[->] (-.6,1)--(D2);
 \draw[->] (M1)--(1.6,0);
 \draw[->] (M2)--(1.6,1);

 \draw[->] (D1)--(M1);
 \draw[->] (D1)--(M2);
 \draw[->] (D2)--(M1);
 \draw[->] (D2)--(M2);

 \draw (2.3,.35) node (=) {$=$};

 \draw (3.4,.5) node (M) {$M$};
\draw (4.4,.5) node (D) {$\Delta$};

 \draw[->] (2.8,0)--(M);
 \draw[->] (2.8,1)--(M);
 \draw[->] (D)--(5,0);
 \draw[->] (D)--(5,1);

 \draw[->] (M)--(D);

 \draw (2.8,-1) node (label_assoc) {(Bi-Algebra)};

\draw (7,0) node (D1) {$\Delta$};
\draw (7.5,.8) node (S1) {$S$};
\draw (8,0) node (M1) {$M$};

 \draw[->] (6.4,0)--(D1);
 \draw[->] (D1)--(M1);
 \draw[->] (D1)--(S1);
 \draw[->] (S1)--(M1);
 \draw[->] (M1)--(8.7,0);

 \draw (9,.35) node (=) {$=$};

\draw (10,.8) node (D2) {$\Delta$};
\draw (10.5,0) node (S2) {$S$};
\draw (11,.8) node (M2) {$M$};

 \draw[->] (9.4,.8)--(D2);
 \draw[->] (D2)--(M2);
 \draw[->] (D2)--(S2);
 \draw[->] (S2)--(M2);
 \draw[->] (M2)--(11.7,.8);

 \draw (12,.35) node (=) {$=$};

\draw (13,.4) node (e) {$\epsilon$};
\draw (13.3,.4) node (n) {$\eta$};

 \draw[->] (12.4,.4)--(e);
 \draw[->] (n)--(13.8,.4);

 \draw (10.8,-1) node (label_unit) {(Antipode)};

 \end{tikzpicture}
\]
\end{itemize}
A map $f:H \to I$ of Hopf algebras is a linear map intertwining the product, unit, coproduct, counit and antipode. The tensor diagram identities for $f$ are clear. \end{definition}

In this paper, we will restrict to the following special class of Hopf algebras.

\begin{definition}[Involutory] A Hopf algebra $H$ is \emph{involutory} if the antipode squares to the identity.
\begin{equation} \label{eqn:involutory_condition}
\to S \to S \to \quad = \quad \to 
\end{equation}
\end{definition}

\begin{remark} In the case where $H$ is a Hopf algebra over a field $k$ of characteristic $0$, $H$ is involutory if and only if $H$ is semisimple by a theorem of Larson and Radford (see \cite{lr1987}).
\end{remark}

In addition to the above structure maps, the following maps arise frequently in the study of Hopf algebras, and in this paper.

\begin{definition}[(Co)Traces/(Co)Integrals] \label{def:trace_cotrace} Let $(H,M,\eta,\Delta,\epsilon)$ be an involutory Hopf algebra. 
\begin{itemize}
	\item[(a)] The \emph{trace} $T:H \to k$ and \emph{cotrace} $C:k \to H$ of $H$ are defined by
\[
\begin{tikzpicture}
  \draw (9,0) node (C) {$C$};
  \draw[->] (C)--(10,0);
  \draw (10.5,0) node (=1) {$:=$};
 \draw (12,0) node (D) {$\Delta$};
  \draw (D) edge [loop left] node {} (D);
  \draw[->] (D)--(13,0);

  \draw (2,0) node (T) {$T$};
  \draw[->] (1,0)--(T);
   \draw (3,0) node (=1) {$:=$};
  \draw (4.5,0) node (M) {$M$};
  \draw (M) edge [loop right] node {} (M);
  \draw[->] (3.5,0)--(M);
\end{tikzpicture}
\]
	\item[(b)] An \emph{integral} $\mu:H \to k$ and a \emph{cointegral} $e: k \to H$ of $H$ are maps such that
  \[\begin{tikzpicture}
  \draw (1,0) node (D1) {$\Delta$};
  \draw (2,-.25) node (i1) {$\mu$};
  \draw[->] (D1)--(i1);
  \draw[->] (D1)--(2,.25);
  \draw[->] (0,0)--(D1);

  \draw (2.5,0) node (=1) {$=$};

  \draw (4,0) node (D2) {$\Delta$};
  \draw (5,.25) node (i2) {$\mu$};
  \draw[->] (D2)--(5,-.25);
  \draw[->] (D2)--(i2);
  \draw[->] (3,0)--(D2);

  \draw (5.5,0) node (=2) {$=$};

  \draw (6.5,0) node (i3) {$\mu$};
  \draw (7,0) node (N) {$\eta$};
  \draw[->] (6,0)--(i3);
  \draw[->] (N)--(7.5,0);
\end{tikzpicture}
\]
 \[
\begin{tikzpicture}
  \draw (1,0) node (M1) {$M$};
  \draw (2,-.25) node (m1) {$e$};
  \draw[<-] (M1)--(m1);
  \draw[<-] (M1)--(2,.25);
  \draw[<-] (0,0)--(M1);

  \draw (2.5,0) node (=1) {$=$};

  \draw (4,0) node (M2) {$M$};
  \draw (5,.25) node (m2) {$e$};
  \draw[<-] (M2)--(5,-.25);
  \draw[<-] (M2)--(m2);
  \draw[<-] (3,0)--(M2);

  \draw (5.5,0) node (=2) {$=$};

  \draw (6.5,0) node (m3) {$e$};
  \draw (7,0) node (E) {$\epsilon$};
  \draw[<-] (6,0)--(m3);
  \draw[<-] (E)--(7.5,0);
\end{tikzpicture}\]
	We remark that there are notions of left and right (co)integrals for non-involutory Hopf algebras, which will not appear in this paper. 
\end{itemize}
\end{definition}

It is a basic fact that integrals/co-integrals for Hopf algebras over a field of characteristic zero are unique up to scalar multiplication. It can also be checked directly that the antipode fixes both $T$ and $C$, namely, $T \circ S = T, \ S\circ C = C$. 
\begin{lemma}[\cite{kuperberg1996noninvolutory}] \label{lem:trace_is_integral} Let $H$ be an involutory Hopf algebra. Then the trace $T$ and cotrace $C$ are, respectively, an integral and a cointegral.
\end{lemma}

\begin{notation} \label{not:product_coproduct_abbreviated_notation} The associativity and coassociativity axioms permit us to adopt the following abbreviated notation for iterated products and coproducts.
\[
\begin{tikzpicture}
  \draw (0,0) node (M1) {$M$};
  \draw (-.6,-.08) node (dots1) {$\dots$};

 \draw[->] (-.8,.5)--(M1);
 \draw[->] (-.8,.2)--(M1);

 \draw[->] (-.8,-.5)--(M1);

 \draw[->] (M1)--(.8,0);

  \draw (1.2,0) node (=1) {$:=$};

  \draw (2.6,0) node  [draw,rectangle] (TM1) {$TM$};
  \draw (1.8,-.08) node (dots2) {$\dots$};

 \draw[-] (1.6,.5)--(TM1);
 \draw[-] (1.6,.2)--(TM1);
 \draw[-] (1.6,-.5)--(TM1);

 \draw[->] (TM1)--(3.4,0);

 \draw (5,0) node (M2) {$M$};
 
 \draw[->] (4.2,0)--(M2);
 \draw[->] (M2)--(5.8,0);

 \draw (6.2,0) node (=2) {$:=$};

 \draw[->] (6.6,0)--(7.4,0);

 \draw (8.4,0) node (M3) {$M$};
 \draw (10.1,0) node (n3) {$\eta$};
 
 \draw[->] (M3)--(9.2,0);

 \draw (9.6,0) node (=3) {$:=$};

 \draw[->] (n3)--(10.9,0);
 \end{tikzpicture}
\]
\[
\begin{tikzpicture}
  \draw (0,0) node (D1) {$\Delta$};
  \draw (-.6,-.08) node (dots1) {$\dots$};

 \draw[<-] (-.8,.5)--(D1);
 \draw[<-] (-.8,.2)--(D1);

 \draw[<-] (-.8,-.5)--(D1);

 \draw[<-] (D1)--(.8,0);

  \draw (1.2,0) node (=1) {$:=$};

  \draw (2.6,0) node  [draw,rectangle] (TD1) {$T\Delta$};
  \draw (1.8,-.08) node (dots2) {$\dots$};

 \draw[<-] (1.6,.5)--(TD1);
 \draw[<-] (1.6,.2)--(TD1);
 \draw[<-] (1.6,-.5)--(TD1);

 \draw[-] (TD1)--(3.4,0);

 \draw (5,0) node (D2) {$\Delta$};
 
 \draw[<-] (4.2,0)--(D2);
 \draw[<-] (D2)--(5.8,0);

 \draw (6.2,0) node (=2) {$:=$};

 \draw[<-] (6.6,0)--(7.4,0);

 \draw (8.4,0) node (D3) {$\Delta$};
 \draw (10.1,0) node (e3) {$\epsilon$};
 
 \draw[<-] (D3)--(9.2,0);

 \draw (9.6,0) node (=3) {$:=$};

 \draw[<-] (e3)--(10.9,0);
 \end{tikzpicture}
\]
\[
\begin{tikzpicture}
  \draw (0,0) node (M1) {$M$};
  \draw (-.6,0) node (dots1) {$\dots$};
  \draw (.65,0) node (dots2) {$\dots$};

 \draw[->] (-.6,.6)--(M1);
 \draw[->] (0,.8)--(M1);
 \draw[->] (.6,.6)--(M1);
 \draw[->] (.6,-.6)--(M1);
 \draw[->] (0,-.8)--(M1);
 \draw[->] (-.6,-.6)--(M1);

  \draw (1.6,0) node (=1) {$:=$};

 \draw (3,0) node (M2) {$M$};
 \draw (2.4,-.08) node (dots3) {$\dots$};
 \draw (4,0) node (T2) {$T$};

 \draw[->] (2.2,.5)--(M2);
 \draw[->] (2.2,.2)--(M2);

 \draw[->] (2.2,-.5)--(M2);

 \draw[->] (M2)--(T2);

 \draw (6,0) node (D1) {$\Delta$};
 \draw (5.4,0) node (dots4) {$\dots$};
 \draw (6.65,0) node (dots5) {$\dots$};

 \draw[<-] (5.4,.6)--(D1);
 \draw[<-] (6,.8)--(D1);
 \draw[<-] (6.6,.6)--(D1);
 \draw[<-] (6.6,-.6)--(D1);
 \draw[<-] (6,-.8)--(D1);
 \draw[<-] (5.4,-.6)--(D1);

  \draw (7.6,0) node (=2) {$:=$};

 \draw (9,0) node (D2) {$\Delta$};
 \draw (8.4,-.08) node (dots6) {$\dots$};
 \draw (10,0) node (C2) {$C$};

 \draw[<-] (8.2,.5)--(D2);
 \draw[<-] (8.2,.2)--(D2);

 \draw[<-] (8.2,-.5)--(D2);

 \draw[<-] (D2)--(C2);
 \end{tikzpicture}
\]
Here $TM$ denotes an arbitrary tree with $i$ in edges, $1$ out edge and only $M$ nodes, and similarly $T\Delta$ denotes an arbitrary tree with $i$ out edges, $1$ in edge and only $\Delta$ nodes. The multiplication-trace and comultiplication-cotrace compositions are symmetric under cyclic permutation, as suggested by the notation.
\end{notation}

Any fixed Hopf algebra gives rise to a number of associated Hopf algebras acquirable by simple alterations of the structure tensors.

\begin{definition} \label{def:dual_op_cop_algebras} Let $H$ be a finite-dimensional Hopf algebra. Then
\begin{itemize}

  \item[(a)] (Dual) The \emph{dual Hopf algebra} $H^*$ is the linear dual $H^*$ equipped multiplication $\Delta$, comultiplication $M$ and antipode $S$, interpreted as tensors for the dual space $H^*$. Written left to right, these are
  \[\begin{tikzpicture}

  \draw (1,0) node (D) {$\Delta$};
  \draw[<-] (0,-.25) to [out=0,in=150,looseness=1] (D);
  \draw[<-] (0,.25) to [out=0,in=210,looseness=1] (D);
  \draw[<-] (D)--(2,0); 

  \draw (4,0) node (M) {$M$};
  \draw[<-] (M) to [out=30,in=180,looseness=1] (5,-.25);
  \draw[<-] (M) to [out=-30,in=180,looseness=1] (5,.25);
  \draw[->] (M)--(3,0); 

  \draw (7,0) node (S) {$S$};
  \draw[<-] (6,0)--(S);
  \draw[<-] (S)--(8,0); 
  \end{tikzpicture}\]  
Note the change in the order of the inputs and outputs in the multiplication and comultiplication tensors. This is necessary due to our input and output ordering convention (see Notation \ref{not:general_tensor_diagrams}).

  \item[(b)] (Op) The \emph{op Hopf algebra} $H^{\op{op}}$ is $H$ equipped with coproduct $\Delta^{\op{op}} := \Delta$, antipode $S^{\op{op}} := S$ and product $M^{\op{op}}$ given by the tensor
  \[\begin{tikzpicture}
  \draw (1,0) node (Mop) {$M^{\op{op}}$};
  \draw[->] (0,-.25)--(Mop);
  \draw[->] (0,.25)--(Mop);
  \draw[->] (Mop)--(2,0); 

  \draw (2.5,0) node (=) {$:=$};

  \draw (4,0) node (M) {$M$};
  \draw[->] (3,-.25) to [out=0,in=150,looseness=1] (M);
  \draw[->] (3,.25) to [out=0,in=210,looseness=1] (M);
  \draw[->] (M)--(5,0);  \end{tikzpicture}\]

  \item[(c)] (Cop) The \emph{cop Hopf algebra} $H^{\op{cop}}$ is $H$ equipped with product $M^{\op{cop}} := M$, antipode $S^{\op{op}} := S$ and coproduct $\Delta^{\op{cop}}$ given by the tensor 
  \[\begin{tikzpicture}
  \draw (1,0) node (Dcop) {$\Delta^{\op{cop}}$};
  \draw[<-] (0,-.25)--(Dcop);
  \draw[<-] (0,.25)--(Dcop);
  \draw[<-] (Dcop)--(2,0); 

  \draw (2.5,0) node (=) {$:=$};

  \draw (4,0) node (D) {$\Delta$};
  \draw[<-] (3,-.25) to [out=0,in=150,looseness=1] (D);
  \draw[<-] (3,.25) to [out=0,in=210,looseness=1] (D);
  \draw[<-] (D)--(5,0);  \end{tikzpicture}\]
\end{itemize}
All of these constructions are functorial, i.e.~a Hopf algebra map $f:H \to I$ induces maps $f^*:I^* \to H^*$, $f^{\op{op}}:H^{\op{op}} \to I^{\op{op}}$ and $f^{\op{cop}}:H^{\op{cop}} \to I^{\op{cop}}$.
\end{definition}

\subsubsection{Hopf Doublets And Triplets} \label{subsubsec:doublets_and_triplets} We are now ready to introduce the Hopf algebra data that are used to formulate $3$-manifold and $4$-manifold invariants. 

The data for $3$-manifold invariants (used for the Kuperberg invariant and its generalization at the end of \S\ref{sec:trisection_kuperberg_invt}) can be formulated in terms of Hopf doublets, which we now define. 

\begin{definition}[Hopf Doublet] \label{def:hopf_doublet} A (involutory) \emph{Hopf doublet} $(H,\langle-\rangle)$ consists of two (involutory) Hopf algebras $H_\alpha$ and $H_\beta$, and a bilinear form
\[
\langle-\rangle:H_\alpha \otimes H_\beta \to k
\]
The form $\langle-\rangle$ must satisfy the following properties.
\begin{itemize}
  \item[(a)] The linear map $H_\alpha \to H_\beta^{*,\op{cop}}$ induced by $\langle-\rangle$ must be a Hopf algebra map.
\end{itemize}
A map of Hopf doublets $f:(H,\langle-\rangle) \to (I,(-))$ is a set of maps of Hopf algebras $f_*:H_* \to I_*$ for $* \in \{\alpha,\beta\}$ intertwining the bilinear forms $\langle-\rangle$ and $(-)$.
\end{definition}

There are several ways of making new Hopf doublets from a single Hopf doublet by applying the various operations of Definition \ref{def:dual_op_cop_algebras}. In particular, we have the following lemma.

\begin{lemma} \label{lem:reverse_doublet} Let $(H_\alpha,H_\beta,\langle-\rangle)$ be a Hopf doublet. Then
\begin{itemize}
\item[(a)] $(H^{\op{op}}_\alpha,H^{\op{cop}}_\beta,\langle-\rangle)$ is a Hopf doublet.
\item[(b)] $(H^{\op{cop}}_\alpha,H^{\op{op}}_\beta,\langle-\rangle)$ is a Hopf doublet.
\item[(c)] $(H^{\op{op},\op{cop}}_\beta,H_\alpha,\langle-\rangle)$ is a Hopf doublet.
\end{itemize}
\end{lemma}

It will be convenient, for later constructions, to introduce the following shorthand notation involving the tensors in a doublet.

\begin{notation}\label{not:TUV_notation} Let $(H_\alpha,H_\beta,\langle-\rangle)$ be a Hopf doublet. We fix the following notation
\[
\begin{tikzpicture}
\draw (1,0) node (Tab) {$T_{\alpha\beta}$};
\draw[->] (0,.25)--(Tab);
\draw[->] (0,-.25)--(Tab);
\draw[->] (Tab)--(2,.25);
\draw[->] (Tab)--(2,-.25);

\draw (3,0) node (=1) {$:=$};

\draw (5,1) node (Db) {$\Delta_\beta$};
\draw (5,0) node (Pab) {$\langle-\rangle_{\alpha\beta}$};
\draw (5,-1) node (Da) {$\Delta_\alpha$};

\draw[->] (4,1)--(Db);
\draw[->] (Db)--(6,1);
\draw[->] (4,-1)--(Da);
\draw[->] (Da)--(6,-1);
\draw[->] (Db)--(Pab);
\draw[->] (Da)--(Pab);

\draw (9,0) node (Tab_inv) {$T^{-1}_{\alpha\beta}$};
\draw[->] (8,.25)--(Tab_inv);
\draw[->] (8,-.25)--(Tab_inv);
\draw[->] (Tab_inv)--(10,.25);
\draw[->] (Tab_inv)--(10,-.25);

\draw (11,0) node (=2) {$=$};

\draw (13,1) node (Db2) {$\Delta_\beta$};
\draw (13.6,0) node (Sa2) {$S_\alpha$};
\draw (12.4,0) node (Pab2) {$\langle-\rangle_{\alpha\beta}$};
\draw (13,-1) node (Da2) {$\Delta_\alpha$};

\draw[->] (12,1)--(Db2);
\draw[->] (Db2)--(14,1);
\draw[->] (12,-1)--(Da2);
\draw[->] (Da2)--(14,-1);
\draw[->] (Db2)--(Pab2);
\draw[->] (Sa2)--(Pab2);
\draw[->] (Da2)--(Sa2);
\end{tikzpicture}
\]
\vspace{5pt}
\[
\begin{tikzpicture}

\draw (-2,0) node (Uab) {$U_{\alpha\beta}$};
\draw[->] (-3,.25)--(Uab);
\draw[->] (-3,-.25)--(Uab);
\draw[->] (Uab)--(-1,.25);
\draw[->] (Uab)--(-1,-.25);

\draw (-.4,0) node (=1) {$:=$};

\draw (0,1) node (in_b) {};
\draw (0,-1) node (in_a) {};
\draw (1,1) node (Sb1) {$S_\beta$};

\draw (2,1) node (Db2) {$\Delta_\beta$};
\draw (2.6,0) node (Sa2) {$S_\alpha$};
\draw (1.4,0) node (Pab2) {$\langle-\rangle_{\alpha\beta}$};
\draw (2,-1) node (Da2) {$\Delta_\alpha$};

\draw (3,1) node (Sb) {$S_\beta$};
\draw (3,-1) node (Sa) {$S_\alpha$};

\draw (4,1) node (Db) {$\Delta_\beta$};
\draw (4,0) node (Pab) {$\langle-\rangle_{\alpha\beta}$};
\draw (4,-1) node (Da) {$\Delta_\alpha$};

\draw (5,-1) node (Sa3) {$S_\alpha$};

\draw (6,1) node (out_b) {};
\draw (6,-1) node (out_a) {};

\draw[->] (in_b)--(Sb1);
\draw[->] (Sb1)--(Db2);
\draw[->] (in_a)--(Da2);

\draw[->] (Db2)--(Pab2);
\draw[->] (Sa2)--(Pab2);
\draw[->] (Da2)--(Sa2);

\draw[->] (Db2)--(Sb);
\draw[->] (Sb)--(Db);
\draw[->] (Da2)--(Sa);
\draw[->] (Sa)--(Da);

\draw[->] (Db)--(Pab);
\draw[->] (Da)--(Pab);

\draw[->] (Db)--(out_b);
\draw[->] (Da)--(Sa3);
\draw[->] (Sa3)--(out_a);

\draw (7,0) node (=1) {$=$};

\draw (7.5,1) node (in_b4) {};
\draw (7.5,-1) node (in_a4) {};
\draw (8.5,1) node (Sb4) {$S_\beta$};
\draw (9,0) node (Tab_inv2) {$T^{-1}_{\alpha\beta}$};
\draw (9.5,1) node (Sb5) {$S_\beta$};
\draw (9.5,-1) node (Sa5) {$S_\alpha$};
\draw (10,0) node (Tab_2) {$T_{\alpha\beta}$};
\draw (10.5,-1) node (Sa6) {$S_\alpha$};
\draw (11.5,1) node (out_b4) {};
\draw (11.5,-1) node (out_a4) {};

\draw[->] (in_b4)--(Sb4);
\draw[->] (Sb4)--(Tab_inv2);
\draw[->] (in_a4)--(Tab_inv2);
\draw[->] (Tab_inv2)--(Sb5);
\draw[->] (Tab_inv2)--(Sa5);
\draw[->] (Sb5)--(Tab_2);
\draw[->] (Sa5)--(Tab_2);
\draw[->] (Tab_2)--(Sa6);
\draw[->] (Sa6)--(out_a4);
\draw[->] (Tab_2)--(out_b4);
\end{tikzpicture}
\]
\vspace{5pt}
\[
\begin{tikzpicture}
\draw (-4,.5) node (in_b0) {};
\draw (-4,-.5) node (in_a0) {};
\draw (-3,0) node (Vab) {$V_{\alpha\beta}$};
\draw (-2,.5) node (out_b0) {};
\draw (-2,-.5) node (out_a0) {};

\draw[->] (in_b0)--(Vab);
\draw[->] (in_a0)--(Vab);
\draw[->] (Vab)--(out_b0);
\draw[->] (Vab)--(out_a0);

\draw (-1,0) node (=1) {:=};

\draw (0,-.5) node (in_a1) {};
\draw (0,.5) node (in_b1) {};
\draw (1,-.5) node (Sa1) {$S_\alpha$};
\draw (1,.5) node (Sb1) {$S_\beta$};
\draw (2,0) node (Tab1) {$T_{\alpha\beta}$};
\draw (3,.5) node (Sb2) {$S_\beta$};
\draw (4,-.5) node (out_a1) {};
\draw (4,.5) node (out_b1) {};

\draw[->] (in_a1)--(Sa1);
\draw[->] (in_b1)--(Sb1);
\draw[->] (Sa1)--(Tab1);
\draw[->] (Sb1)--(Tab1);
\draw[->] (Tab1)--(Sb2);
\draw[->] (Tab1)--(out_a1);
\draw[->] (Sb2)--(out_b1);
\end{tikzpicture}
\]
\end{notation}
We will make use of the following two identities relating the above tensors.
\begin{lemma} \label{lem:UTV_identity} Let $(H_\alpha,H_\beta,\langle-\rangle_{\alpha\beta})$ be an involutory Hopf doublet. Then the tensors $T_{\alpha\beta}, U_{\alpha\beta}$ and $V_{\alpha\beta}$ satisfy the following relations.
\begin{equation} \label{eqn:TUV_identity_1}
\begin{tikzpicture}
\draw (0,-.5) node (in_a1) {};
\draw (0,.5) node (in_b1) {};
\draw (1,0) node (Vab) {$V_{\alpha\beta}$};
\draw (3,0) node (Uab) {$U_{\alpha\beta}$};
\draw (4,-.5) node (out_a1) {};
\draw (4,.5) node (out_b1) {};

\draw[->] (in_a1)--(Vab);
\draw[->] (in_b1)--(Vab);
\draw[->] (Vab) to [out=30,in=150,looseness=1] (Uab);
\draw[->] (Vab) to [out=-30,in=210,looseness=1] (Uab);
\draw[->] (Uab)--(out_a1);
\draw[->] (Uab)--(out_b1);

\draw (4.5,0) node (=1) {=};

\draw (5,-.5) node (in_a2) {};
\draw (5,.5) node (in_b2) {};
\draw (6,0) node (Tab) {$T_{\alpha\beta}$};
\draw (7,-.5) node (Sa) {$S_\alpha$};
\draw (8,-.5) node (out_a2) {};
\draw (8,.5) node (out_b2) {};

\draw[->] (in_a2)--(Tab);
\draw[->] (in_b2)--(Tab);
\draw[->] (Tab)--(out_b2);
\draw[->] (Tab)--(Sa);
\draw[->] (Sa)--(out_a2);
\end{tikzpicture}
\end{equation}
\begin{equation} \label{eqn:TUV_identity_2}
\begin{tikzpicture}
\draw (0,.5) node (in_b) {};
\draw (0,-.5) node (in_a) {};
\draw (1,0) node (Tab_1) {$T_{\alpha\beta}$};
\draw (2,.5) node (Sb) {$S_\beta$};
\draw (2,-.5) node (Sa) {$S_\alpha$};
\draw (3,0) node (Tab_2) {$T^{-1}_{\alpha\beta}$};
\draw (4,.5) node (Sb2) {$S_\beta$};
\draw (4,-.5) node (Sa2) {$S_\alpha$};
\draw (5,.5) node (out_b) {};
\draw (5,-.5) node (out_a) {};

\draw[->] (in_b)--(Tab_1);
\draw[->] (in_a)--(Tab_1);
\draw[->] (Tab_1)--(Sb);
\draw[->] (Tab_1)--(Sa);
\draw[->] (Sb)--(Tab_2);
\draw[->] (Sa)--(Tab_2);
\draw[->] (Tab_2)--(Sb2);
\draw[->] (Tab_2)--(Sa2);
\draw[->] (Sb2)--(out_b);
\draw[->] (Sa2)--(out_a);
\end{tikzpicture}
\begin{tikzpicture}

\draw (.5,0) node (=1) {$=$};
\draw (1,.5) node (in_b) {};
\draw (1,-.5) node (in_a) {};
\draw (2,.5) node (Sb) {$S_\beta$};
\draw (2,-.5) node (Sa) {$S_\alpha$};
\draw (3,0) node (Tab_2) {$T^{-1}_{\alpha\beta}$};
\draw (4,.5) node (Sb2) {$S_\beta$};
\draw (4,-.5) node (Sa2) {$S_\alpha$};
\draw (5,0) node (Tab_3) {$T_{\alpha\beta}$};
\draw (6,.5) node (out_b) {};
\draw (6,-.5) node (out_a) {};

\draw[->] (in_b)--(Sb);
\draw[->] (in_a)--(Sa);
\draw[->] (Sb)--(Tab_2);
\draw[->] (Sa)--(Tab_2);
\draw[->] (Tab_2)--(Sb2);
\draw[->] (Tab_2)--(Sa2);
\draw[->] (Sb2)--(Tab_3);
\draw[->] (Sa2)--(Tab_3);
\draw[->] (Tab_3)--(out_b);
\draw[->] (Tab_3)--(out_a);
\end{tikzpicture}
\end{equation}
\end{lemma}

A given Hopf doublet can be used to construct a type of twisted tensor product Hopf algebra, called the Drinfeld double of the pair. 

\begin{definition}[Drinfeld Double] The \emph{Drinfeld double} $D(H_\alpha,H_\beta)$ of a Hopf doublet $(H_\alpha,H_\beta,\langle-\rangle)$ is the involutory Hopf algebra defined as follows. 

The underlying $k$-module is $D(H_\alpha,H_\beta) = H_\alpha \otimes H_\beta$. The coproduct $\Delta_{\alpha\beta}$, counit $\epsilon_{\alpha\beta}$ and unit $\eta_{\alpha\beta}$ are given by tensor products of the corresponding tensors of $H_\alpha$ and $H_\beta$. On the other hand, the product $M_{\alpha\beta}$ and antipode $S_{\alpha\beta}$ are given by the following tensor diagrams
\[
\begin{tikzpicture}
   
  \draw (0,1) node (in_Ma1) {};
  \draw (0,.5) node (in_Mb1) {};
  \draw (0,-.5) node (in_Ma2) {};
  \draw (0,-1) node (in_Mb2) {};
  \draw (1,0) node (Mab) {$M_{\alpha\beta}$};
  \draw (2,.3) node (out_Ma1) {};
  \draw (2,-.3) node (out_Mb1) {};

  \draw[->] (in_Ma1)--(Mab);
  \draw[->] (in_Mb1)--(Mab);
  \draw[->] (in_Ma2)--(Mab);
  \draw[->] (in_Mb2)--(Mab);
  \draw[->] (Mab)--(out_Ma1);
  \draw[->] (Mab)--(out_Mb1); 
 
  \draw (3,0) node (=1) {$:=$};

  \draw (4,1) node (in_Ma3) {};
  \draw (4,.5) node (in_Mb3) {};
  \draw (4,-.5) node (in_Ma4) {};
  \draw (4,-1) node (in_Mb4) {};
  \draw (5,0) node (Uab1) {$U_{\alpha\beta}$};
  \draw (6,1) node (Ma1) {$M_\alpha$};
  \draw (6,-1) node (Mb1) {$M_\beta$};
  \draw (7,.5) node (out_Ma2) {};
  \draw (7,-.5) node (out_Mb2) {};

  \draw[->] (in_Ma3)--(Ma1);
  \draw[->] (in_Mb3)--(Uab1);
  \draw[->] (in_Ma4)--(Uab1);
  \draw[->] (in_Mb4)--(Mb1);
  \draw[->] (Uab1) to [out=-30,in=-90,looseness=1] (Ma1);
  \draw[->] (Uab1) to [out=30,in=90,looseness=1] (Mb1);
  \draw[->] (Ma1) to (out_Ma2) {};
  \draw[->] (Mb1) to (out_Mb2) {};

  \draw (8,.3) node (Sab_a_in) {};
  \draw (8,-.3) node (Sab_b_in) {};
  \draw (9,0) node (Sab) {$S_{\alpha\beta}$};
  \draw (10,.3) node (Sab_a_out) {};
  \draw (10,-.3) node (Sab_b_out) {};

  \draw[->] (Sab_a_in)--(Sab);
  \draw[->] (Sab_b_in)--(Sab);
  \draw[->] (Sab)--(Sab_a_out);
  \draw[->] (Sab)--(Sab_b_out);

  \draw (10.5,0) node (=2) {$:=$};

  \draw (11,.3) node (Sab_a_in2) {};
  \draw (11,-.3) node (Sab_b_in2) {};
  \draw (12,.3) node (Sa1) {$S_\alpha$};
  \draw (12,-.3) node (Sb1) {$S_\beta$};
  \draw (13,0) node (Uab2) {$U_{\alpha\beta}$};
  \draw (14,.3) node (Sab_a_out2) {};
  \draw (14,-.3) node (Sab_b_out2) {};

  \draw[->] (Sab_a_in2)--(Sa1);
  \draw[->] (Sab_b_in2)--(Sb1);
  \draw[->] (Sa1)--(Uab2);
  \draw[->] (Sb1)--(Uab2);
  \draw[->] (Uab2)--(Sab_a_out2);
  \draw[->] (Uab2)--(Sab_b_out2);

\end{tikzpicture}
\]

The Drinfeld double defines a functor, in the sense that a map $f:(H_*,\langle-\rangle) \to (I_*,(-))$ induces a map of doubles.
\[D(f):D(H_\alpha,H_\beta) \to D(I_\alpha,I_\beta)\] \end{definition}

\begin{notation}[Double of $H$] For the tautological doublet $(H_\alpha,H_\beta,\langle-\rangle)$ of a single Hopf algebra $H$, where $H_\alpha = H^{*,\op{cop}}$, $H_\beta = H$ and $\langle-\rangle$ is the usual dual pairing, we will use the notation $D(H)$ to denote the Drinfeld double. 
\end{notation}

The data needed to formulate the $4$-manifold invariants discussed in this paper can be formulated in terms of Hopf triplets, in analogy with the $3$-manifold case.

\begin{definition}[Hopf Triplet] \label{def:hopf_triplet} An (involutory) \emph{Hopf triplet} $\mathcal{H} = (H_\alpha,H_\beta,H_\kappa,\langle-\rangle)$ consists of three (involutory) Hopf algebras $H_\alpha,H_\beta$ and $H_\kappa$, and three pairings
\[
\langle-\rangle_{\alpha\beta}:H_\alpha \otimes H_\beta \to k \qquad \langle-\rangle_{\beta\kappa}:H_\beta \otimes H_\kappa \to k \qquad \langle-\rangle_{\kappa\alpha}:H_\kappa \otimes H_\alpha \to k 
\]
The bilinear forms $\langle-\rangle_*$ must satisfy the following properties.
\begin{itemize}
\item[(a)] Each of the pairs $(H_\alpha,H_\beta,\langle-\rangle_{\alpha\beta})$, $(H_\beta,H_\kappa,\langle-\rangle_{\beta\kappa})$ and $(H_\kappa,H_\alpha,\langle-\rangle_{\kappa\alpha})$ is a Hopf doublet in the sense of Definition \ref{def:hopf_doublet}.

\item[(b)] The three $k$-linear maps form the following doubles
\[D(H_\alpha^{\op{op}},H_\beta^{\op{cop}}) \to H_\kappa^* \qquad D(H_\beta^{\op{op}},H_\kappa^{\op{cop}}) \to H^*_\alpha \qquad D(H_\kappa^{\op{op}},H_\alpha^{\op{cop}}) \to H^*_\beta\]
which we define, respectively, via the following tensor diagrams
\[\hspace{-24pt}\begin{tikzpicture}
  \draw (-.3,1) node (i_a1) {};
  \draw (-.3,0) node (i_b1) {};
  \draw (1,1) node (Pa1) {$\langle-\rangle_{\kappa\alpha}$};
  \draw (1,0) node (Pb1) {$\langle-\rangle_{\beta\kappa}$};
  \draw (2.5,.5) node (Dc1) {$\Delta_\kappa$};
  \draw (3.5,.5) node (o_c1) {};

  \draw[->] (i_a1)--(Pa1);
  \draw[->] (i_b1)--(Pb1);
  \draw[->] (Dc1) to [out=200,in=-30,looseness=1] (Pa1);
  \draw[->] (Dc1) to [out=160,in=30,looseness=1] (Pb1);
  \draw[<-] (Dc1)--(o_c1);

  \draw (4.2,1) node (i_b2) {};
  \draw (4.2,0) node (i_c2) {};
  \draw (5.5,1) node (Pb2) {$\langle-\rangle_{\alpha\beta}$};
  \draw (5.5,0) node (Pc2) {$\langle-\rangle_{\kappa\alpha}$};
  \draw (7,.5) node (Da2) {$\Delta_\alpha$};
  \draw (8,.5) node (o_a2) {};

  \draw[->] (i_b2)--(Pb2);
  \draw[->] (i_c2)--(Pc2);
  \draw[->] (Da2) to [out=200,in=-30,looseness=1] (Pb2);
  \draw[->] (Da2) to [out=160,in=30,looseness=1] (Pc2);
  \draw[<-] (Da2)--(o_a2);

  \draw (8.7,1) node (i_c3) {};
  \draw (8.7,0) node (i_a3) {};
  \draw (10,1) node (Pc3) {$\langle-\rangle_{\beta\kappa}$};
  \draw (10,0) node (Pa3) {$\langle-\rangle_{\alpha\beta}$};
  \draw (11.5,.5) node (Db3) {$\Delta_\beta$};
  \draw (12.5,.5) node (o_b3) {};

  \draw[->] (i_c3)--(Pc3);
  \draw[->] (i_a3)--(Pa3);
  \draw[->] (Db3) to [out=200,in=-30,looseness=1] (Pc3);
  \draw[->] (Db3) to [out=160,in=30,looseness=1] (Pa3);
  \draw[<-] (Db3)--(o_b3);
\end{tikzpicture}\]
are maps of Hopf algebras.  In our diagrams above, the ordering of the arrows entering the pairings does not matter so long as the arrows carry the correct labels; for instance, $\langle - \rangle_{\alpha \beta}$ should have one $\alpha$ in-arrow and one $\beta$ in-arrow but it does not matter how these in-arrows are co-located.
\end{itemize}
A map of Hopf triplets $f:\mathcal{H} \to \mathcal{I}$ is a triple of maps of Hopf algebras $f_*:H_* \to I_*$ for $* \in \{\alpha,\beta,\kappa\}$ intertwining the pairwise bilinear forms on both sides.
\end{definition}

\begin{remark}[Taking Hopf Algebras to be Involutory]
Although Hopf doublets and Hopf triplets are well-defined regardless of whether the constituent Hopf algebras are involutory or non-involutory, for the purposes of this paper we will take all Hopf algebras to be involutory unless otherwise specified.
\end{remark}

\begin{notation}[Pairing Notation] We will use two notations for the tensor diagrams of the pairings in a Hopf triplet. These notations are 
\[
\rightarrow \langle-\rangle_{\alpha\beta} \leftarrow \qquad \text{or} \qquad \rightarrow \bullet \leftarrow
\]
The first notation is the obvious one, while the second (which we deem \emph{bullet notation}) will be a helpful abbreviation that we will use exclusively in the more elaborate tensor diagrams in \S \ref{sec:trisection_kuperberg_invt}.  As mentioned in Definition~\ref{def:hopf_triplet}, the ordering of the in-arrows going into a pairing does not matter.
\end{notation}

We now prove a fundamental Lemma giving alternate formulations of Definition \ref{def:hopf_triplet}(b). This lemma will be used for many purposes, including the construction of examples of Hopf triplets and at a key point in the proof of invariance of our invariant in \S \ref{subsec:main_definition_and_properties}. 

\begin{lemma}[Fundamental Lemma of Triplets] \label{lem:fundamental_triplet_lemma} Let $(H,\langle-\rangle)$ be a set of three Hopf algebras $H_*$, $* \in \{\alpha,\beta,\kappa\}$, as in Definition \ref{def:hopf_triplet} along with three pairings $\langle-\rangle$ satisfying Definition \ref{def:hopf_triplet}(a). 

\vspace{5pt}

Then the following are equivalent.
\begin{itemize}
\item[(a)] $(H,\langle-\rangle)$ is a Hopf triplet, i.e. the three of the maps in Definition \ref{def:hopf_triplet}(b) 
\[D(H_\alpha^{\op{op}},H_\beta^{\op{cop}}) \to H_\kappa^* \qquad D(H_\beta^{\op{op}},H_\kappa^{\op{cop}}) \to H^*_\alpha \qquad D(H_\kappa^{\op{op}},H_\alpha^{\op{cop}}) \to H^*_\beta\]
are Hopf algebra maps.
\item[(b)] The map $D(H_\alpha^{\op{op}},H_\beta^{\op{cop}}) \to H_\kappa^*$ of Definition \ref{def:hopf_triplet}(b) is a Hopf algebra map.
\item[(c)]  The following identity relates the structure tensors of the triples.
\begin{equation}\label{eqn:symmetric_triplet_identity_c}
\begin{tikzpicture}
\draw (0,-3) node (ia1) {};
\draw (0,0) node (ib1) {};
\draw (4,-1.5) node (ic1) {};

\draw (1,-3) node (Da) {$\Delta_\alpha$};
\draw (1,0) node (Db) {$\Delta_\beta$};
\draw (3.2,-1.5) node (Dc) {$\Delta_\kappa$};

\draw (3.2,-3) node (Pca) {$\langle-\rangle_{\kappa\alpha}$};
\draw (1,-1.5) node (Pab) {$\langle-\rangle_{\alpha\beta}$};
\draw (3.2,0) node (Pbc) {$\langle-\rangle_{\beta\kappa}$};

\draw[->] (ia1)--(Da);
\draw[->] (ib1)--(Db);
\draw[->] (ic1)--(Dc);

\draw[->] (Da)--(Pab);
\draw[->] (Da)--(Pca);
\draw[->] (Db)--(Pbc);
\draw[->] (Db)--(Pab);
\draw[->] (Dc)--(Pca);
\draw[->] (Dc)--(Pbc);
\end{tikzpicture}
\begin{tikzpicture}
\draw (-1,-1.5) node (=) {$=$};

\draw (-1,-3) node (ia1) {};
\draw (-1,-0) node (ib1) {};
\draw (5,-1.5) node (ic1) {};

\draw (0,-3) node (Sa1) {$S_\alpha$};
\draw (0,-0) node (Sb1) {$S_\beta$};
\draw (4.2,-1.5) node (Sc1) {$S_\kappa$};

\draw (1,-3) node (Da) {$\Delta_\alpha$};
\draw (1,-0) node (Db) {$\Delta_\beta$};
\draw (3.2,-1.5) node (Dc) {$\Delta_\kappa$};

\draw (3.2,-3) node (Pca) {$\langle-\rangle_{\kappa\alpha}$};
\draw (1,-1.5) node (Pab) {$\langle-\rangle_{\alpha\beta}$};
\draw (3.2,-0) node (Pbc) {$\langle-\rangle_{\beta\kappa}$};

\draw[->] (ia1)--(Sa1);
\draw[->] (ib1)--(Sb1);
\draw[->] (ic1)--(Sc1);

\draw[->] (Sa1)--(Da);
\draw[->] (Sb1)--(Db);
\draw[->] (Sc1)--(Dc);

\draw[->] (Da)--(Pab);
\draw[->] (Da)--(Pca);
\draw[->] (Db)--(Pbc);
\draw[->] (Db)--(Pab);
\draw[->] (Dc)--(Pca);
\draw[->] (Dc)--(Pbc);
\end{tikzpicture}\end{equation}

\item[(d)]  The following identity relates the structure tensors of the triples.
\begin{equation}\label{eqn:symmetric_triplet_identity_d}
\begin{tikzpicture}
\draw (0,-3) node (ia1) {};
\draw (0,-0) node (ib1) {};
\draw (4,-1.5) node (ic1) {};

\draw (1,-3) node (Da) {$\Delta_\alpha$};
\draw (1,-0) node (Db) {$\Delta_\beta$};
\draw (3.2,-1.5) node (Dc) {$\Delta_\kappa$};

\draw (3.2,-3) node (Pca) {$\langle-\rangle_{\kappa\alpha}$};
\draw (1,-1) node (Pab) {$\langle-\rangle_{\alpha\beta}$};
\draw (3.2,-0) node (Pbc) {$\langle-\rangle_{\beta\kappa}$};

\draw (2,-3) node (Sa2) {$S_\alpha$};
\draw (2,-0) node (Sb2) {$S_\beta$};
\draw (1,-2) node (Sa3) {$S_\alpha$};

\draw[->] (ia1)--(Da);
\draw[->] (ib1)--(Db);
\draw[->] (ic1)--(Dc);

\draw[->] (Da)--(Sa3);
\draw[->] (Sa3)--(Pab);
\draw[->] (Da)--(Sa2);
\draw[->] (Sa2)--(Pca);
\draw[->] (Db)--(Sb2);
\draw[->] (Sb2)--(Pbc);
\draw[->] (Db)--(Pab);
\draw[->] (Dc)--(Pca);
\draw[->] (Dc)--(Pbc);
\end{tikzpicture}
\begin{tikzpicture}
\draw (-1,-1.5) node (=) {$=$};

\draw (-1,-3) node (ia1) {};
\draw (-1,-0) node (ib1) {};
\draw (5,-1.5) node (ic1) {};

\draw (0,-3) node (Sa1) {$S_\alpha$};
\draw (0,-0) node (Sb1) {$S_\beta$};
\draw (4.2,-1.5) node (Sc1) {$S_\kappa$};

\draw (1,-3) node (Da) {$\Delta_\alpha$};
\draw (1,-0) node (Db) {$\Delta_\beta$};
\draw (3.2,-1.5) node (Dc) {$\Delta_\kappa$};

\draw (3.2,-3) node (Pca) {$\langle-\rangle_{\kappa\alpha}$};
\draw (1,-1) node (Pab) {$\langle-\rangle_{\alpha\beta}$};
\draw (3.2,-0) node (Pbc) {$\langle-\rangle_{\beta\kappa}$};

\draw (2,-3) node (Sa2) {$S_\alpha$};
\draw (2,-0) node (Sb2) {$S_\beta$};
\draw (1,-2) node (Sa3) {$S_\alpha$};

\draw[->] (ia1)--(Sa1);
\draw[->] (ib1)--(Sb1);
\draw[->] (ic1)--(Sc1);

\draw[->] (Sa1)--(Da);
\draw[->] (Sb1)--(Db);
\draw[->] (Sc1)--(Dc);

\draw[->] (Da)--(Sa3);
\draw[->] (Sa3)--(Pab);
\draw[->] (Da)--(Sa2);
\draw[->] (Sa2)--(Pca);
\draw[->] (Db)--(Sb2);
\draw[->] (Sb2)--(Pbc);
\draw[->] (Db)--(Pab);
\draw[->] (Dc)--(Pca);
\draw[->] (Dc)--(Pbc);
\end{tikzpicture}\end{equation}
\end{itemize}
\end{lemma}
\begin{proof} We will demonstrate the following four equivalences between the statements 
\[
(a) \implies (b)\,, \qquad (b) \iff (d)\,, \qquad (c) \iff (d)\,, \qquad (c) \implies (a)
\]
Taken together, these equivalences prove the desired equivalence of (a)-(d). The fact that $(a) \implies (b)$ is trivial, and the fact that $(b) \iff (d) \iff (c)$ along with the symmetry of the tensor diagrams (\ref{eqn:symmetric_triplet_identity_d}) and (\ref{eqn:symmetric_triplet_identity_c}) will imply $(c) \implies (a)$. Thus, we prove the middle two equivalences: $(b) \iff (d)$ and $(c) \iff (d)$.

\vspace{5pt}

\emph{$(b) \iff (d)$.} Let $\Phi^\kappa_{\alpha\beta}:D(H^{\op{op}}_\alpha,H^{\op{cop}}_\beta) \to H^{*}_\kappa$ denote the linear map of (a). We must check that $\Phi^\kappa_{\alpha\beta}$ intertwines the product, unit, coproduct and counit. Below, we will denote the $T,U$ and $V$ tensors of Notation \ref{not:TUV_notation} for the doublet $(H^{\op{op}}_\alpha,H^{\op{cop}}_\beta,\langle-\rangle)$ by $T'_{\alpha\beta}$, $U'_{\alpha\beta}$ and $V'_{\alpha\beta}$, and we denote the structure tensors (multiplication, comultiplication, etc.) similarly. 

We will in fact show that $\Phi^\kappa_{\alpha\beta}$ automatically intertwines the coproduct, counit and unit (without assuming either (b) or (d)). We then show that $\Phi^\kappa_{\alpha\beta}$ intertwines comultiplication (i.e. satisfies (b)) if and only if (d) is satisfied.

\vspace{5pt}

(Counit/Comultiplication/Unit) Recall that the counit, coproduct and unit of $D(H_\alpha^{\op{op}},H_\beta^{\op{cop}})$ agree with those of $H_\alpha^{\op{op}} \otimes H_\beta^{\op{cop}}$, which are simply the tensor products of the corresponding tensors of $H^{\op{op}}_\alpha$ and $H^{\op{cop}}_\beta$. Thus it suffices to show that $\Phi^\kappa_{\alpha\beta}:H_\alpha^{\op{op}} \otimes H_\beta^{\op{cop}} \to H_\kappa^*$ intertwines the counit, coproduct and unit. The map $\Phi^\kappa_{\alpha\beta}$ is defined as the composition of two maps
\[
H^{\op{op}}_\alpha \otimes H^{\op{cop}}_\beta \xrightarrow{\langle-\rangle_{\kappa\beta} \otimes \langle-\rangle_{\beta\kappa}} H^*_\kappa \otimes H^*_\kappa \xrightarrow{\Delta_\kappa} H^*_\kappa
\]
Above, the first map is the tensor product of the Hopf algebra maps $H^{\op{op}}_\alpha \to H^*_\kappa$ and $H^{\op{cop}}_\beta \to H^*_\kappa$ induced by the pairings $\langle-\rangle_{\kappa\alpha}$ and $\langle-\rangle_{\beta\kappa}$, and the second map is the product in $H^*_\kappa$ (i.e. the coproduct in $H_\kappa$). 

The first map is a Hopf algebra map because it is the tensor product of Hopf algebra maps. Thus, it suffices to check that the product $H^*_\kappa \otimes H^*_\kappa \to H^*_\kappa$ intertwines the counit, coproduct and unit. The counit and coproduct are intertwined because multiplication is a coalgebra homomorphism. The unit is intertwined because this is equivalent to the fact that the unit squares to itself.

\vspace{5pt}

(Multiplication) Consider the following homomorphism identity. 
\begin{equation}\label{eqn:homomorphism_identity_1}\begin{tikzpicture}
\draw (0,1) node (ib1) {};
\draw (0,.5) node (ia1) {};
\draw (0,-.5) node (ib2) {};
\draw (0,-1) node (ia2) {};
\draw (1,.5) node (F1) {$\Phi^\kappa_{\alpha\beta}$};
\draw (1,-.5) node (F2) {$\Phi^\kappa_{\alpha\beta}$};
\draw (2.5,0) node (Dc) {$\Delta_\kappa$};
\draw (3.5,0) node (oc1) {};

\draw[->] (ib1)--(F1);
\draw[->] (ia1)--(F1);
\draw[->] (ib2)--(F2);
\draw[->] (ia2)--(F2);

\draw[->] (Dc) to [out=200,in=0,looseness=1] (F1);
\draw[->] (Dc) to [out=160,in=0,looseness=1] (F2);
\draw[<-] (Dc)--(oc1);

\draw (4.25,0) node (=1) {$=$};

\draw (5,1) node (ib3) {};
\draw (5,.5) node (ia3) {};
\draw (5,-.5) node (ib4) {};
\draw (5,-1) node (ia4) {};

\draw (6,0) node (Mab) {$M'_{\alpha\beta}$};
\draw (8,0) node (F3) {$\Phi^\kappa_{\alpha\beta}$};

\draw (9,0) node (oc2) {};

\draw[->] (ib3)--(Mab);
\draw[->] (ia3)--(Mab);
\draw[->] (ib4)--(Mab);
\draw[->] (ia4)--(Mab);
\draw[->] (Mab) to [out=20,in=160,looseness=1] (F3);
\draw[->] (Mab) to [out=-20,in=200,looseness=1] (F3);

\draw[<-] (F3)--(oc2);
\end{tikzpicture}\end{equation}
We will now show that this identity is equivalent to (\ref{eqn:symmetric_triplet_identity_d}).

First, consider the left-hand side of (\ref{eqn:homomorphism_identity_1}). By precomposing the middle two inputs with $V_{\alpha\beta}'$, we can convert that tensor into the following one.
\begin{equation}\label{eqn:homomorphism_identity_2} 
\begin{tikzpicture}
\draw (0,2) node (in_a2) {};
\draw (0,1) node (in_b1) {};
\draw (0,-1) node (in_a1) {};
\draw (0,-2) node (in_b2) {};

\draw (.8,-1) node (Sa1) {$S_\alpha$};
\draw (.8,1) node (Sb1) {$S_\beta$};

\draw (4,2) node (Pca2) {$\langle-\rangle_{\kappa\alpha}$};
\draw (2,1) node (Db1) {$\Delta_\beta$};
\draw (2,0) node (Pab) {$\langle-\rangle_{\alpha\beta}$};
\draw (2,-1) node (Da1) {$\Delta_\alpha$};
\draw (4,-2) node (Pbc2) {$\langle-\rangle_{\beta\kappa}$};

\draw (3,1) node (Sb2) {$S_\beta$};

\draw (4.2,1) node (Pbc1) {$\langle-\rangle_{\beta\kappa}$};
\draw (4.2,-1) node (Pca1) {$\langle-\rangle_{\kappa\alpha}$};

\draw (5.3,1.5) node (Dc1) {$\Delta_\kappa$};
\draw (5.3,-1.5) node (Dc2) {$\Delta_\kappa$};

\draw (6,0) node (Dc3) {$\Delta_\kappa$};
\draw (6.8,0) node (out_1) {};

\draw[->] (in_a1)--(Sa1);
\draw[->] (in_b1)--(Sb1);

\draw[->] (in_a2)--(Pca2);
\draw[->] (Sa1)--(Da1);
\draw[->] (Sb1)--(Db1);
\draw[->] (in_b2)--(Pbc2);

\draw[->] (Da1)--(Pab);
\draw[->] (Db1) to [out=-70,in=90,looseness=1] (Pab);

\draw[->] (Da1)--(Pca1);
\draw[->] (Db1) to [out=-90,in=210,looseness=1] (Sb2);

\draw[->] (Sb2)--(Pbc1);

\draw[->] (Dc1) to [out=200,in=-20,looseness=1] (Pca2);
\draw[->] (Dc1) to [out=180,in=90,looseness=1] (Pbc1);
\draw[->] (Dc2) to [out=180,in=-90,looseness=1] (Pca1);
\draw[->] (Dc2) to [out=160,in=20,looseness=1] (Pbc2);

\draw[->] (Dc3) to [out=200,in=-110,looseness=1] (Dc1);
\draw[->] (Dc3) to [out=160,in=110,looseness=1] (Dc2);

\draw[->] (out_1)--(Dc3); 
\end{tikzpicture}
\begin{tikzpicture}
\draw (-.5,0) node (=1) {$=$};

\draw (0,2) node (in_a2) {};
\draw (0,1) node (in_b1) {};
\draw (0,-1) node (in_a1) {};
\draw (0,-2) node (in_b2) {};

\draw (.8,-1) node (Sa1) {$S_\alpha$};
\draw (.8,1) node (Sb1) {$S_\beta$};

\draw (4,2) node (Pca2) {$\langle-\rangle_{\kappa\alpha}$};
\draw (2,1) node (Db1) {$\Delta_\beta$};
\draw (2,0) node (Pab) {$\langle-\rangle_{\alpha\beta}$};
\draw (2,-1) node (Da1) {$\Delta_\alpha$};
\draw (4,-2) node (Pbc2) {$\langle-\rangle_{\beta\kappa}$};

\draw (3,1) node (Sb2) {$S_\beta$};

\draw (4.2,1) node (Pbc1) {$\langle-\rangle_{\beta\kappa}$};
\draw (4.2,-1) node (Pca1) {$\langle-\rangle_{\kappa\alpha}$};

\draw (4.6,0) node (Dc1) {$\Delta_\kappa$};
\draw (6,0) node (Dc2) {$\Delta_\kappa$};
\draw (6.8,0) node (out_1) {};

\draw[->] (in_a1)--(Sa1);
\draw[->] (in_b1)--(Sb1);

\draw[->] (in_a2)--(Pca2);
\draw[->] (Sa1)--(Da1);
\draw[->] (Sb1)--(Db1);
\draw[->] (in_b2)--(Pbc2);

\draw[->] (Da1)--(Pab);
\draw[->] (Db1) to [out=-70,in=90,looseness=1] (Pab);

\draw[->] (Da1)--(Pca1);
\draw[->] (Db1) to [out=-90,in=210,looseness=1] (Sb2);

\draw[->] (Sb2)--(Pbc1);

\draw[->] (Dc2) to [out=200,in=0,looseness=1] (Pca2);
\draw[->] (Dc1) to [out=190,in=-110,looseness=1] (Pbc1);
\draw[->] (Dc1) to [out=170,in=110,looseness=1] (Pca1);
\draw[->] (Dc2) to [out=160,in=0,looseness=1] (Pbc2);

\draw[->] (Dc2)--(Dc1);
\draw[->] (out_1)--(Dc2);
\end{tikzpicture}
\end{equation}
Likewise, consider the right-hand side of (\ref{eqn:homomorphism_identity_1}). By similarly precomposing the middle two inputs with $V_{\alpha\beta}'$ and applying (\ref{eqn:TUV_identity_1}) from Lemma \ref{lem:UTV_identity}, we acquire the following tensor.
\begin{equation}\label{eqn:homomorphism_identity_3} 
\begin{tikzpicture}
\draw (0,2) node (in_a2) {};
\draw (0,1) node (in_b1) {};
\draw (0,-1) node (in_a1) {};
\draw (0,-2) node (in_b2) {};

\draw (2,1) node (Db1) {$\Delta_\beta$};
\draw (2,0) node (Pab) {$\langle-\rangle_{\alpha\beta}$};
\draw (2,-1) node (Da1) {$\Delta_\alpha$};

\draw (3.5,.5) node (Sa3) {$S_\alpha$};

\draw (4,2) node (Ma2) {$M_\alpha$};
\draw (4,-2) node (Mb2) {$M_\beta$};

\draw (5,1) node (Pca) {$\langle-\rangle_{\kappa\alpha}$};
\draw (5,-1) node (Pbc) {$\langle-\rangle_{\beta\kappa}$};

\draw (6,0) node (Dc) {$\Delta_\kappa$};
\draw (6.8,0) node (out1) {};

\draw[->] (in_b2)--(Mb2);
\draw[->] (in_a1)--(Da1);
\draw[->] (in_b1)--(Db1);
\draw[->] (in_a2) to [out=0,in=200,looseness=1] (Ma2);

\draw[->] (Da1)--(Pab);
\draw[->] (Db1) to [out=-50,in=70,looseness=1] (Pab);

\draw[->] (Db1) to [out=-70,in=110,looseness=1] (Mb2);
\draw[->] (Da1)--(Sa3);
\draw[->] (Sa3) to [out=90,in=180,looseness=1] (Ma2);

\draw[->] (Mb2)--(Pbc);
\draw[->] (Ma2)--(Pca);

\draw[->] (Dc) to [out=160,in=90,looseness=1] (Pbc);
\draw[->] (Dc) to [out=200,in=-90,looseness=1] (Pca);

\draw[->] (out_1)--(Dc); 
\end{tikzpicture}
\begin{tikzpicture}
\draw (.5,0) node (=1) {$=$};

\draw (1,2) node (in_a2) {};
\draw (1,1) node (in_b1) {};
\draw (1,-1) node (in_a1) {};
\draw (1,-2) node (in_b2) {};

\draw (2,-1) node (Da1) {$\Delta_\alpha$};
\draw (2,0) node (Pab) {$\langle-\rangle_{\alpha\beta}$};
\draw (2,1) node (Db1) {$\Delta_\beta$};

\draw (3,-1) node (Sa3) {$S_\alpha$};

\draw (4,2) node (Pca2) {$\langle-\rangle_{\kappa\alpha}$};
\draw (4.5,1) node (Pbc1) {$\langle-\rangle_{\beta\kappa}$};
\draw (4.5,-1) node (Pca1) {$\langle-\rangle_{\kappa\alpha}$};
\draw (4,-2) node (Pbc2) {$\langle-\rangle_{\beta\kappa}$};

\draw (5.5,0) node (D1) {$\Delta_\kappa$};
\draw (7,0) node (D2) {$\Delta_\kappa$};
\draw (7.8,0) node (out_1) {};

\draw[->] (in_a2)--(Pca2);
\draw[->] (in_b1)--(Db1);
\draw[->] (in_a1)--(Da1);
\draw[->] (in_b2)--(Pbc2);

\draw[->] (Da1)--(Pab);
\draw[->] (Db1) to [out=-70,in=90,looseness=1] (Pab);

\draw[->] (Db1) to [out=-90,in=180,looseness=1] (Pbc1);
\draw[->] (Da1)--(Sa3);
\draw[->] (Sa3)--(Pca1);

\draw[->] (D1)--(Pbc1);
\draw[->] (D1)--(Pca1);
\draw[->] (D2) to [out=270,in=0,looseness=1] (Pca2);
\draw[->] (D2) to [out=90,in=0,looseness=1] (Pbc2);

\draw[->] (D2)--(D1);
\draw[->] (out_1)--(D2); 
\end{tikzpicture}
\end{equation}
Examining the middle portions of (\ref{eqn:homomorphism_identity_2}) and (\ref{eqn:homomorphism_identity_3}), we find that the equality (\ref{eqn:homomorphism_identity_1}) is equivalent to the identity
\begin{equation}\label{eqn:homomorphism_identity_4} 
\begin{tikzpicture}
\draw (0,1) node (in_b1) {};
\draw (0,-1) node (in_a1) {};

\draw (.8,-1) node (Sa1) {$S_\alpha$};
\draw (.8,1) node (Sb1) {$S_\beta$};

\draw (2,1) node (Db1) {$\Delta_\beta$};
\draw (2,0) node (Pab) {$\langle-\rangle_{\alpha\beta}$};
\draw (2,-1) node (Da1) {$\Delta_\alpha$};

\draw (3,1) node (Sb2) {$S_\beta$};

\draw (4.2,1) node (Pbc1) {$\langle-\rangle_{\beta\kappa}$};
\draw (4.2,-1) node (Pca1) {$\langle-\rangle_{\kappa\alpha}$};

\draw (4.6,0) node (Dc1) {$\Delta_\kappa$};
\draw (5.8,0) node (out_1) {};

\draw[->] (in_a1)--(Sa1);
\draw[->] (in_b1)--(Sb1);

\draw[->] (Sa1)--(Da1);
\draw[->] (Sb1)--(Db1);

\draw[->] (Da1)--(Pab);
\draw[->] (Db1) to [out=-70,in=90,looseness=1] (Pab);

\draw[->] (Da1)--(Pca1);
\draw[->] (Db1) to [out=-90,in=210,looseness=1] (Sb2);

\draw[->] (Sb2)--(Pbc1);

\draw[->] (Dc1) to [out=190,in=-110,looseness=1] (Pbc1);
\draw[->] (Dc1) to [out=170,in=110,looseness=1] (Pca1);

\draw[->] (out_1)--(Dc1);
\end{tikzpicture}
\begin{tikzpicture}
\draw (.5,0) node (=1) {$=$};

\draw (1,1) node (in_b1) {};
\draw (1,-1) node (in_a1) {};

\draw (2,-1) node (Da1) {$\Delta_\alpha$};
\draw (2,0) node (Pab) {$\langle-\rangle_{\alpha\beta}$};
\draw (2,1) node (Db1) {$\Delta_\beta$};

\draw (3,-1) node (Sa3) {$S_\alpha$};

\draw (4.5,1) node (Pbc1) {$\langle-\rangle_{\beta\kappa}$};
\draw (4.5,-1) node (Pca1) {$\langle-\rangle_{\kappa\alpha}$};

\draw (5.5,0) node (D1) {$\Delta_\kappa$};
\draw (6.8,0) node (out_1) {};

\draw[->] (in_b1)--(Db1);
\draw[->] (in_a1)--(Da1);

\draw[->] (Da1)--(Pab);
\draw[->] (Db1) to [out=-70,in=90,looseness=1] (Pab);

\draw[->] (Db1) to [out=-90,in=180,looseness=1] (Pbc1);
\draw[->] (Da1)--(Sa3);
\draw[->] (Sa3)--(Pca1);

\draw[->] (D1)--(Pbc1);
\draw[->] (D1)--(Pca1);

\draw[->] (out_1)--(D1); 
\end{tikzpicture}
\end{equation}
By applying the anti-homomorphism property of the antipodes $S_\alpha,S_\beta$ and $S_\kappa$ at all of the coproduct nodes in (\ref{eqn:homomorphism_identity_4}) that have twisted output edges, we can convert that identity into the following one.
\begin{equation}\label{eqn:homomorphism_identity_5} 
\begin{tikzpicture}
\draw (-1,-3) node (ia1) {};
\draw (-1,-0) node (ib1) {};
\draw (5,-1.5) node (ic1) {};

\draw (0,-3) node (Sa1) {$S_\alpha$};
\draw (4.2,-1.5) node (Sc1) {$S_\kappa$};

\draw (1,-3) node (Da) {$\Delta_\alpha$};
\draw (1,-0) node (Db) {$\Delta_\beta$};
\draw (3.2,-1.5) node (Dc) {$\Delta_\kappa$};

\draw (3.2,-3) node (Pca) {$\langle-\rangle_{\kappa\alpha}$};
\draw (1,-1) node (Pab) {$\langle-\rangle_{\alpha\beta}$};
\draw (3.2,-0) node (Pbc) {$\langle-\rangle_{\beta\kappa}$};

\draw (2,-3) node (Sa2) {$S_\alpha$};
\draw (2,-0) node (Sb2) {$S_\beta$};
\draw (1,-2) node (Sa3) {$S_\alpha$};

\draw[->] (ia1)--(Sa1);
\draw[->] (ic1)--(Sc1);

\draw[->] (Sa1)--(Da);
\draw[->] (ib1)--(Db);
\draw[->] (Sc1)--(Dc);

\draw[->] (Da)--(Sa3);
\draw[->] (Sa3)--(Pab);
\draw[->] (Da)--(Sa2);
\draw[->] (Sa2)--(Pca);
\draw[->] (Db)--(Sb2);
\draw[->] (Sb2)--(Pbc);
\draw[->] (Db)--(Pab);
\draw[->] (Dc)--(Pca);
\draw[->] (Dc)--(Pbc);
\end{tikzpicture}
\begin{tikzpicture}
\draw (-1.5,-1.5) node (=1) {$=$};

\draw (-1,-3) node (ia1) {};
\draw (-1,0) node (ib1) {};
\draw (4,-1.5) node (ic1) {};
 
\draw (0,0) node (Sb1) {$S_\beta$};

\draw (1,-3) node (Da) {$\Delta_\alpha$};
\draw (1,-0) node (Db) {$\Delta_\beta$};
\draw (3.2,-1.5) node (Dc) {$\Delta_\kappa$};

\draw (3.2,-3) node (Pca) {$\langle-\rangle_{\kappa\alpha}$};
\draw (1,-1) node (Pab) {$\langle-\rangle_{\alpha\beta}$};
\draw (3.2,-0) node (Pbc) {$\langle-\rangle_{\beta\kappa}$};

\draw (2,-3) node (Sa2) {$S_\alpha$};
\draw (2,-0) node (Sb2) {$S_\beta$};
\draw (1,-2) node (Sa3) {$S_\alpha$};

\draw[->] (ia1)--(Da);
\draw[->] (ib1)--(Sb1);
\draw[->] (Sb1)--(Db);
\draw[->] (ic1)--(Dc);

\draw[->] (Da)--(Sa3);
\draw[->] (Sa3)--(Pab);
\draw[->] (Da)--(Sa2);
\draw[->] (Sa2)--(Pca);
\draw[->] (Db)--(Sb2);
\draw[->] (Sb2)--(Pbc);
\draw[->] (Db)--(Pab);
\draw[->] (Dc)--(Pca);
\draw[->] (Dc)--(Pbc);
\end{tikzpicture}
\end{equation}
This identity, (\ref{eqn:homomorphism_identity_5}), is evidently equivalent to (\ref{eqn:symmetric_triplet_identity_d}), modulo composition at the inputs with some antipodes and commutation of some antipodes across pairings. Thus we have shown that $\Phi^\kappa_{\alpha\beta}$ intertwines multiplication if and only if the identity (\ref{eqn:symmetric_triplet_identity_d}), and thus (d), is satisfied. This concludes the proof that $(b) \iff (d)$. 

\vspace{5pt}

\emph{$(c) \iff (d)$.} We want to show that (\ref{eqn:symmetric_triplet_identity_c}) and (\ref{eqn:symmetric_triplet_identity_d}) are equivalent. It is convenient for us to rewrite these two identities in terms of the tensors $T_{\alpha\beta}$ and $T^{-1}_{\alpha\beta}$, as well as the antipode tensors. These identities are, respectively,
\begin{equation} \label{eqn:symmetric_triplet_identity_c_using_T}
\begin{tikzpicture}
\draw (0,-.5) node (in_a) {};
\draw (0,.5) node (in_b) {};
\draw (1,0) node (Tab) {$T_{\alpha\beta}$};
\draw (2.5,.5) node (Pbc) {$\langle-\rangle_{\beta\kappa}$};
\draw (2.5,-.5) node (Pca) {$\langle-\rangle_{\kappa\alpha}$};
\draw (4,0) node (Dc) {$\Delta_\kappa$};
\draw (4.8,0) node (in_c) {};

\draw[->] (in_a)--(Tab);
\draw[->] (in_b)--(Tab);
\draw[->] (Tab)--(Pbc);
\draw[->] (Tab)--(Pca);
\draw[->] (Dc)--(Pbc);
\draw[->] (Dc)--(Pca);
\draw[->] (in_c)--(Dc);
\end{tikzpicture}
\begin{tikzpicture}
\draw (-1.5,0) node (=1) {$=$};

\draw (-1,-.5) node (in_a) {};
\draw (-1,.5) node (in_b) {};
\draw (0,-.5) node (Sa) {$S_\alpha$};
\draw (0,.5) node (Sb) {$S_\beta$};
\draw (1,0) node (Tab) {$T_{\alpha\beta}$};
\draw (2.5,.5) node (Pbc) {$\langle-\rangle_{\beta\kappa}$};
\draw (2.5,-.5) node (Pca) {$\langle-\rangle_{\kappa\alpha}$};
\draw (4,0) node (Dc) {$\Delta_\kappa$};
\draw (5,0) node (Sc) {$S_\kappa$};
\draw (5.8,0) node (in_c) {};

\draw[->] (in_a)--(Sa);
\draw[->] (in_b)--(Sb);
\draw[->] (Sa)--(Tab);
\draw[->] (Sb)--(Tab);
\draw[->] (Tab)--(Pbc);
\draw[->] (Tab)--(Pca);
\draw[->] (Dc)--(Pbc);
\draw[->] (Dc)--(Pca);
\draw[->] (Sc)--(Dc);
\draw[->] (in_c)--(Sc);
\end{tikzpicture}
\end{equation}
\vskip.05cm
\begin{equation} \label{eqn:symmetric_triplet_identity_d_using_T}
\begin{tikzpicture}
\draw (0,-.5) node (in_a) {};
\draw (0,.5) node (in_b) {};
\draw (1,0) node (Tab_inv) {$T^{-1}_{\alpha\beta}$};
\draw (2,-.5) node (Sa1) {$S_\alpha$};
\draw (2,.5) node (Sb1) {$S_\beta$};
\draw (3.5,.5) node (Pbc) {$\langle-\rangle_{\beta\kappa}$};
\draw (3.5,-.5) node (Pca) {$\langle-\rangle_{\kappa\alpha}$};
\draw (5,0) node (Dc) {$\Delta_\kappa$};
\draw (5.8,0) node (in_c) {};

\draw[->] (in_a)--(Tab_inv);
\draw[->] (in_b)--(Tab_inv);
\draw[->] (Tab_inv)--(Sb1);
\draw[->] (Sb1)--(Pbc);
\draw[->] (Tab_inv)--(Sa1);
\draw[->] (Sa1)--(Pca);
\draw[->] (Dc)--(Pbc);
\draw[->] (Dc)--(Pca);
\draw[->] (in_c)--(Dc);
\end{tikzpicture}
\begin{tikzpicture}
\draw (-1.5,0) node (=1) {$=$};

\draw (-1,-.5) node (in_a) {};
\draw (-1,.5) node (in_b) {};
\draw (0,-.5) node (Sa) {$S_\alpha$};
\draw (0,.5) node (Sb) {$S_\beta$};
\draw (1,0) node (Tab_inv) {$T^{-1}_{\alpha\beta}$};
\draw (2,-.5) node (Sa1) {$S_\alpha$};
\draw (2,.5) node (Sb1) {$S_\beta$};
\draw (3.5,.5) node (Pbc) {$\langle-\rangle_{\beta\kappa}$};
\draw (3.5,-.5) node (Pca) {$\langle-\rangle_{\kappa\alpha}$};
\draw (5,0) node (Dc) {$\Delta_\kappa$};
\draw (6,0) node (Sc) {$S_\kappa$};
\draw (6.8,0) node (in_c) {};

\draw[->] (in_a)--(Sa);
\draw[->] (in_b)--(Sb);
\draw[->] (Sa)--(Tab_inv);
\draw[->] (Sb)--(Tab_inv);
\draw[->] (Tab_inv)--(Sb1);
\draw[->] (Sb1)--(Pbc);
\draw[->] (Tab_inv)--(Sa1);
\draw[->] (Sa1)--(Pca);
\draw[->] (Dc)--(Pbc);
\draw[->] (Dc)--(Pca);
\draw[->] (Sc)--(Dc);
\draw[->] (in_c)--(Sc);
\end{tikzpicture}
\end{equation}
With the above rewriting in mind, we consider the following tensor diagram.
\begin{equation} \label{eqn:main_triplet_lemma_cancelling_tensor}
\begin{tikzpicture}
\draw (0,.5) node (in_b) {};
\draw (0,-.5) node (in_a) {};
\draw (1,0) node (Tab_1) {$T^{-1}_{\alpha\beta}$};
\draw (2,.5) node (Sb) {$S_\beta$};
\draw (2,-.5) node (Sa) {$S_\alpha$};
\draw (3,0) node (Tab_2) {$T^{-1}_{\alpha\beta}$};
\draw (4,.5) node (out_b) {};
\draw (4,-.5) node (out_a) {};

\draw[->] (in_b)--(Tab_1);
\draw[->] (in_a)--(Tab_1);
\draw[->] (Tab_1)--(Sb);
\draw[->] (Tab_1)--(Sa);
\draw[->] (Sb)--(Tab_2);
\draw[->] (Sa)--(Tab_2);
\draw[->] (Tab_2)--(out_b);
\draw[->] (Tab_2)--(out_a);
\end{tikzpicture}
\end{equation}
Note that this tensor is invertible. This follows by expressing it in terms of $T_{\alpha\beta}^{-1}$ (see Notation \ref{not:TUV_notation}) and the antipodes $S_\alpha$ and $S_\beta$. 

Now observe that composing the tensor (\ref{eqn:main_triplet_lemma_cancelling_tensor}) on the left with the left-most tensor of (\ref{eqn:symmetric_triplet_identity_c_using_T}) produces (using involutarity and the identities in Notation \ref{not:TUV_notation}) the left-most tensor of (\ref{eqn:symmetric_triplet_identity_d_using_T}). Thus we must show that the right-most tensor of (\ref{eqn:symmetric_triplet_identity_c_using_T}) becomes the right-most tensor of (\ref{eqn:symmetric_triplet_identity_d_using_T}). In particular, it suffices to check that
\begin{equation} \label{eqn:main_triplet_lemma_cancelling_tensor_final_id}
\begin{tikzpicture}
\draw (0,.5) node (in_b) {};
\draw (0,-.5) node (in_a) {};
\draw (1,0) node (Tab_1) {$T^{-1}_{\alpha\beta}$};
\draw (2,.5) node (Sb) {$S_\beta$};
\draw (2,-.5) node (Sa) {$S_\alpha$};
\draw (3,0) node (Tab_2) {$T^{-1}_{\alpha\beta}$};
\draw (4,.5) node (Sb2) {$S_\beta$};
\draw (4,-.5) node (Sa2) {$S_\alpha$};
\draw (5,0) node (Tab_3) {$T_{\alpha\beta}$};
\draw (6,.5) node (out_b) {};
\draw (6,-.5) node (out_a) {};

\draw[->] (in_b)--(Tab_1);
\draw[->] (in_a)--(Tab_1);
\draw[->] (Tab_1)--(Sb);
\draw[->] (Tab_1)--(Sa);
\draw[->] (Sb)--(Tab_2);
\draw[->] (Sa)--(Tab_2);
\draw[->] (Tab_2)--(Sb2);
\draw[->] (Tab_2)--(Sa2);
\draw[->] (Sb2)--(Tab_3);
\draw[->] (Sa2)--(Tab_3);
\draw[->] (Tab_3)--(out_b);
\draw[->] (Tab_3)--(out_a);
\end{tikzpicture}
\begin{tikzpicture}

\draw (.5,0) node (=1) {$=$};
\draw (1,.5) node (in_b) {};
\draw (1,-.5) node (in_a) {};
\draw (2,.5) node (Sb) {$S_\beta$};
\draw (2,-.5) node (Sa) {$S_\alpha$};
\draw (3,0) node (Tab_2) {$T^{-1}_{\alpha\beta}$};
\draw (4,.5) node (Sb2) {$S_\beta$};
\draw (4,-.5) node (Sa2) {$S_\alpha$};
\draw (5,.5) node (out_b) {};
\draw (5,-.5) node (out_a) {};

\draw[->] (in_b)--(Sb);
\draw[->] (in_a)--(Sa);
\draw[->] (Sb)--(Tab_2);
\draw[->] (Sa)--(Tab_2);
\draw[->] (Tab_2)--(Sb2);
\draw[->] (Tab_2)--(Sa2);
\draw[->] (Sb2)--(out_b);
\draw[->] (Sa2)--(out_a);
\end{tikzpicture}
\end{equation}
This identity is the second identity (\ref{eqn:TUV_identity_2}) in Lemma \ref{lem:UTV_identity}. 
\end{proof}

An immediate corollary of the fundamental lemma for triplets is the following, which provides a rich source of examples.

\begin{corollary} \label{cor:map_version_of_triplet} Let $(H_\alpha,H_\beta,(-))$ be an involutory Hopf doublet and let $\pi:D(H_\alpha^{\op{op}},H_\beta^{\op{cop}}) \to H^*_\kappa$ be a Hopf algebra map to a third involutory Hopf algebra. Denote by $\iota_{\alpha}: H_\alpha \to D(H_\alpha^{\op{op}},H_\beta^{\op{cop}})$ the embedding sending $v$ to $v \otimes 1$ and similarly by $\iota_\beta: H_\beta \to D(H_\alpha^{\op{op}},H_\beta^{\op{cop}}) $ sending $w$ to $1 \otimes w$.

Then $\mathcal{H} = (H_\alpha, H_\beta, H_\kappa;\langle-\rangle)$ has the structure of a Hopf triplet with the pairings
\[\begin{tikzpicture}
  \draw (.2,0) node (Pab) {$\langle-\rangle_{\alpha\beta}$};
  \draw (1,0) node (=ab) {$:=$};
  \draw (1.5,0) node (i_a1) {};
  \draw (4,0) node (i_b1) {};
  \draw (2.75,0) node (Qab) {$(-)$};

  \draw[->] (i_a1)--(Qab);
  \draw[->] (i_b1)--(Qab);

  \draw (5.2,0) node (Pbc) {$\langle-\rangle_{\beta\kappa}$};
  \draw (6,0) node (=bc) {$:=$};
  \draw (6.3,0) node (i_b2) {};
  \draw (9.2,0) node (i_c2) {};
  \draw (7.2,0) node (Ibc) {$\iota_\beta$};
  \draw (8.3,0) node (Qbc) {$\pi$};

  \draw[->] (i_b2)--(Ibc);
  \draw[->] (Ibc)--(Qbc);
  \draw[->] (i_c2)--(Qbc);

  \draw (10.2,0) node (Pca) {$\langle-\rangle_{\kappa\alpha}$};
  \draw (11,0) node (=ca) {$:=$};
  \draw (14.2,0) node (i_a3) {};
  \draw (11.3,0) node (i_c3) {};
  \draw (13.2,0) node (Ica) {$\iota_\alpha$};
  \draw (12.3,0) node (Qca) {$\pi$};

  \draw[->] (i_a3)--(Ica);
  \draw[->] (Ica)--(Qca);
  \draw[->] (i_c3)--(Qca);
\end{tikzpicture}\]
\end{corollary} 

\begin{proof} We must check Definition \ref{def:hopf_triplet}(a) and (b). Clearly $(H_\alpha,H_\beta,\langle-\rangle_{\alpha\beta})$ is a Hopf doublet. The brackets $\langle-\rangle_{\beta\kappa}$ and $\langle-\rangle_{\kappa\alpha}$ as defined above provide Hopf algebra maps
\[
H_\alpha^{\op{op}} \to H_\kappa^* \qquad H_\beta^{\op{cop}} \to H_\kappa^* 
\]
Dualizing the map on the left and cop-ing the map on the right, we find that the pairings are equivalently Hopf algebra maps
\[
H_\kappa \to (H^{\op{op}}_\alpha)^* = H^{*,\op{cop}}_\alpha \qquad H_\beta \to H_\kappa^{*,\op{cop}} 
\]
Thus every pair of indices determines a Hopf doublet, and Definition \ref{def:hopf_triplet}(a) is satisfied. Definition \ref{def:hopf_triplet}(b) follows from the fact that the map $D(H_\alpha^{\op{op}},H_\beta^{\op{cop}}) \to H^*_\kappa$ defined as Definition \ref{def:hopf_triplet}(b) agrees with $\pi$, which is a Hopf algebra map by hypothesis.
\end{proof}

To conclude this part, we provide a few more specific examples of Hopf triplets.

\begin{example} \label{ex:basic_examples_of_triplets} Here are some basic examples of Hopf triplets.
\begin{enumerate}[label = (\alph*)]
  \item(Tautological) Let $(H_\alpha,H_\beta,(-))$ be a finite-dimensional, involutory Hopf doublet. Then we may form a tautological triplet by letting $H_\kappa := D(H_\alpha^{\op{op}},H_\beta^{\op{cop}})^*$. We can define a map
  \[
  \pi:D(H_\alpha^{\op{op}},H_\beta^{\op{cop}}) \to H_\kappa^* = D(H_\alpha^{\op{op}},H_\beta^{\op{cop}})
  \]
  to be the identity. Corollary \ref{cor:map_version_of_triplet} then gives a Hopf triplet structure on these three Hopf algebras using $\pi$.
% \item (Drinfeld double) Let $H, K$ be two involutory Hopf algebras and $\phi: D(H) \to K$ be a Hopf algebra morphism. Denote the restriction of $\phi$ on $H^{*,\Cop}$ by $\phi_1: H^{*,\Cop} \to K$ and that on $H$ by $\phi_2: H \to K$. Then both $\phi_1 $ and $\phi_2$ are Hopf algebra morphisms and $\phi = M \circ (\phi_1 \otimes \phi_2)$.  We can construct a Hopf triplet $\Hopf{H, K; \phi} = (H_{\alpha}, H_{\beta}, H_{\kappa}; \langle\;, \;\rangle)$ by letting  $H_{\alpha} = H^{*, \Cop, \Op},\  H_{\beta} = H^{\Cop}$ and  $H_{\kappa} = K^*$ with the pairings defined as follows. The pairing $\pair_{\alpha,\beta}$ on $H_{\alpha} \otimes H_{\beta}$ is the canonical one, $\pair_{\beta,\kappa}$ on $H_{\beta} \otimes H_{\kappa}$ is given by $\langle h, p\rangle_{\beta,\kappa} = p \circ \phi_2(h)$, and  $\pair_{\kappa,\alpha}$ on $H_{\kappa} \otimes H_{\alpha}$ is given by $\langle p, q\rangle_{\kappa,\alpha} = p \circ \phi_1(q)$. 
  \item(Quasi-triangular) Let $(H,R)$ be a quasi-triangular involutory Hopf algebra $H$ equipped with an R-matrix $R \in H \otimes H$. Then the following linear map is a Hopf algebra morphism (c.f. Radford \cite{radford2011hopfalgebras}).
\[\begin{tikzpicture}
\draw (-5,.5) node (label) {$\pi:D(H) = D(H^{*,\op{cop}},H) \to H \quad\text{given by}$};
  \draw (0,0) node (i1) {};
  \draw (0,1) node (i2) {};
  \draw (1,.5) node [draw,rectangle] (R) {$R$};
  \draw (2,0) node (M) {$M$};
  \draw (3,.2) node (o) {};

  \draw[->] (R) to [out=20,in=20,looseness=1.5] (i2);
  \draw[->] (i1)--(M);
  \draw[->] (R) to [out=-20,in=150,looseness=1.5] (M);
  \draw[->] (M)--(o);
  \end{tikzpicture}\]
This map may be alternativly written (in the less diagrammatic notation of \cite{radford2011hopfalgebras}) as $\pi = M \circ (f_R \otimes \text{Id}) $ with $f_R(q):= (q \otimes \text{Id})R : H^{*, \Cop} \to H$. Now consider the following Hopf algebras and pairing derived from $H$.
  \[
  H_\alpha := (H^{\op{cop}})^{*,\op{cop}} \quad H_\beta := H^{\op{cop}} \quad\text{and}\quad H_\kappa := H^*
  \]
  \[
  (-):H_\alpha \otimes H_\beta \to k \qquad (a,b) = b(a)
  \]
We see that $D(H_\alpha^{\op{op}},H_\beta^{\op{cop}})$ is simply the usual double $D(H)$ of $H$ and $\pi$ is a Hopf algebra map $D(H_\alpha^{\op{op}},H_\beta^{\op{cop}}) \to H_\kappa^{*}$. 

We thus acquire a Hopf triplet $\mathcal{H}_H = (H_\alpha,H_\beta,H_\kappa,\langle-\rangle)$ by Corollary \ref{cor:map_version_of_triplet}, which we call the \emph{quasi-triangular triplet} associated to $(H,R)$. Note that the pairings $\langle-\rangle_{\alpha\beta}$ and $\langle-\rangle_{\beta\kappa}$ are simply the standard dual pairings, while the pairing $\langle-\rangle_{\kappa\alpha}$ agrees with $R$ interpreted as a map $H^* \otimes H^* \to k$. In terms of the notation of \cite{radford2011hopfalgebras}, the $\kappa\alpha$ pairing sends a pair $p,q \in H^*$ to $\langle p,q \rangle_{\kappa\alpha} = p(f_R(q)) = (q \otimes p)R$.
\end{enumerate}
\end{example}

\begin{remark}[Generalized Quasi-triangular] Hopf doublets $(H_\alpha, H_\beta, \langle - \rangle)$ more broadly can sometimes be equipped with a quasi-triangular structure, given by a generalized $R$-matrix $R \in H_\alpha^{*,\text{cop}} \otimes H_\beta$ satisfying suitably modified axioms (now involving the pairing $\langle - \rangle$).  Analogously, $\pi = M \circ (f_R \otimes \text{Id}) : D(H_\alpha, H_\beta) \to H_\beta$ is a Hopf algebra morphism, and this morphism can be used to construct a triplet as in Example \ref{ex:basic_examples_of_triplets}(b). However, we will not use any Hopf triplets arising in this way in the current paper.\end{remark}

\subsection{Trisection Diagrams} \label{subsec:trisections} Now we review the theory of trisections and trisection diagrams. Trisections diagrams for 4-manifolds were introduced in \cite{gk2016} (also see \cite{agk2018grouptrisections} and \cite{cgpc2018relativetrisections}).

Like Kirby diagrams, which are essentially handlebody diagrams, a trisection specifies some set of instructions for constructing a 4-manifold. However, trisections are more directly similar to Heegaard diagrams in the sense that the data for the 4-manifold is specified by a surface along with some curves on that surface.

\begin{definition}[Trisection Diagrams] \label{def:trisection_diagram} An \emph{oriented trisection diagram} $T$ is a triple $(\Sigma,\alpha,\beta,\kappa)$ consisting of the following data.
\begin{itemize}
	\item[(a)] (Surface) A closed, oriented 2-manifold $\Sigma$ of genus $g$.
	\item[(b)] (Curves) Three sets of $g$ non-separating, embedded curves $\{\alpha_i\},\{\beta_i\}$ and $\{\kappa_i\}$ on $\Sigma$ such that:
	\begin{itemize}
		\item[(i)] All curves from a single set are disjoint, i.e. $\alpha_i \cap \alpha_j$ is empty when $i \neq j$.
		\item[(ii)] Any pair of the three curve sets form a Heegaard diagram for $\#_{i=1}^k S^1 \times S^2$, for some $k$ independent of which two curve sets are used. By convention, we say that $\#_{i=1}^k S^1 \times S^2 = S^3$ in the case where $k = 0$.
	\end{itemize}
\end{itemize}
\end{definition}

\begin{figure}[h]

\centering
\includegraphics[width=\textwidth]{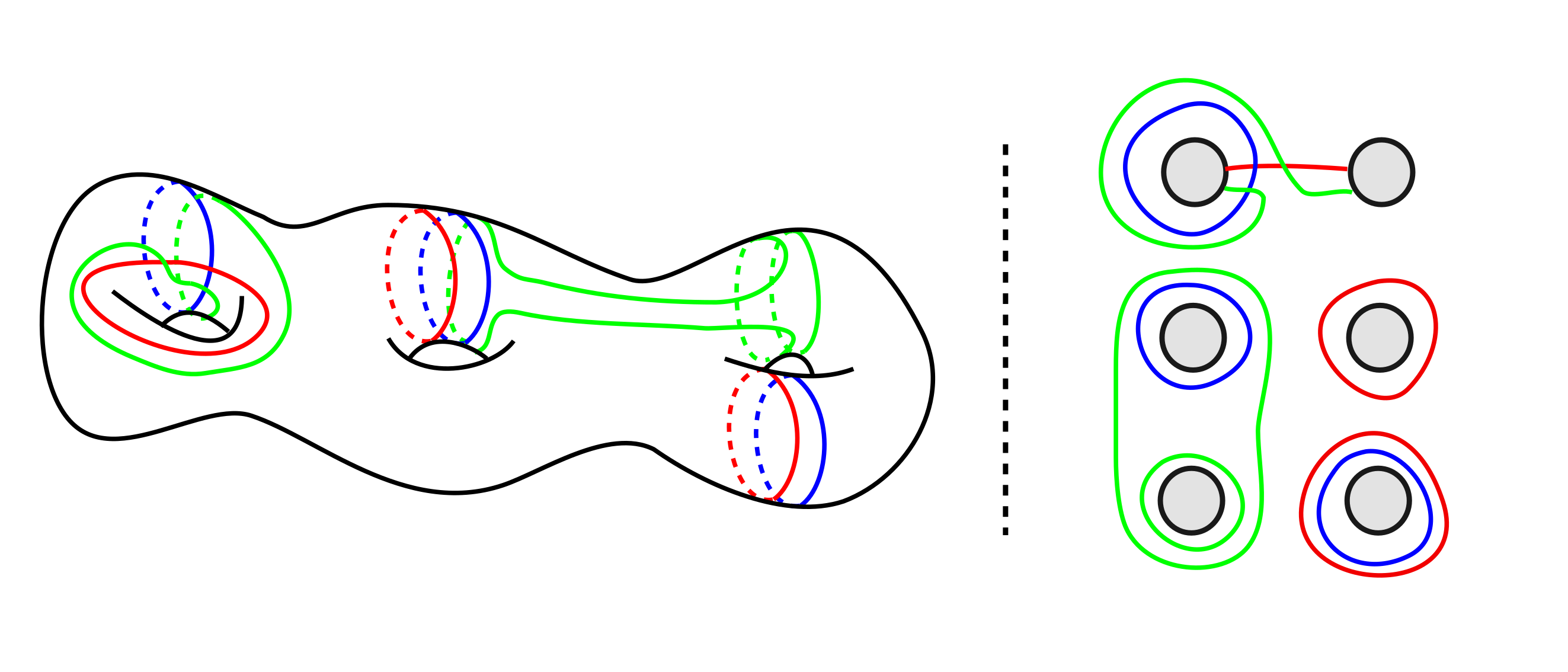}
\caption{A simple trisection diagram. On the left, we have a visualization of the trisection on a surface in $\R^3$ with the relevant curves. On the right, we have a Heegaard diagram type visualization of the same trisection.}
\label{fig:example_trisection_diagram}
\end{figure}
\newpage
\begin{definition}[Basic Constructions] \label{def:basic_trisection_constructions} We adopt the following terminology for the most basic topological operations on trisections.
\begin{itemize}
	\item[(a)] (Diffeomorphism) A \emph{diffeomorphism} $\varphi:T \simeq T'$ of trisections is a diffeomorphism $\varphi:\Sigma \simeq \Sigma'$ of the underlying surfaces that intertwines the curve sets, i.e. $\varphi(\alpha) = \alpha'$, $\varphi(\beta) = \beta'$ and $\varphi(\kappa) = \kappa'$.
	\item[(b)] (Isotopy) An \emph{isotopy} between trisections $T$ and $T'$ with the same underlying surface $\Sigma = \Sigma'$ is simply an isotopy of the corresponding curve sets $\alpha^t, \beta^t$ and $\kappa^t$ so that $\alpha^0 = \alpha$, $\alpha^1 = \alpha'$, $\beta^0 = \beta$ etc.
	\item[(c)] (Connect Sum) A \emph{connect sum} $T \# T'$ of two trisections $T$ and $T'$ (along disks $D \subset \Sigma$ and $D' \subset \Sigma'$ disjoint from the $\alpha,\beta$ and $\kappa$ curves) is defined as follows. The surface of $T \# T'$ is given by a connect sum $\Sigma \# \Sigma'$ along the boundary created by removing $D$ from $\Sigma$ and $D'$ from $\Sigma'$. Note that the diffeomorphism type of the result of this operation depends on the choice of $D$ and $D'$.
	\item[(d)] (Orientation Reversal) The \emph{orientation reversed} trisection $\overline{T}$ of a trisection $T$ has underlying surface given by $\overline{\Sigma}$, i.e.~$\Sigma$ with the opposite orientation, and the same curve sets $\alpha$,$\beta$ and $\kappa$.
\end{itemize}
\end{definition}

General isotopies, even of immersed curves in surfaces, can be quite complicated. In order to prove invariance of our invariants in \S \ref{sec:trisection_kuperberg_invt}, we need to work with a more restricted, combinatorial class of isotopies.

\begin{definition}[Two/Three-Point Moves] \label{def:two_three_point_moves} Let $T$ be a trisection. 

\vspace{5pt}

A \emph{two-point move} on $T$ is a new trisection $T'$ acquired as so. First, identify a disk $D \subset \Sigma$ possessing a diffeomorphism of pairs
\[
(D,D \cap (\alpha \cup \beta \cup \kappa)) \simeq (D_+,\gamma_+ \cup \eta_+) \text{ or }(D_-,\gamma_- \cup \eta_-)
\]
Then replace $(D_+,\gamma_+ \cup \eta_+) \subset \Sigma$ with $(D_-,\gamma_- \cup \eta_-)$ in the $+$ case, or alternatively replace $(D_-,\gamma_- \cup \eta_-) \subset \Sigma$ with $(D_+,\gamma_+ \cup \eta_+)$ in the $-$ case. Here $(D_+,\gamma_+ \cup \eta_+)$ and $(D_-,\gamma_- \cup \eta_-)$ are given by the following pictures.
\[
\includegraphics[width=.6\textwidth]{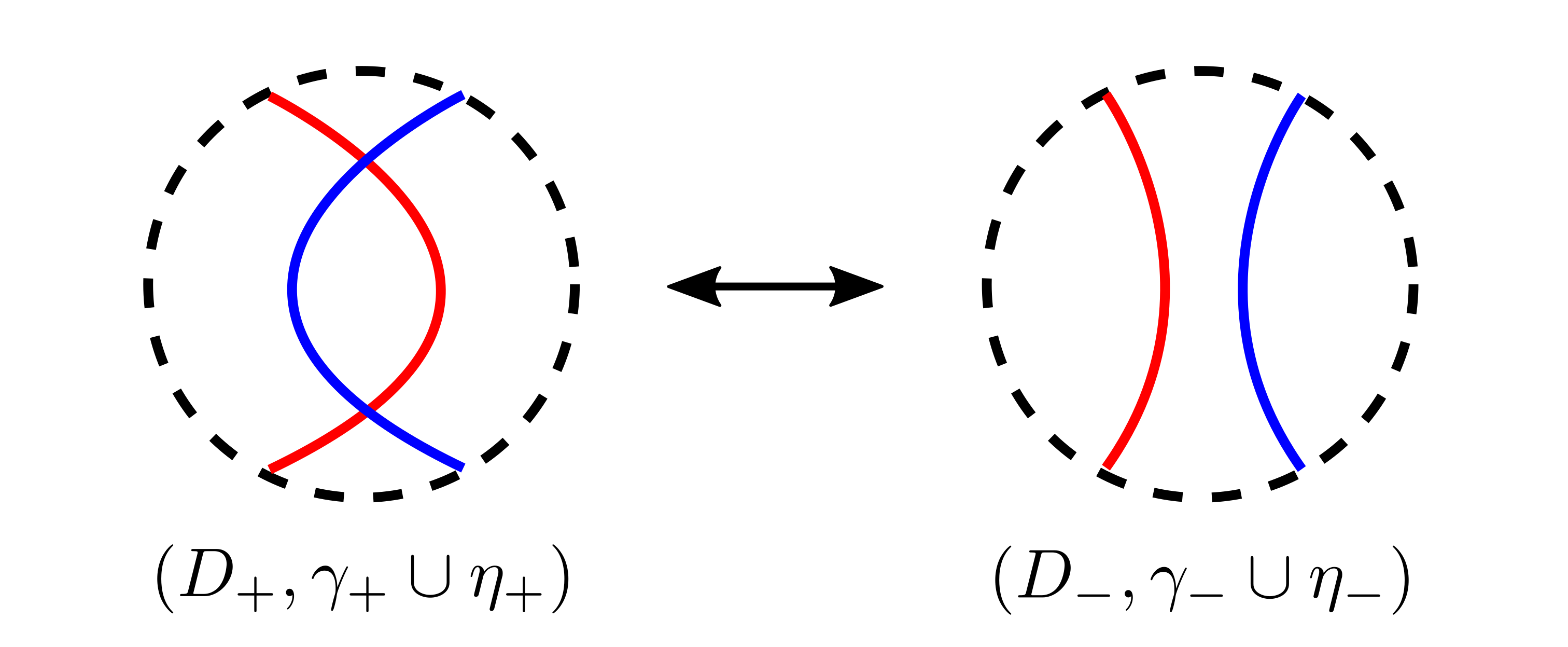}
\]

\vspace{5pt}

A \emph{three-point move} on $T$ is a new trisection $T'$ acquired as so. First identify a disk $D \subset \Sigma$ possessing a diffeomorphism of pairs
\[
(D,D \cap (\alpha \cup \beta \cup \kappa)) \simeq (D_+,\gamma_+ \cup \eta_+ \cup \xi_+) \text{ or }(D_-,\gamma_- \cup \eta_- \cup \xi_-)
\]
Then replace $(D_+,\gamma_+ \cup \eta_+ \cup \xi_+) \subset \Sigma$ with $(D_-,\gamma_- \cup \eta_- \cup \xi_-)$ in the $+$ case, or alternatively replace $(D_-,\gamma_- \cup \eta_- \cup \xi_-) \subset \Sigma$ with $(D_+,\gamma_+ \cup \eta_+ \cup \xi_+)$ in the $-$ case. Here $(D_+,\gamma_+ \cup \eta_+ \cup \xi_+)$ and $(D_-,\gamma_- \cup \eta_- \cup \xi_-)$ are given by the following pictures.
\[
\includegraphics[width=.6\textwidth]{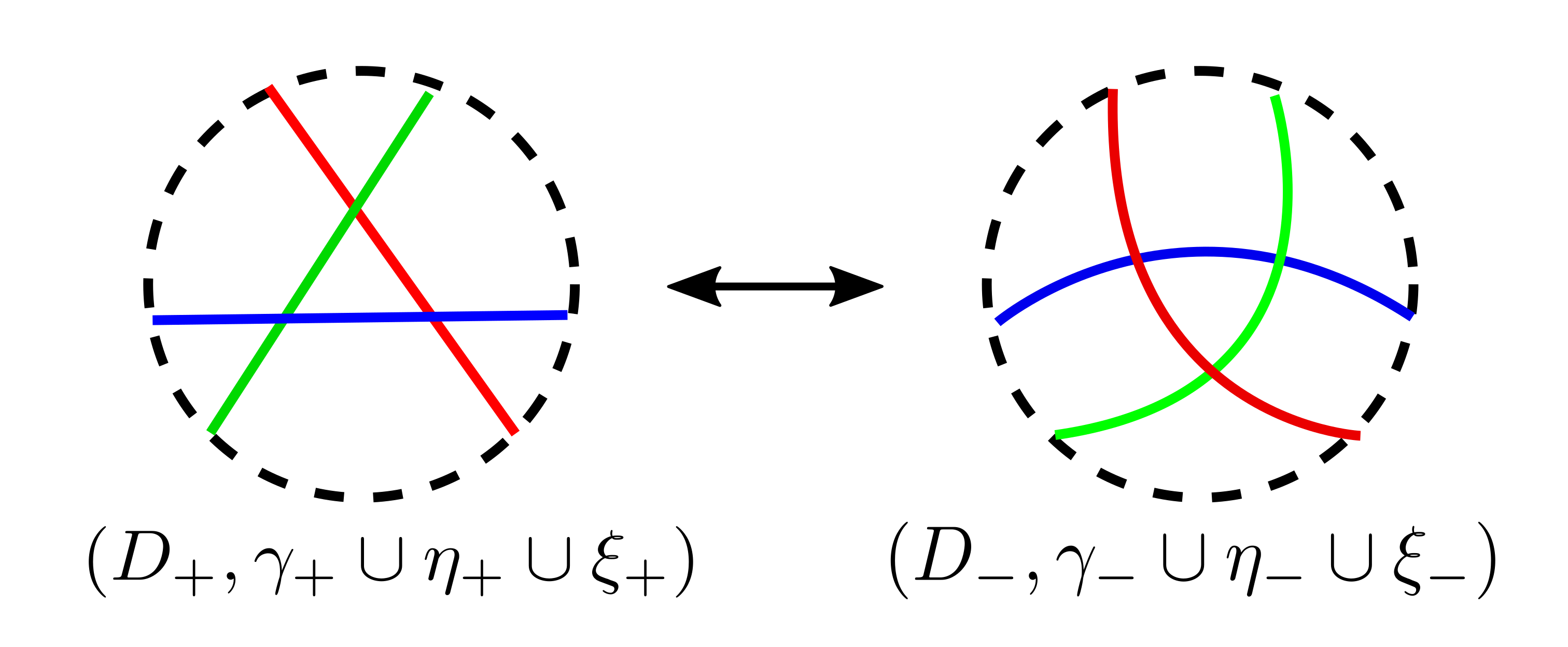}
\]

\end{definition}

Note that the two and three point moves above are analogues of Redemeister 2 and 3 moves from knot theory. As in knot theory, we are essentially allowed to reduce to the case of such isotopies by the following lemma.

\begin{lemma} \label{lem:curve_isotopy} Let $\Gamma$ be a closed $1$-manifold and let $\Sigma$ be a closed $2$-manifold. Let $\iota_0,\iota_1:\Gamma \to \Sigma$ be a pair of homotopic immersions of $\Gamma$ such that
\begin{itemize}
  \item[(a)] Each component $C$ of $\Gamma$ is embedded by $\iota_i$.
  \item[(b)] The components of $\iota_i(\Gamma)$ only intersect transversely at double points.
\end{itemize} 
Then $\iota_0$ and $\iota_1$ are diffeomorphic after sequence of two-point and three-point moves.
\end{lemma}

\begin{proof}  First assume that $\iota_0(\Gamma)$ and $\iota_0(\Gamma)$ both have the minimal possible number of transverse double points in their homotopy class. Any two such minimal immersions are ambiently isotopic after a sequence of three-point moves (see p.~231-232 and Lemma 3.4 of \cite{patterson2002loops}). 

So it suffices to show that $\iota_0$ and $\iota_1$ can be isotoped to minimal immersions $\iota_0'$ and $\iota_1'$ satisfying (a) and (b) using two-point moves. This is implied immediately by Lemma 3.1 of \cite{hs1985}, which states that if $\iota_0$ (for instance) does not minimize intersections, then there is an inner-most 2-gon (i.e.~a copy of $(D_+,\gamma_+ \cup \eta_+)$ as above) on which one can perform a two-point move to decrease the self-intersections by $1$.
\end{proof}

By applying Lemma \ref{lem:curve_isotopy} to the $\alpha$, $\beta$ and $\gamma$ curves of trisection diagrams, we acquire the following corollary.

\begin{corollary} \label{lem:two_three_point_moves} Let $T$ and $T'$ be isotopic trisection diagrams. Then $T$ and $T'$ are diffeomorphic after a sequence of two-point and three-point moves.
\end{corollary}

There are three types of operations beyond diffeomorphism and isotopy that emerge in the study of trisection diagrams as presentations of $4$-manifolds. These are the following.

\begin{definition} \label{def:trisection_moves} (Trisection Moves) Let $T = (\Sigma,\alpha,\beta,\kappa)$ be a trisection diagram. We now describe three special operations for producing a new trisection diagram $T \rightsquigarrow T'$, collectively called \emph{trisection moves}.
\begin{enumerate}
	\item[(a)] (Handle Slides) Given two distinct $\alpha$ curves $\alpha_0$ and $\alpha_1$ along with an arc $\gamma$ connecting $\alpha_0$ to $\alpha_1$, one may alter $T$ to a new trisection $T'$ by replacing $\alpha_0$ by the \emph{handle-slide} $\alpha_0 \#_\gamma \alpha_1$ of $\alpha_0$ over $\alpha_1$ via $\gamma$. Here $\alpha_0 \#_\gamma \alpha_1$ is defined as so. Let $U \subset \Sigma$ be a ribbon neighborhood of $\alpha_0 \sqcup \gamma \sqcup \alpha_1$. The boundary $\partial U$ then decomposes into three closed curves: a normal push-off of $\alpha_0$, a normal push-off of $\alpha_1$ and a third piece, which is precisely the handle-slide $\alpha_0 \#_\gamma \alpha_1$.
	\item[(b)] (Stabilization) A \emph{stabilization} $T'$ of $T$ is the trisection given by the connect sum $T \# T_{\op{st}}$ where $T_{\op{st}}$ is the genus $3$ \emph{stabilized sphere trisection} in Figure \ref{fig:stabilized_sphere_trisection}.
	\item[(c)] (De-Stabilization) A \emph{destabilization} $T'$ of a trisection $T$ which is diffeomorphic $T \simeq T' \# T_{\op{st}}$ with a stabilization is simply the summand $T'$.
\end{enumerate} 
\end{definition}

\begin{figure}[h]
\centering
\includegraphics[width=\textwidth]{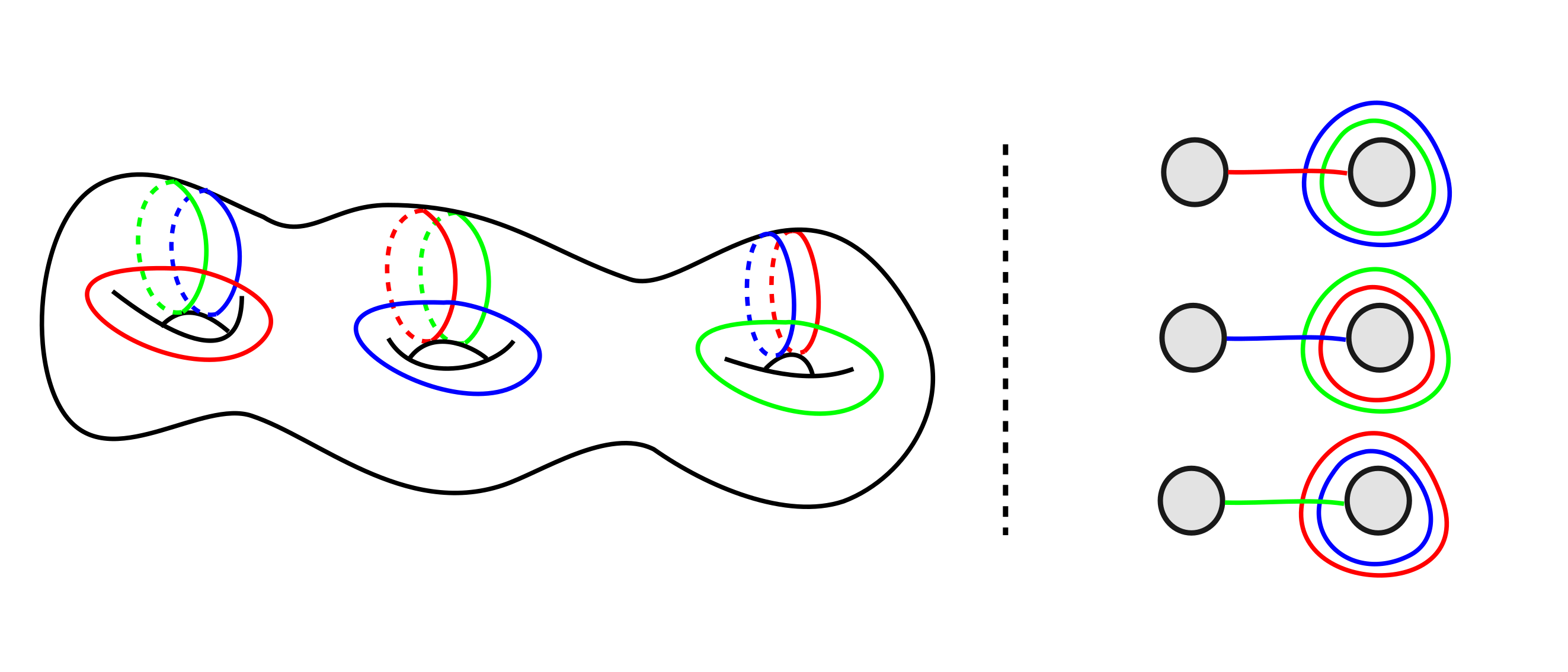}
\caption{The standard stabilized sphere trisection. As in Figure \ref{fig:example_trisection_diagram}, we include an embedded surface depiction and a Heegaard diagram type depiction.}
\label{fig:stabilized_sphere_trisection}
\end{figure}

The significance of trisection diagrams comes from their usefulness for specifying a particular diffeomorphism class of 4-manifolds, via the following construction.

\begin{definition} \label{def:4_manifold_of_trisection} (4-Manifold of a Trisection) Let $T = (\Sigma,\alpha,\beta,\kappa)$ be a trisection diagram. The \emph{4-manifold $X(T)$ of the trisection $T$} is the oriented 4-manifold constructed by the procedure below. 

To construct $X(T)$, we proceed as follows. Let $\Sigma \times D^2$ be the surface $\Sigma$ thickened by a 2-disk. For each $* \in \{\alpha,\beta,\kappa\}$, let $H_*$ denote a genus $g$ handlebody (i.e., the boundary sum $H_* \simeq \natural_{i=1}^g S^1 \times D^2$) and let $H_* \times D^1$ denote that handlebody thickened by a 1-disk. Divide the boundary $\partial D^2$ into a union $D^1_\alpha \cup D^1_\beta \cup D^1_\kappa$ of three cyclically ordered intervals $D^1_*$ meeting at their boundaries (with an explicit oriented diffeomorphism given by $\iota_*:D^1 \simeq D_*$). 

For each $* \in \alpha,\beta,\kappa$, the $*$-curves in $\Sigma$ determine a unique oriented diffeomorphism $\varphi_*:\partial H_* \simeq \Sigma$ up to isotopy sending the belt spheres $\bigsqcup_{i=1}^g \{1/2\} \times S^1\subset \natural_{i=1}^g S^1 \times D^2 \simeq H_*$ of $H_*$ to the $*$-curves. We thus get a unique oriented embedding $\varphi_\alpha \times \iota_\alpha:\partial H_\alpha \times D^1 \to \Sigma \times D^1_*$. If we glue $\Sigma \times D^2$ to $H_\alpha \times D^1$ via each of the maps $\varphi_\alpha$, the boundary $\partial W(T)$ of the resulting glued space $W(T)$ decomposes into 3 connected pieces $\partial_{**} W(T)$ with $** \in \{\alpha\beta,\beta\kappa,\kappa\alpha\}$, where each piece admits an oriented diffeomorphism $\phi_{**}:\partial_{**}W(T) \simeq \partial(\natural_{i=1}^k S^1 \times D^3)$. 

To get $X(T)$, we simply glue $W(T)$ and three copies of $\natural_{i=1}^k S^1 \times D^3$ along their boundaries via the maps $\phi_{**}$. The orientation of $X(T)$ is induced by the product orientation on $\Sigma \times D^2$, where we take the standard orientation on $D^2$.

\vspace{5pt}

A 4-manifold $X$ along with a diffeomorphism $X \simeq X(T)$ is said to be \emph{trisected}.
\end{definition}

\begin{theorem} \cite{gk2016} The 4-manifold $X(T)$ of a trisection diagram $T$ is independent of the choices made in Definition \ref{def:4_manifold_of_trisection} up to oriented diffeomorphism.
\end{theorem}

\begin{lemma} The construction $T \to X(T)$ has the following naturality properties with respect to the operations in Definition \ref{def:basic_trisection_constructions} and Definition \ref{def:trisection_moves}. Let $T$ and $T'$ be a pair of trisection diagrams.
\begin{enumerate}
	\item[(a)] (Diffeomorphism) If $T$ and $T'$ are oriented diffeomorphic, then $X(T) \simeq X(T')$.
	\item[(b)] (Trisection Moves) If $T$ and $T'$ are diffeomorphic after a sequence of trisection moves and isotopies, then $X(T) \simeq X(T')$.
	\item[(c)] (Connect Sum) There is an oriented diffeomorphism $X(T \# T') \simeq X(T) \# X(T')$.
	\item[(d)] (Orientation Reversal) There is an oriented diffeomorphism $\overline{X(T)} \simeq X(\overline{T})$.
\end{enumerate}
\end{lemma}

The fundamental theorem of trisections is that any closed, oriented 4-manifold can be trisected and that this trisection is unique modulo the trisection moves.

\begin{theorem} \cite{gk2016} Let $X$ be a closed oriented $4$-manifold. Then:
\begin{itemize}
\item[(a)] (Existence) $X$ admits a trisection, i.e.~there exists a trisection diagram $T$ and an oriented diffeomorphism $\varphi:X \simeq X(T)$.
\item[(b)] (Uniqueness) Any two trisection diagrams $T$ and $T'$ of $X$ are oriented diffeomorphic after a series of trisection moves and isotopies are applied to $T$.
\end{itemize}
\end{theorem}

\section{4-Manifold Invariants} \label{sec:trisection_kuperberg_invt} In this section, we describe the construction of our family of 4-manifold invariants and demonstrate its basic properties. In \S \ref{subsec:trisection_bracket}, we construct an auxiliary (non-invariant) number called the trisection bracket, and prove its essential properties. In \S \ref{subsec:main_definition_and_properties}, we apply the results of the previous section to quickly define the $4$-manifold invariants of interest.

\subsection{Trisection Bracket} \label{subsec:trisection_bracket} We begin this section by introducing the following bracket.

\begin{definition} \label{def:trisection_bracket} (Trisection Bracket) Let $\mathcal{H} = (H_\alpha,H_\beta,H_\kappa,\langle - \rangle)$ be a Hopf triplet over a field $k$ of characteristic zero and $T = (\Sigma,\alpha,\beta,\kappa)$ be a trisection diagram. 

The \emph{trisection bracket} $\langle T\rangle_{\mathcal{H}} \in k$ is defined to be the scalar specified by a particular tensor diagram, which is constructed according to the following procedure.

\begin{itemize}
	\item[(a)] Begin by setting $\langle T\rangle_{\mathcal{H}}$ to be the empty tensor diagram. Fix arbitrary orientations $o_\alpha,o_\beta$ and $o_\kappa$ of the $\alpha,\beta$ and $\kappa$ curves of the trisection $T$.
	\item[(b)] For each $\gamma \in \{\alpha,\beta,\kappa\}$ and each $\gamma$-curve $\gamma_i$, add a comultiplication node to the diagram $\langle T\rangle_{\mathcal{H}}$ as so. Let $m = m^\gamma_i$ denote the number of intersections of $\gamma_i$ with the other curves on $T$, i.e. $m = |\gamma_i \cap (\alpha \sqcup \beta \sqcup \kappa - \gamma_i)|$. Let $\{\iota^\gamma_{i,j}\}_{j=1}^m$ denote the sequence of intersection points between $\gamma$ and the other curves. We order the sequence $\{\iota^\gamma_{i,j}\}_{j=1}^m$ according to the cyclic ordering induced by the orientation $o_\gamma$ on $\gamma_i$.

	In terms of the above notation, we include a comultiplication $C_\gamma \rightarrow \Delta_\gamma \rightrightarrows$ in $\langle T\rangle_{\mathcal{H}}$ from the Hopf algebra $H_\gamma$ with $1$ input (from a cotrace $C_\gamma$) and $m^\gamma_i$ outputs. We also label the outputs by the intersection points $\iota^\gamma_{i,1}, \iota^\gamma_{i,2},\dots$ in counter-clockwise cyclic order. In tensor diagram notation, we are performing the following move.\[
	\includegraphics[width=.2\linewidth]{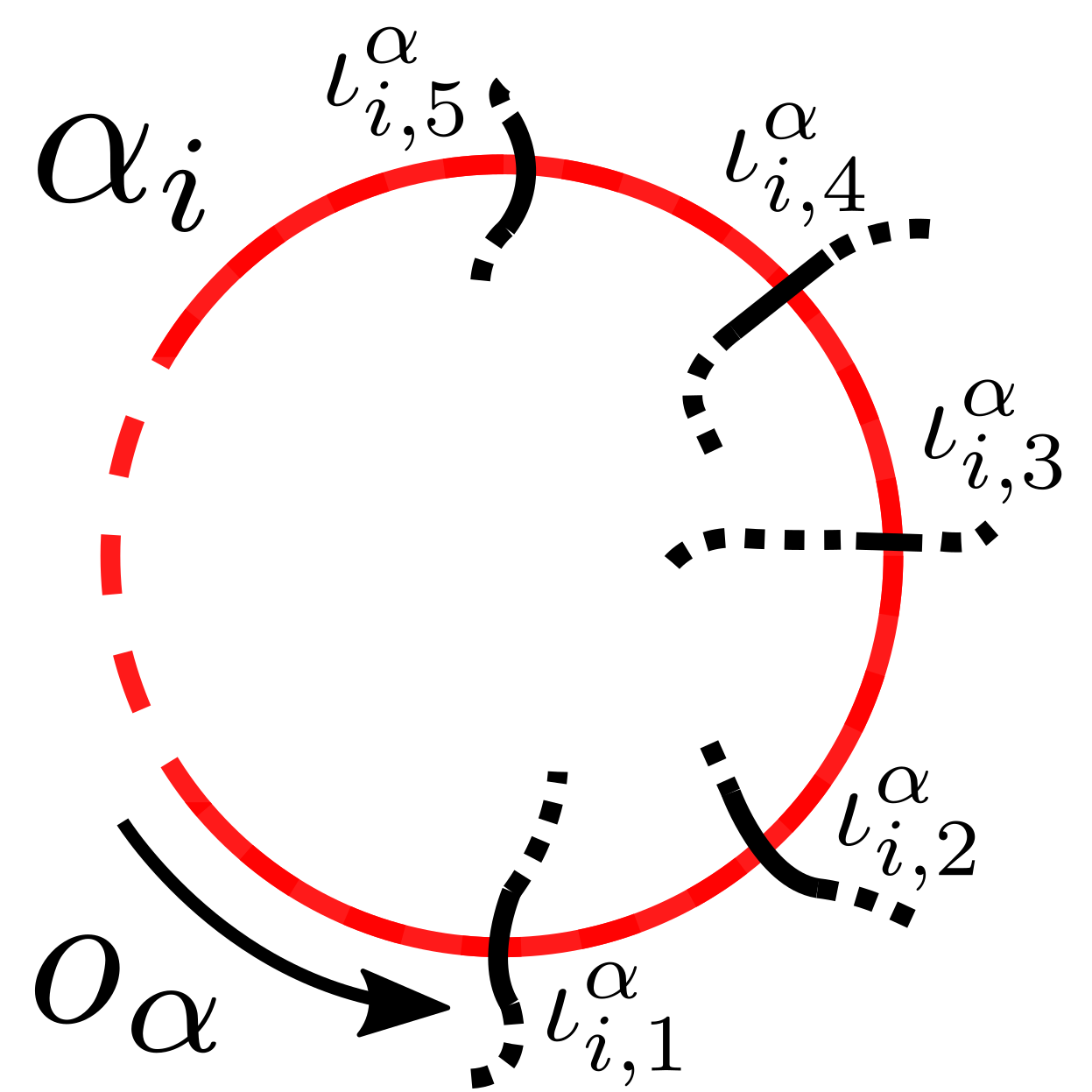}
\begin{tikzpicture}

  \draw (0,-1.2) node (space) { };
  \draw (-2.8,0) node (space) { };
	
  \draw (-2.2,0) node (=>) {$\mapsto$};

  \draw (0,0) node (D1) {$\Delta_\alpha$};
  \draw (-1,0) node (C1) {$C_\alpha$};
  \draw (0,-1) node (i1) {$\iota^\alpha_{i,1}$};
  \draw (.8,-.8) node (i2) {$\iota^\alpha_{i,2}$};
  \draw (1,0) node (i3) {$\iota^\alpha_{i,3}$};
  \draw (.8,.8) node (i4) {$\iota^\alpha_{i,4}$};
  \draw (0,1) node (i5) {$\iota^\alpha_{i,5}$};
  \draw (-.6,.6) node (dots1) {$\dots$};

 \draw[->] (C1)--(D1);

 \draw[<-] (i1)--(D1);
 \draw[<-] (i2)--(D1);
 \draw[<-] (i3)--(D1);
 \draw[<-] (i4)--(D1);
 \draw[<-] (i5)--(D1);

 \end{tikzpicture}\]

 \item[(c)] On each pair of outgoing edges $\Delta_\gamma \rightarrow \iota^\gamma_{i,a}$ and $\Delta_\eta \rightarrow \iota^\eta_{j,b}$ labelled by the same geometric intersection $\iota := \iota^\gamma_{i,a} = \iota^\eta_{j,b}$, we perform the following contraction within $\langle T\rangle_{\mathcal{H}}$.

 First assign a sign, denoted by $\op{sgn}(\iota) \in \{+,-\}$, to the intersection $\iota$ according to the following rule. Let $\gamma,\eta \in \{\alpha,\beta,\kappa\}$ be the type of the intersecting curves $\gamma_i$ and $\eta_j$ as above. Relabel the curves $\gamma_i$ and $\eta_j$ in the pair so that $\gamma \prec \eta$ with respect to the cyclic ordering $\alpha \prec \beta \prec \kappa \prec \alpha$. The orientations $o_\gamma$ and $o_\eta$ induce orientations of the tangent spaces $T_\iota\gamma_i$ and $T_\iota\eta_j$ to the curves at $\iota$. This in turn induces an orientation $o_\gamma \otimes o_\eta$ on $T_\iota\Sigma = T_\iota \gamma_i \oplus T_\iota \eta_j$. On the other hand, $T_\iota\Sigma$ is oriented by a background orientation $o_\Sigma$, since $\Sigma$ is an oriented surface. The sign of $\iota$ is thus defined by the relation $o_\gamma \otimes o_\eta = \op{sgn}(\iota) \cdot o_\Sigma$. 

 Pictorally, this amounts to the following sign assignments when the plane is given the standard orientation.
 \[
 \includegraphics[width=.4\linewidth]{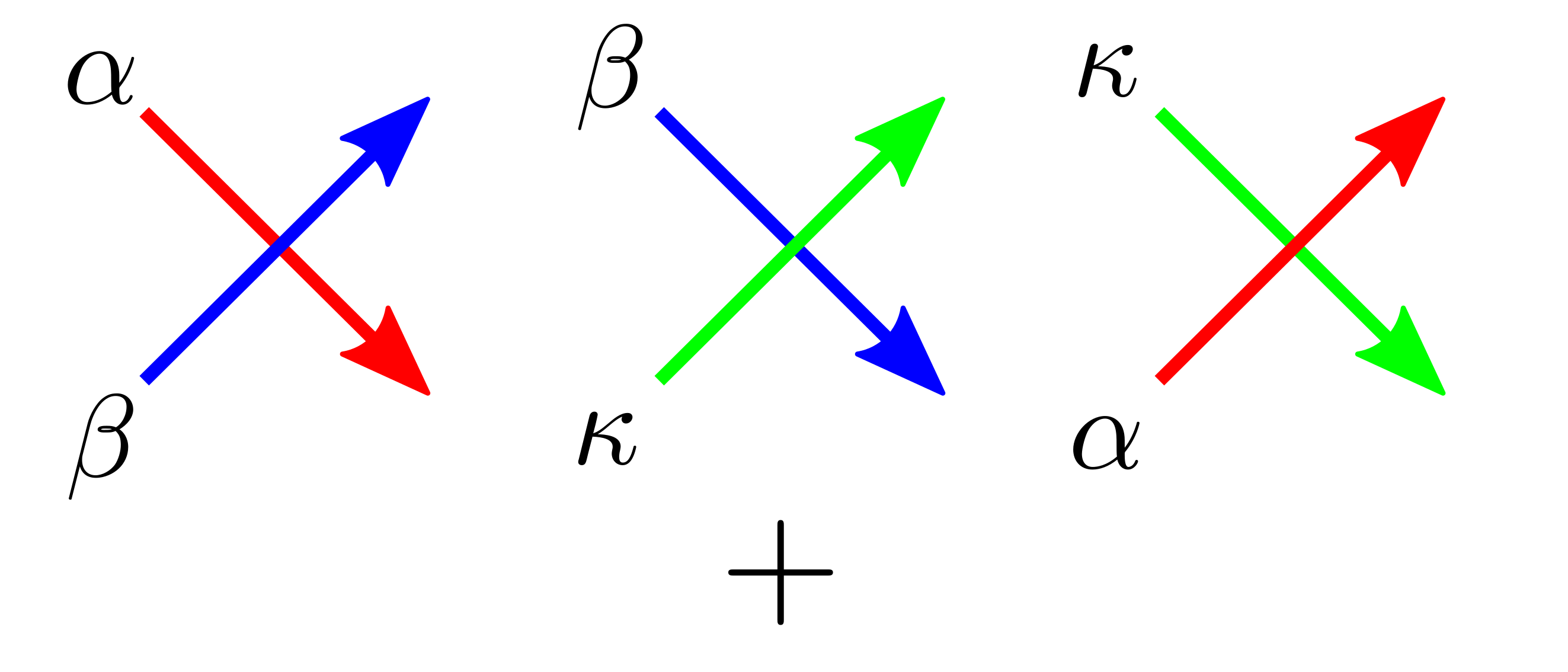}\hspace{40pt}
 \includegraphics[width=.4\linewidth]{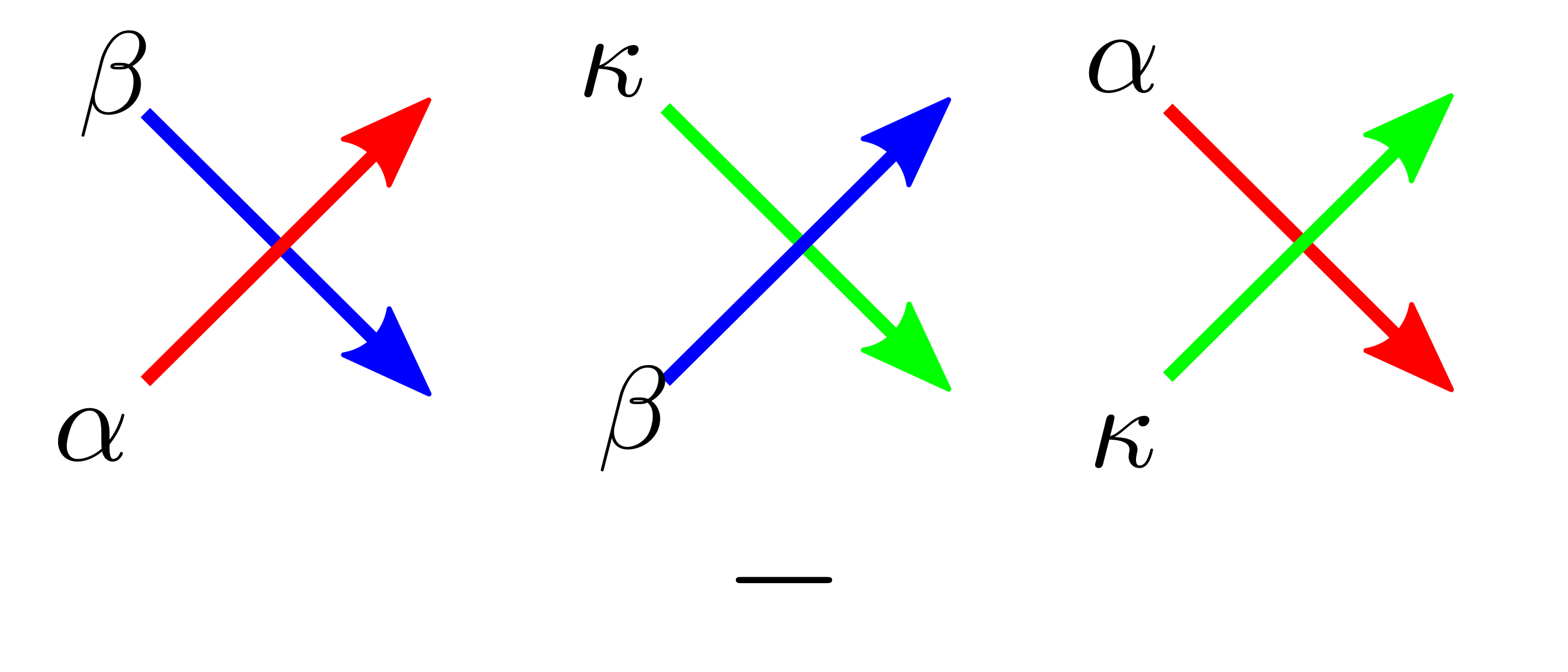}
 \]Finally, we perform the following subtitution. If $\op{sgn}(\iota)$ is positive, we pair the out edges $\rightarrow \iota^\gamma_{i,a}$ and $\rightarrow^\eta_{j,b}$ via a pairing node $\rightarrow \langle -\rangle_{\gamma\eta} \leftarrow$. If $\op{sgn}(\iota)$ is negative, we pair the out edges $\rightarrow \iota^\gamma_{i,a}$ and $\rightarrow^\eta_{j,b}$ via a pairing node $\rightarrow S_\gamma \rightarrow \langle -\rangle_{\gamma\eta} \leftarrow$. Pictorally (with the bullet notation) this can be written as:
 \[\begin{tikzpicture}
  \draw (0,0) node (D1) {$\Delta_\gamma$};
  \draw (3,0) node (D2) {$\Delta_\eta$};
  \draw (1,0) node (i1) {$\iota^\gamma_{i,a}$};
  \draw (2,0) node (i2) {$\iota^\eta_{j,b}$};
  \draw (-.6,0) node (dots1) {$\dots$};
  \draw (3.6,0) node (dots2) {$\dots$};

 \draw[->] (D1)--(i1);
 \draw[->] (D2)--(i2);

 \draw (4.5,0) node (=>) {$\mapsto$};

  \draw (6,0) node (D1) {$\Delta_\gamma$};
  \draw (9,0) node (D2) {$\Delta_\eta$};
  \draw (7.5,0) node (P1) {$\bullet$};
  \draw (5.4,0) node (dots1) {$\dots$};
  \draw (9.6,0) node (dots2) {$\dots$}; 
  \draw (11.5,0) node (label1) {if $\op{sgn}(\iota) = +$}; 

  \draw[->] (D1)--(P1);
  \draw[->] (D2)--(P1);

 \end{tikzpicture}\]\[\begin{tikzpicture}
  \draw (0,0) node (D1) {$\Delta_\gamma$};
  \draw (3,0) node (D2) {$\Delta_\eta$};
  \draw (1,0) node (i1) {$\iota^{\gamma}_{i,a}$};
  \draw (2,0) node (i2) {$\iota^{\eta}_{j,b}$};
  \draw (-.6,0) node (dots1) {$\dots$};
  \draw (3.6,0) node (dots2) {$\dots$};

 \draw[->] (D1)--(i1);
 \draw[->] (D2)--(i2);

 \draw (4.5,0) node (=>) {$\mapsto$};

  \draw (6,0) node (D3) {$\Delta_\gamma$};
  \draw (9,0) node (D4) {$\Delta_\eta$};
  \draw (8,0) node (P3) {$\bullet$};
  \draw (7,0) node (S3) {$S_\gamma$};
  \draw (5.4,0) node (dots3) {$\dots$};
  \draw (9.6,0) node (dots4) {$\dots$}; 
  \draw (11.5,0) node (label1) {if $\op{sgn}(\iota) = -$};

  \draw[->] (D3)--(S3);
  \draw[->] (S3)--(P3);
  \draw[->] (D4)--(P3);

 \end{tikzpicture}\]
\end{itemize}
The tensor diagram $\langle T\rangle_{\mathcal{H}}$ acquired after performing steps (a)-(c) above will have no input or output edges by construction, and will therefore define a scalar as claimed.
\end{definition}

The following lemma demonstrates that the bracket $\langle T\rangle_{\mathcal{H}}$ depends only on $T$ and $H$, and not on the extraneous choices made in the definition.

\begin{lemma} \label{lem:bracket_well_defined} The bracket $\langle T\rangle_{\mathcal{H}} \in k$ of a trisection $T$ is independent on the orientations of the $\alpha,\beta$ and $\kappa$ curves chosen in step (a) of Definition \ref{def:trisection_bracket}.
\end{lemma}

\begin{proof} Let $\gamma \in \{\alpha,\beta,\kappa\}$ and let $\gamma_i$ be a $\gamma$-curve. It suffices to show that $\langle T\rangle_{\mathcal{H}}$ is invariant under changes of choices of orientation for $\gamma_i$. 

Thus let $T_+$ and $T_-$ be the trisections with curve orientations chosen to match on all $\alpha,\beta$ and $\kappa$ curves except at $\gamma_i$, where the orientations are opposite. Label the $m$ intersections of $\gamma_i$ with other curves as $\iota_1,\dots,\iota_m$ in the cyclic order determined by the $T_+$ orientation. Let $\sigma:\{1,\dots,m\} \to \{0,1\}$ be defined by $\sigma_i = 0$ if $\op{sgn}(\iota_i) = +$ and $\sigma_i = 1$ if $\op{sgn}(\iota_i) = -$. Then the tensor diagrams in the two cases can be written as \[
\raisebox{3pt}{\includegraphics[width=.25\linewidth]{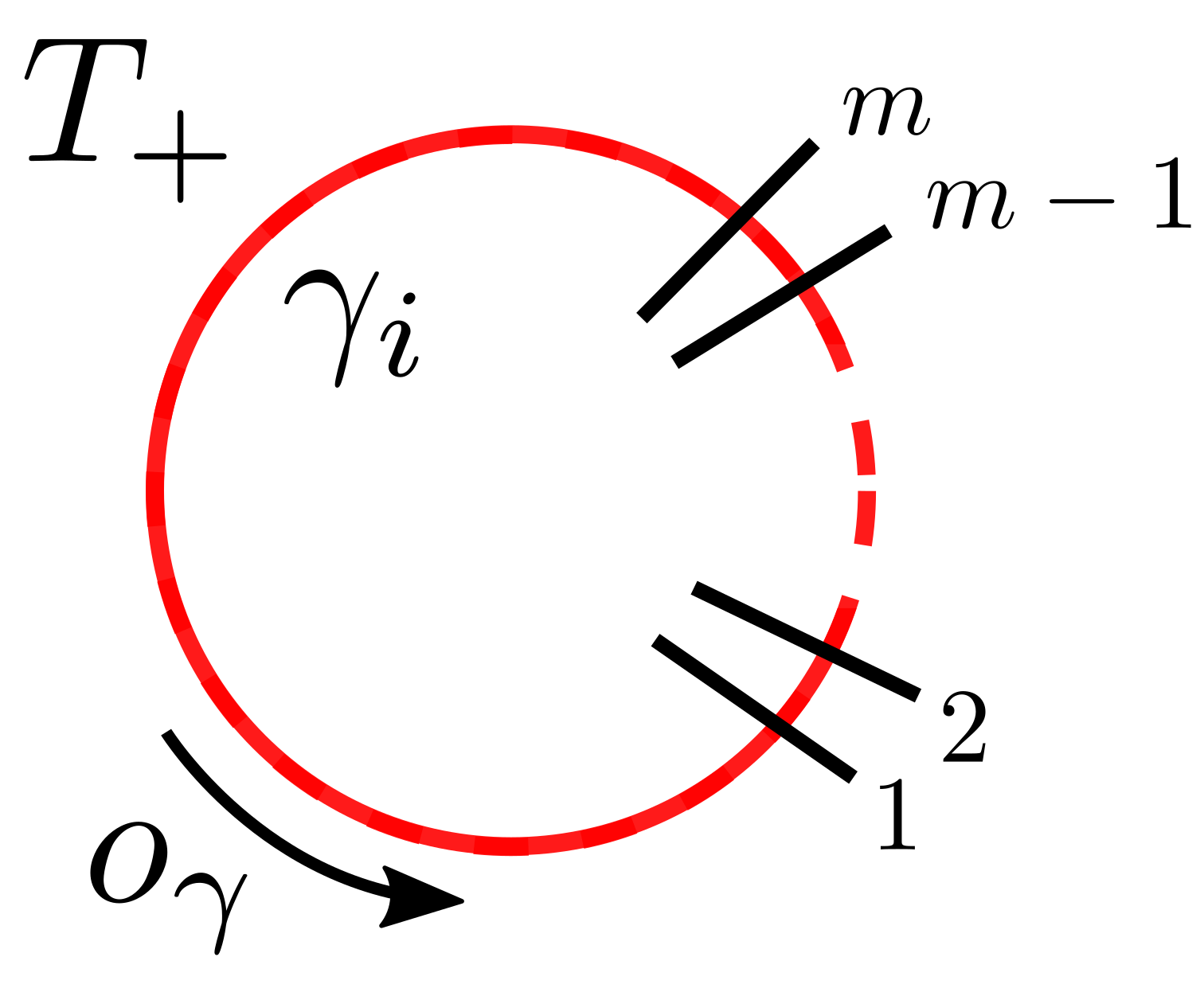}}
\begin{tikzpicture}
  \draw (-1.4,0) node (Sp1) { };
  \draw (-.9,0) node (=>) {$\mapsto$};
  \draw (0,2) node (Sp2) { };

  \draw (0,0) node (C) {$C_\gamma$};
  \draw (1,0) node (D) {$\Delta_\gamma$};
  \draw (2.5,0) node (dots) {$\dots$};
  \node at (4.3,0) [draw,rectangle,dashed] (Id) {$\op{Id}^{\otimes m}$};
  \node at (6.2,0) [draw,circle,dashed] (A) {$A$};

  \draw (2.2,-1.2) node (S1) {$S_\gamma^{\sigma_1}$};
  \draw (2.2,-.6) node (S2) {$S_\gamma^{\sigma_2}$};
  \draw (2.2,.6) node (S3) {$S_\gamma^{\sigma_{m-1}}$};
  \draw (2.2,1.2) node (S4) {$S_\gamma^{\sigma_m}$};
  \draw (3.4,-1.2) node (P1) {$\bullet$};
  \draw (3.4,-.6) node (P2) {$\bullet$};
  \draw (3.4,.6) node (P3) {$\bullet$};
  \draw (3.4,1.2) node (P4) {$\bullet$};

  \draw[->] (C)--(D);

  \draw[->] (D)  to [out=-90,in=180,looseness=1.5] (S1);
  \draw[->] (D)  to [out=-80,in=180,looseness=1.5] (S2);
  \draw [->] (D)--(1.6,.2);
  \draw [->] (D)--(1.6,0);
  \draw [->] (D)--(1.6,-.2);
  \draw[->] (D)  to [out=80,in=180,looseness=1.5] (S3);
  \draw[->] (D)  to [out=90,in=180,looseness=1.5] (S4);

  \draw[->] (S1)--(P1);
  \draw[->] (S2)--(P2);
  \draw[->] (S3)--(P3);
  \draw[->] (S4)--(P4);

  \draw[->] (P1)  to [in=-90,out=0,looseness=1.5] (Id);
  \draw[->] (P2)  to [in=-120,out=0,looseness=1.5] (Id);
  \draw [->] (3.4,.2)--(Id);
  \draw [->] (3.4,0)--(Id);
  \draw [->] (3.4,-.2)--(Id);
  \draw[->] (P3)  to [in=120,out=0,looseness=1.5] (Id);
  \draw[->] (P4)  to [in=90,out=0,looseness=1.5] (Id);

  \draw[->] (Id)  to [out=-30,in=210,looseness=1] (A);
  \draw[->] (Id)  to [out=-15,in=195,looseness=1] (A);
  \draw (5.3,0) node (dots) {$\dots$};
  \draw[->] (Id)  to [out=15,in=165,looseness=1] (A);
  \draw[->] (Id)  to [out=30,in=150,looseness=1] (A);
 \end{tikzpicture}
\]
\vspace{-30pt}
\[
\raisebox{3pt}{\includegraphics[width=.25\linewidth]{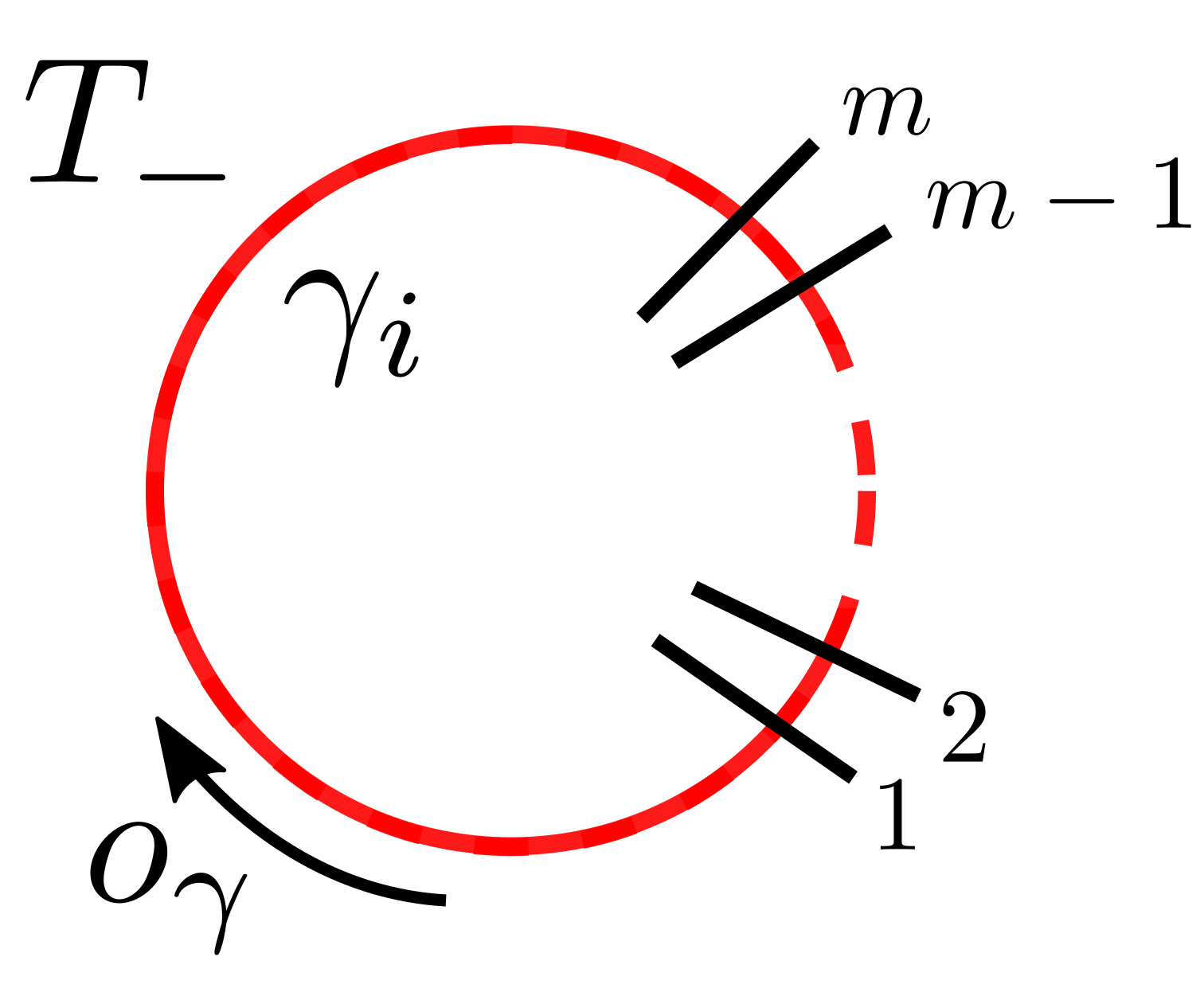}}
\begin{tikzpicture}
  \draw (-1.4,0) node (Sp1) { };
  \draw (-.9,0) node (=>) {$\mapsto$};
  \draw (0,2) node (Sp2) { };

  \draw (0,0) node (C) {$C_\gamma$};
  \draw (1,0) node (D) {$\Delta_\gamma$};
  \draw (2.5,0) node (dots) {$\dots$};
  \node at (4.3,0) [draw,rectangle,dashed] (Id) {$\op{Fl}^{(m)}$};
  \node at (6.2,0) [draw,circle,dashed] (A) {$A$};

  \draw (2.2,-1.2) node (S1) {$S_\gamma^{\sigma_m+1}$};
  \draw (2.2,-.6) node (S2) {$S_\gamma^{\sigma_{m-1}+1}$};
  \draw (2.2,.6) node (S3) {$S_\gamma^{\sigma_2+1}$};
  \draw (2.2,1.2) node (S4) {$S_\gamma^{\sigma_1+1}$};

  \draw (3.4,-1.2) node (P1) {$\bullet$};
  \draw (3.4,-.6) node (P2) {$\bullet$};
  \draw (3.4,.6) node (P3) {$\bullet$};
  \draw (3.4,1.2) node (P4) {$\bullet$};

  \draw[->] (C)--(D);

  \draw[->] (D)  to [out=-90,in=180,looseness=1.5] (S1);
  \draw[->] (D)  to [out=-80,in=180,looseness=1.5] (S2);
  \draw [->] (D)--(1.6,.2);
  \draw [->] (D)--(1.6,0);
  \draw [->] (D)--(1.6,-.2);
  \draw[->] (D)  to [out=80,in=180,looseness=1.5] (S3);
  \draw[->] (D)  to [out=90,in=180,looseness=1.5] (S4);

  \draw[->] (S1)--(P1);
  \draw[->] (S2)--(P2);
  \draw[->] (S3)--(P3);
  \draw[->] (S4)--(P4);

  \draw[->] (P1)  to [in=-90,out=0,looseness=1.5] (Id);
  \draw[->] (P2)  to [in=-120,out=0,looseness=1.5] (Id);
  \draw [->] (3.4,.2)--(Id);
  \draw [->] (3.4,0)--(Id);
  \draw [->] (3.4,-.2)--(Id);
  \draw[->] (P3)  to [in=120,out=0,looseness=1.5] (Id);
  \draw[->] (P4)  to [in=90,out=0,looseness=1.5] (Id);

  \draw[->] (Id)  to [out=-30,in=210,looseness=1] (A);
  \draw[->] (Id)  to [out=-15,in=195,looseness=1] (A);
  \draw (5.3,0) node (dots) {$\dots$};
  \draw[->] (Id)  to [out=15,in=165,looseness=1] (A);
  \draw[->] (Id)  to [out=30,in=150,looseness=1] (A);
 \end{tikzpicture}
\]Here $A$ is the same $m$-input tensor sub-diagram in both of the right-most diagrams and $\op{Fl}^{(m)}$ denotes the $m$ input and $m$ output tensor permuting the $i$-th input to the $(m-i)$-th output. We use the fact that $S_\gamma^2 = \op{Id}$, so that $S_\gamma^k$ depends only on $k \mod 2$. Now we compute that:\[\begin{tikzpicture}

%eqn 1

  \draw (0,0) node (C) {$C_\gamma$};
  \draw (1,0) node (D) {$\Delta_\gamma$};
  \draw (2.5,0) node (dots) {$\dots$};
  \node at (4.3,0) [draw,rectangle,dashed] (Id) {$\op{Fl}^{(m)}$};

  \draw (2.2,-1.2) node (S1) {$S_\gamma^{\sigma_m+1}$};
  \draw (2.2,-.6) node (S2) {$S_\gamma^{\sigma_{m-1}+1}$};
  \draw (2.2,.6) node (S3) {$S_\gamma^{\sigma_2+1}$};
  \draw (2.2,1.2) node (S4) {$S_\gamma^{\sigma_1+1}$};

  \draw (3.4,-1.2) node (P1) {$\bullet$};
  \draw (3.4,-.6) node (P2) {$\bullet$};
  \draw (3.4,.6) node (P3) {$\bullet$};
  \draw (3.4,1.2) node (P4) {$\bullet$};

  \draw[->] (C)--(D);

  \draw[->] (D)  to [out=-90,in=180,looseness=1.5] (S1);
  \draw[->] (D)  to [out=-80,in=180,looseness=1.5] (S2);
  \draw [->] (D)--(1.6,.2);
  \draw [->] (D)--(1.6,0);
  \draw [->] (D)--(1.6,-.2);
  \draw[->] (D)  to [out=80,in=180,looseness=1.5] (S3);
  \draw[->] (D)  to [out=90,in=180,looseness=1.5] (S4);

  \draw[->] (S1)--(P1);
  \draw[->] (S2)--(P2);
  \draw[->] (S3)--(P3);
  \draw[->] (S4)--(P4);

  \draw[->] (P1)  to [in=-90,out=0,looseness=1.5] (Id);
  \draw[->] (P2)  to [in=-120,out=0,looseness=1.5] (Id);
  \draw [->] (3.4,.2)--(Id);
  \draw [->] (3.4,0)--(Id);
  \draw [->] (3.4,-.2)--(Id);
  \draw[->] (P3)  to [in=120,out=0,looseness=1.5] (Id);
  \draw[->] (P4)  to [in=90,out=0,looseness=1.5] (Id);

  \draw[->] (Id)--(5,.5);
  \draw[->] (Id)--(5.2,.25);
  \draw (5.3,0) node (dots) {$\dots$};
  \draw[->] (Id)--(5.2,-.25);
  \draw[->] (Id)--(5,-.5);

%eqn 2

  \draw (5.9,0) node (=1) {$=$};

  \draw (6.6,0) node (C) {$C_\gamma$};
  \draw (7.6,0) node (D) {$\Delta_\gamma$};
  \node at (9.1,0) [draw,rectangle,dashed] (F1) {$\op{Fl}^{(m)}$};
  \draw (10.9,0) node (dots) {$\dots$};
  \node at (12.7,0) [draw,rectangle,dashed] (Id) {$\op{Fl}^{(m)}$};

  \draw (10.6,-1.2) node (S1) {$S_\gamma^{\sigma_m}$};
  \draw (10.6,-.6) node (S2) {$S_\gamma^{\sigma_{m-1}}$};
  \draw (10.6,.6) node (S3) {$S_\gamma^{\sigma_2}$};
  \draw (10.6,1.2) node (S4) {$S_\gamma^{\sigma_1}$};

  \draw (11.8,-1.2) node (P1) {$\bullet$};
  \draw (11.8,-.6) node (P2) {$\bullet$};
  \draw (11.8,.6) node (P3) {$\bullet$};
  \draw (11.8,1.2) node (P4) {$\bullet$};

  \draw[->] (C)--(D);

  \draw[->] (D)  to [out=40,in=140,looseness=1.5] (F1);
  \draw[->] (D)  to [out=20,in=170,looseness=1.5] (F1);
  \draw (8.2,0) node (dots) {$\dots$};
  \draw[->] (D)  to [out=-20,in=190,looseness=1.5] (F1);
  \draw[->] (D)  to [out=-40,in=220,looseness=1.5] (F1);

  \draw[->] (F1)  to [out=-90,in=180,looseness=1.5] (S1);
  \draw[->] (F1)  to [out=-80,in=180,looseness=1.5] (S2);
  \draw [->] (F1)--(10,.2);
  \draw [->] (F1)--(10,0);
  \draw [->] (F1)--(10,-.2);
  \draw[->] (F1)  to [out=80,in=180,looseness=1.5] (S3);
  \draw[->] (F1)  to [out=90,in=180,looseness=1.5] (S4);

  \draw[->] (S1)--(P1);
  \draw[->] (S2)--(P2);
  \draw[->] (S3)--(P3);
  \draw[->] (S4)--(P4);

  \draw[->] (P1)  to [in=-90,out=0,looseness=1.5] (Id);
  \draw[->] (P2)  to [in=-120,out=0,looseness=1.5] (Id);
  \draw [->] (11.8,.2)--(Id);
  \draw [->] (11.8,0)--(Id);
  \draw [->] (11.8,-.2)--(Id);
  \draw[->] (P3)  to [in=120,out=0,looseness=1.5] (Id);
  \draw[->] (P4)  to [in=90,out=0,looseness=1.5] (Id);

  \draw[->] (Id)--(13.4,.5);
  \draw[->] (Id)--(13.6,.25);
  \draw (13.7,0) node (dots) {$\dots$};
  \draw[->] (Id)--(13.6,-.25);
  \draw[->] (Id)--(13.4,-.5);

  \draw (14.4,0) node (=2) {$=$};
 \end{tikzpicture}\]\[\begin{tikzpicture}

%eqn 3

  \draw (0,0) node (C) {$C_\gamma$};
  \draw (1,0) node (D) {$\Delta_\gamma$};
  \draw (2.5,0) node (dots) {$\dots$};
  \node at (4.3,0) [draw,rectangle,dashed] (Id) {$\op{Fl}^{(m)}$};
  \node at (6.2,0) [draw,rectangle,dashed] (Id2) {$\op{Fl}^{(m)}$};

  \draw (2.2,-1.2) node (S1) {$S_\gamma^{\sigma_1}$};
  \draw (2.2,-.6) node (S2) {$S_\gamma^{\sigma_2}$};
  \draw (2.2,.6) node (S3) {$S_\gamma^{\sigma_{m-1}}$};
  \draw (2.2,1.2) node (S4) {$S_\gamma^{\sigma_m}$};

  \draw (3.4,-1.2) node (P1) {$\bullet$};
  \draw (3.4,-.6) node (P2) {$\bullet$};
  \draw (3.4,.6) node (P3) {$\bullet$};
  \draw (3.4,1.2) node (P4) {$\bullet$};

  \draw[->] (C)--(D);

  \draw[->] (D)  to [out=-90,in=180,looseness=1.5] (S1);
  \draw[->] (D)  to [out=-80,in=180,looseness=1.5] (S2);
  \draw [->] (D)--(1.6,.2);
  \draw [->] (D)--(1.6,0);
  \draw [->] (D)--(1.6,-.2);
  \draw[->] (D)  to [out=80,in=180,looseness=1.5] (S3);
  \draw[->] (D)  to [out=90,in=180,looseness=1.5] (S4);

  \draw[->] (S1)--(P1);
  \draw[->] (S2)--(P2);
  \draw[->] (S3)--(P3);
  \draw[->] (S4)--(P4);

  \draw[->] (P1)  to [in=-90,out=0,looseness=1.5] (Id);
  \draw[->] (P2)  to [in=-120,out=0,looseness=1.5] (Id);
  \draw [->] (3.4,.2)--(Id);
  \draw [->] (3.4,0)--(Id);
  \draw [->] (3.4,-.2)--(Id);
  \draw[->] (P3)  to [in=120,out=0,looseness=1.5] (Id);
  \draw[->] (P4)  to [in=90,out=0,looseness=1.5] (Id);

  \draw[->] (Id)  to [out=-30,in=210,looseness=1] (Id2);
  \draw[->] (Id)  to [out=-15,in=195,looseness=1] (Id2);
  \draw (5.3,0) node (dots) {$\dots$};
  \draw[->] (Id)  to [out=15,in=165,looseness=1] (Id2);
  \draw[->] (Id)  to [out=30,in=150,looseness=1] (Id2);

  \draw[->] (Id2)--(6.8,.5);
  \draw[->] (Id2)--(7,.25);
  \draw (7.2,0) node (dots) {$\dots$};
  \draw[->] (Id2)--(7,-.25);
  \draw[->] (Id2)--(6.8,-.5);

%eqn 4

  \draw (7.8,0) node (=1) {$=$};

  \draw (8.4,0) node (C) {$C_\gamma$};
  \draw (9.4,0) node (D) {$\Delta_\gamma$};
  \draw (10.9,0) node (dots) {$\dots$};
  \node at (12.7,0) [draw,rectangle,dashed] (Id) {$\op{Id}^{(m)}$};

  \draw (10.6,-1.2) node (S1) {$S_\gamma^{\sigma_1}$};
  \draw (10.6,-.6) node (S2) {$S_\gamma^{\sigma_2}$};
  \draw (10.6,.6) node (S3) {$S_\gamma^{\sigma_{m-1}}$};
  \draw (10.6,1.2) node (S4) {$S_\gamma^{\sigma_m}$};

  \draw (11.7,-1.2) node (P1) {$\bullet$};
  \draw (11.7,-.6) node (P2) {$\bullet$};
  \draw (11.7,.6) node (P3) {$\bullet$};
  \draw (11.7,1.2) node (P4) {$\bullet$};

  \draw[->] (C)--(D);

  \draw[->] (D)  to [out=-90,in=180,looseness=1.5] (S1);
  \draw[->] (D)  to [out=-80,in=180,looseness=1.5] (S2);
  \draw [->] (D)--(10,.2);
  \draw [->] (D)--(10,0);
  \draw [->] (D)--(10,-.2);
  \draw[->] (D)  to [out=80,in=180,looseness=1.5] (S3);
  \draw[->] (D)  to [out=90,in=180,looseness=1.5] (S4);

  \draw[->] (S1)--(P1);
  \draw[->] (S2)--(P2);
  \draw[->] (S3)--(P3);
  \draw[->] (S4)--(P4);

  \draw[->] (P1)  to [in=-90,out=0,looseness=1.5] (Id);
  \draw[->] (P2)  to [in=-120,out=0,looseness=1.5] (Id);
  \draw [->] (11.8,.2)--(Id);
  \draw [->] (11.8,0)--(Id);
  \draw [->] (11.8,-.2)--(Id);
  \draw[->] (P3)  to [in=120,out=0,looseness=1.5] (Id);
  \draw[->] (P4)  to [in=90,out=0,looseness=1.5] (Id);

  \draw[->] (Id)--(13.4,.5);
  \draw[->] (Id)--(13.6,.25);
  \draw (13.7,0) node (dots) {$\dots$};
  \draw[->] (Id)--(13.6,-.25);
  \draw[->] (Id)--(13.4,-.5);

\end{tikzpicture}\]
Here we are using the cotrace/antipode identity $C_\gamma \rightarrow S_\gamma \rightarrow = C_\gamma \rightarrow $ and the coproduct/antipode identity $\rightarrow \Delta_\gamma \rightrightarrows S_\gamma^{\otimes m} \rightrightarrows \,\,=\,\, \rightarrow S_\gamma \rightarrow \Delta_\gamma \rightrightarrows \op{Fl}^{(m)}\rightrightarrows$\,. This yields the desired tensorial equality.
\end{proof}

Next, we illustrate the various ways that $\langle -\rangle_{\mathcal{H}}$ transforms under the elementary operations on trisections, discussed in Definition \ref{def:basic_trisection_constructions}.

\begin{proposition} \label{prop:properties_of_bracket} (Properties of Bracket) Let $\mathcal{H}$ be a Hopf triplet and let $T,T'$ be trisections. The trisection bracket $\langle -\rangle_{\mathcal{H}}$ has the following properties.
\begin{itemize}
	\item[(a)] (Diffeomorphism) $\langle -\rangle_{\mathcal{H}}$ is invariant under oriented diffeomorphism.
	\item[(b)] (Isotopy) $\langle -\rangle_{\mathcal{H}}$ is invariant under isotopy of trisections.
	\item[(c)] (Connect Sum) $\langle -\rangle_{\mathcal{H}}$ satisfies $\langle T \# T'\rangle_{\mathcal{H}} = \langle T\rangle_{\mathcal{H}} \cdot \langle T'\rangle_{\mathcal{H}}$.
	\item[(d)] (Handle Slides) $\langle -\rangle_{\mathcal{H}}$ is invariant under handle-slides.
\end{itemize}
\end{proposition}

\begin{proof}
\noindent \emph{(a) - Diffeomorphism.} The number of $\alpha,\beta$ and $\kappa$ curves as well as the number, order and sign of the pairwise intersections are all preserved under the diffeomorphism. Thus the tensor diagrams defining $\langle T\rangle_{\mathcal{H}}$ and $\langle T'\rangle_{\mathcal{H}}$ are the same when $T$ and $T'$ are diffeomorphic.

\vspace{5pt}

\noindent \emph{(b) - Isotopy.} For isotopies, let $T_+$ and $T_-$ be isotopic. By Lemma \ref{lem:two_three_point_moves} and diffeomorphism invariance, we simply need to show that $\langle T_+\rangle_{\mathcal{H}} = \langle T_-\rangle_{\mathcal{H}}$ if $T_+$ and $T_-$ are related by a two-point move or a three-point move (see Definition \ref{def:two_three_point_moves}). We proceed with these two cases.

\vspace{5pt}

\noindent \emph{(b)(i) - Two-Point Move.} Let $T_+$ and $T_-$ be two diagrams related by a two-point move. After orienting and relabelling $T_+$ and $T_-$, we have the following pair of sub-diagrams of the trisection diagrams, along with their corresponding tensor diagrams.
\[
\includegraphics[width=.25\linewidth]{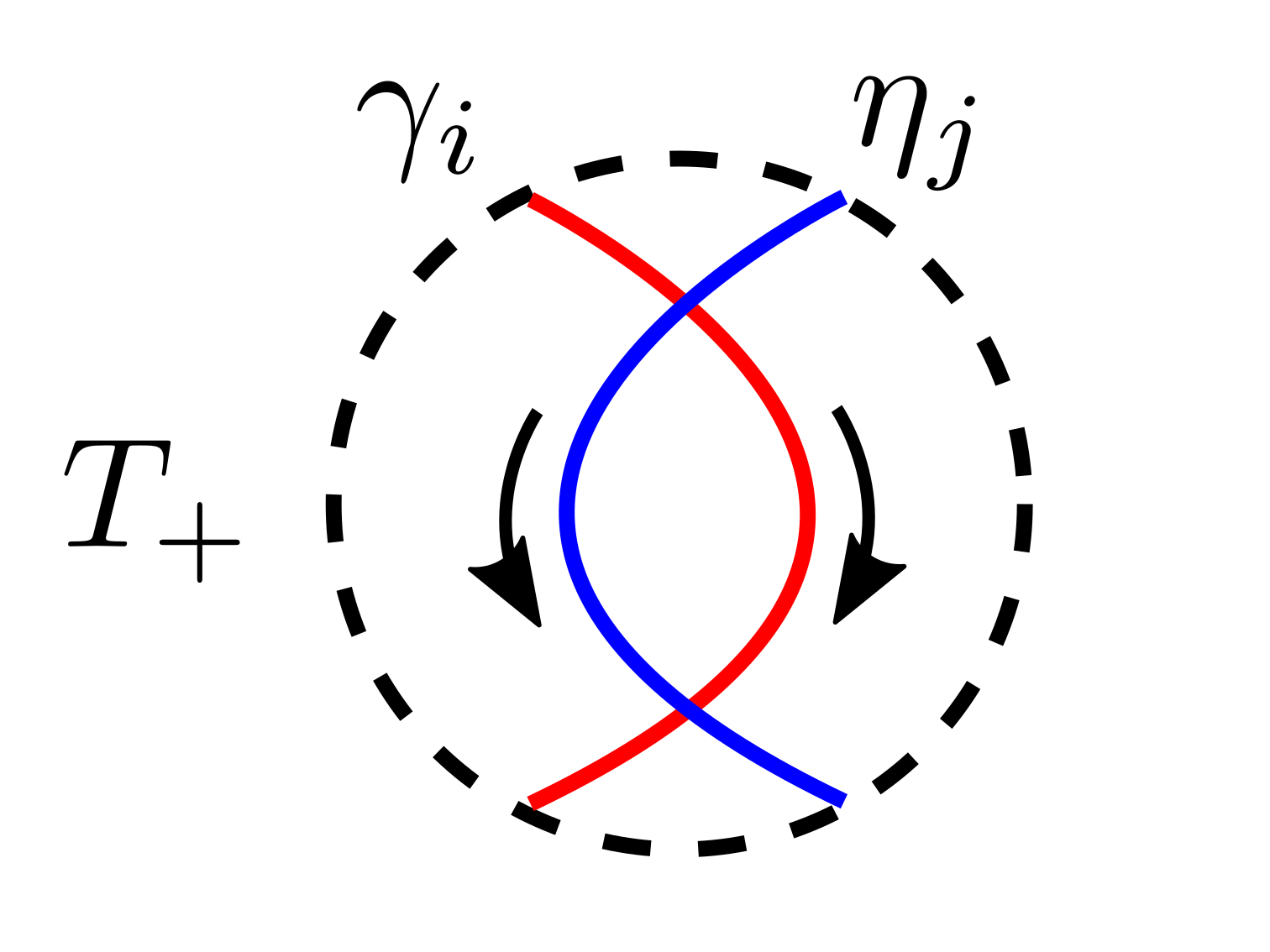}
\begin{tikzpicture}
  \draw (-1.8,0) node (=>) {$\mapsto$};
  \draw (0,-1.2) node (Sp2) { };

  \draw (0,0) node (D1) {$\Delta_\gamma$};
  \draw (1,1) node (S1) {$S_\gamma$};
  \draw (2,1) node (P1) {$\bullet$};
  \draw (1.5,-1) node (P2) {$\bullet$};
  \draw (3,0) node (D2) {$\Delta_\eta$};
  \draw (-.6,0) node (dots1) {$\dots$};
  \draw (3.6,0) node (dots2) {$\dots$};

 \draw[->] (D1)--(0,1);
 \draw[->] (D1)--(0,-1);
 \draw[->] (D1)--(-.5,1);
 \draw[->] (D1)--(-.5,-1);

 \draw[->] (D2)--(3,1);
 \draw[->] (D2)--(3,-1);
 \draw[->] (D2)--(3.5,1);
 \draw[->] (D2)--(3.5,-1);

 \draw[->] (D1)--(S1);
 \draw[->] (S1)--(P1);
 \draw[->] (D2)--(P1);
 \draw[->] (D1)--(P2);
 \draw[->] (D2)--(P2);

 \draw (4.5,0) node (=) {$=$};

  \draw (5.5,0) node (D3) {$\Delta_\gamma$};
  \draw (6.5,1) node (S3) {$S_\gamma$};
  \draw (7.5,1) node (P3) {$\bullet$};
  \draw (7,-1) node (P4) {$\bullet$};
  \draw (8.5,0) node (D4) {$\Delta_\eta$};
  \node at (7,0) [draw,circle,dashed] (T) {$T$};

 \draw[->] (D3)--(S3);
 \draw[->] (S3)--(P3);
 \draw[->] (D4)--(P3);
 \draw[->] (D3)--(P4);
 \draw[->] (D4)--(P4);

 \draw[->] (T) to [out=220,in=270,looseness=1.5] (D3);
 \draw[->] (T) to [out=320,in=270,looseness=1.5] (D4);
 \end{tikzpicture}
\]\[
\includegraphics[width=.25\linewidth]{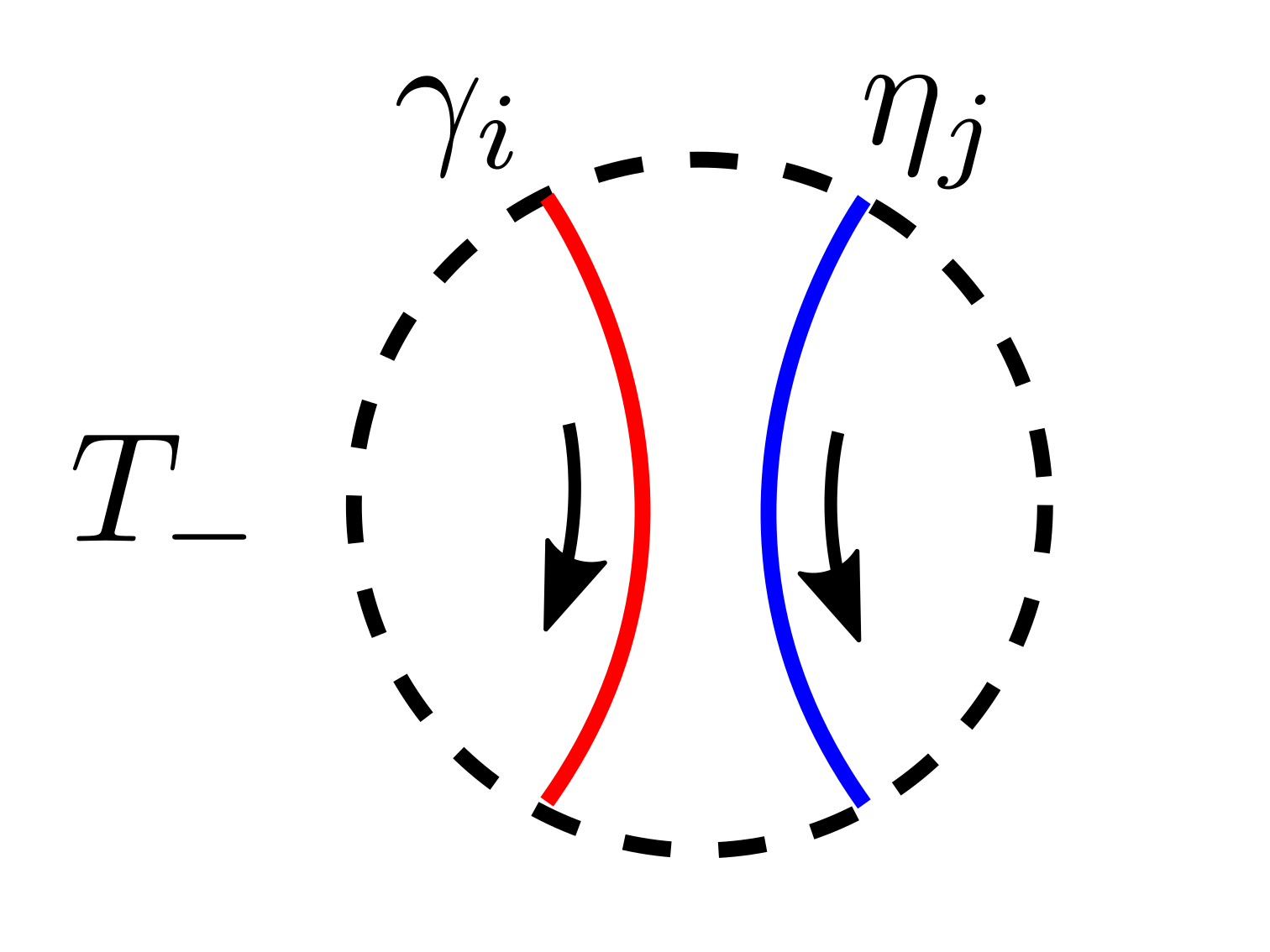}
\begin{tikzpicture}
  \draw (-1.8,0) node (=>) {$\mapsto$};
  \draw (0,-1.2) node (Sp2) { };

  \draw (0,0) node (D1) {$\Delta_\gamma$};
  \draw (3,0) node (D2) {$\Delta_\eta$};
  \draw (-.6,0) node (dots1) {$\dots$};
  \draw (3.6,0) node (dots2) {$\dots$};

 \draw[->] (D1)--(0,1);
 \draw[->] (D1)--(0,-1);
 \draw[->] (D1)--(-.5,1);
 \draw[->] (D1)--(-.5,-1);

 \draw[->] (D2)--(3,1);
 \draw[->] (D2)--(3,-1);
 \draw[->] (D2)--(3.5,1);
 \draw[->] (D2)--(3.5,-1);

  \draw (4.5,0) node (=) {$=$};

  \draw (5.5,0) node (e3) {$\epsilon_\gamma$};
  \draw (8.5,0) node (e4) {$\epsilon_\eta$};
  \node at (7,0) [draw,circle,dashed] (T) {$T$};

 \draw[->] (T)--(e3);
 \draw[->] (T)--(e4);
 \end{tikzpicture}
\]
Here the diagrams are equal outside of the region depicted, and $T$ denotes the same tensor sub-diagrams in both of the right-most diagrams. We now compute that:
\[
\begin{tikzpicture}
  \draw (0,0) node (D3) {$\Delta_\gamma$};
  \draw (1,1) node (S3) {$S_\gamma$};
  \draw (2,1) node (P3) {$\bullet$};
  \draw (1.5,0) node (P4) {$\bullet$};
  \draw (3,0) node (D4) {$\Delta_\eta$};

 \draw[->] (0,1)--(D3);
 \draw[->] (3,1)--(D4);
 \draw[->] (D3)--(S3);
 \draw[->] (S3)--(P3);
 \draw[->] (D4)--(P3);
 \draw[->] (D3)--(P4);
 \draw[->] (D4)--(P4);

 \draw (4,.5) node (=) {$=$};

  \draw (5,0) node (D5) {$\Delta_\gamma$};
  \draw (6,1) node (S5) {$S_\gamma$};
  \draw (6,0) node (M5) {$M_\gamma$};
  \draw (7,0) node (P5) {$\bullet$};

 \draw[->] (5,1)--(D5);
 \draw[->] (D5)--(S5);
 \draw[->] (D5)--(M5);
 \draw[->] (S5)--(M5);
 \draw[->] (M5)--(P5);
 \draw[->] (7,1)--(P5);

 \draw (8,.5) node (=) {$=$};

  \draw (9,0) node (e5) {$\epsilon_\gamma$};
  \draw (10,1) node (P5) {$\bullet$};
  \draw (10,0) node (n5) {$\eta_\gamma$};

 \draw[->] (9,1)--(e5);
 \draw[->] (n5)--(P5);
 \draw[->] (11,1)--(P5);

 \draw (12,.5) node (=) {$=$};

   \draw (13,0) node (e7) {$\epsilon_\gamma$};
  \draw (14,0) node (e8) {$\epsilon_\eta$};

 \draw[->] (13,1)--(e7);
 \draw[->] (14,1)--(e8);
 \end{tikzpicture}
\]
This proves that the diagrams computing $\langle T_+\rangle_{\mathcal{H}}$ and $\langle T_-\rangle_{\mathcal{H}}$ specify the same tensor.

\vspace{5pt}

\noindent \emph{(b)(ii) - Three-Point Move.} Let $T_+$ and $T_-$ be two diagrams related by a 3-point move. By choosing our curve orientations properly, we can ensure that the diagrams $T_+$ and $T_-$ will be modelled (locally, near the move region) by one of the following pairs of diagrams.

The first local model gives a counter-clockwise order to the curves in the $+$ diagram.
\[
\includegraphics[width=.25\linewidth]{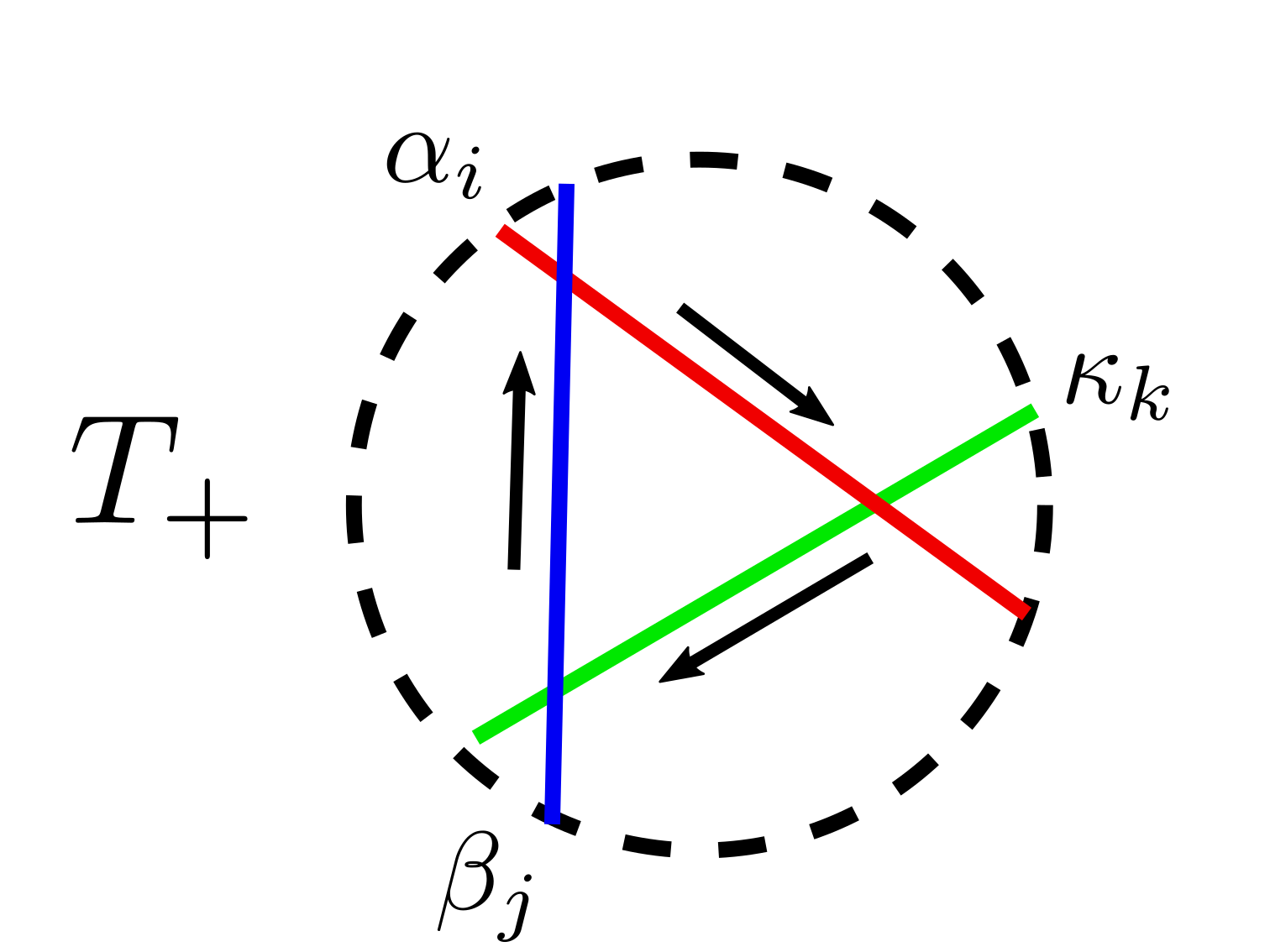} 
\qquad
\includegraphics[width=.25\linewidth]{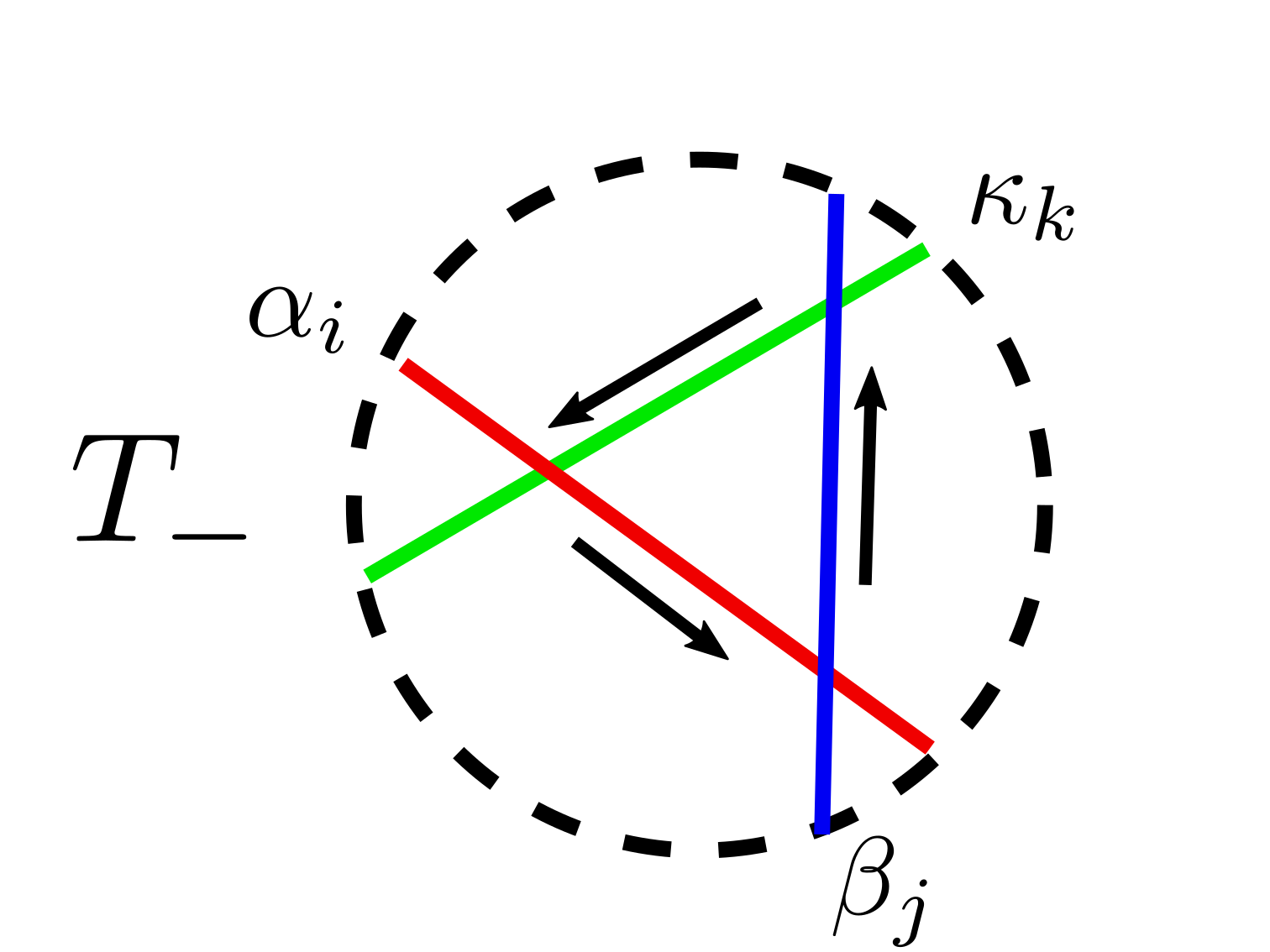}
\]
The second local model gives a clockwise cyclic order to the curves in the $+$ diagram.
\[
\includegraphics[width=.25\linewidth]{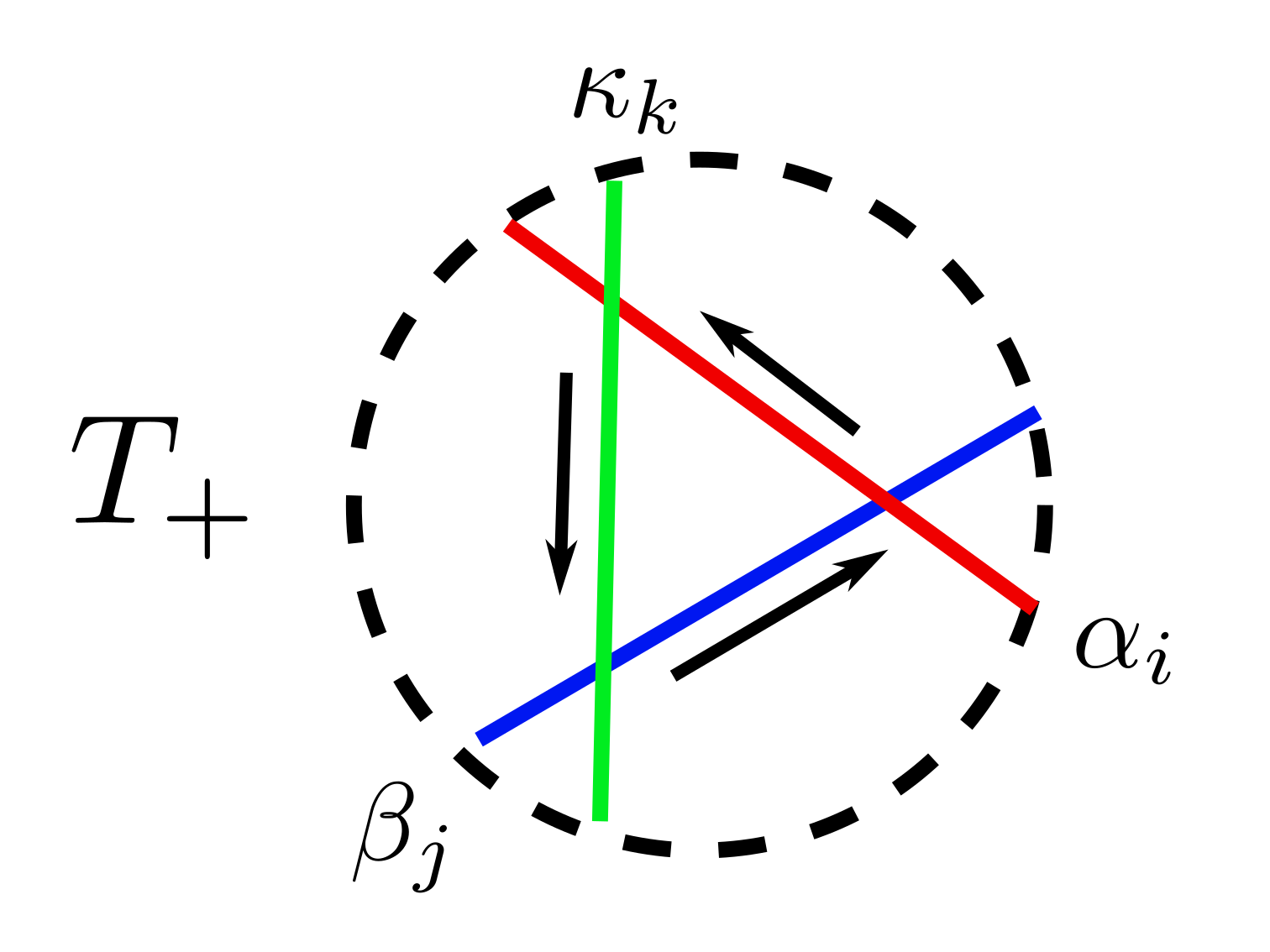} 
\qquad
\includegraphics[width=.25\linewidth]{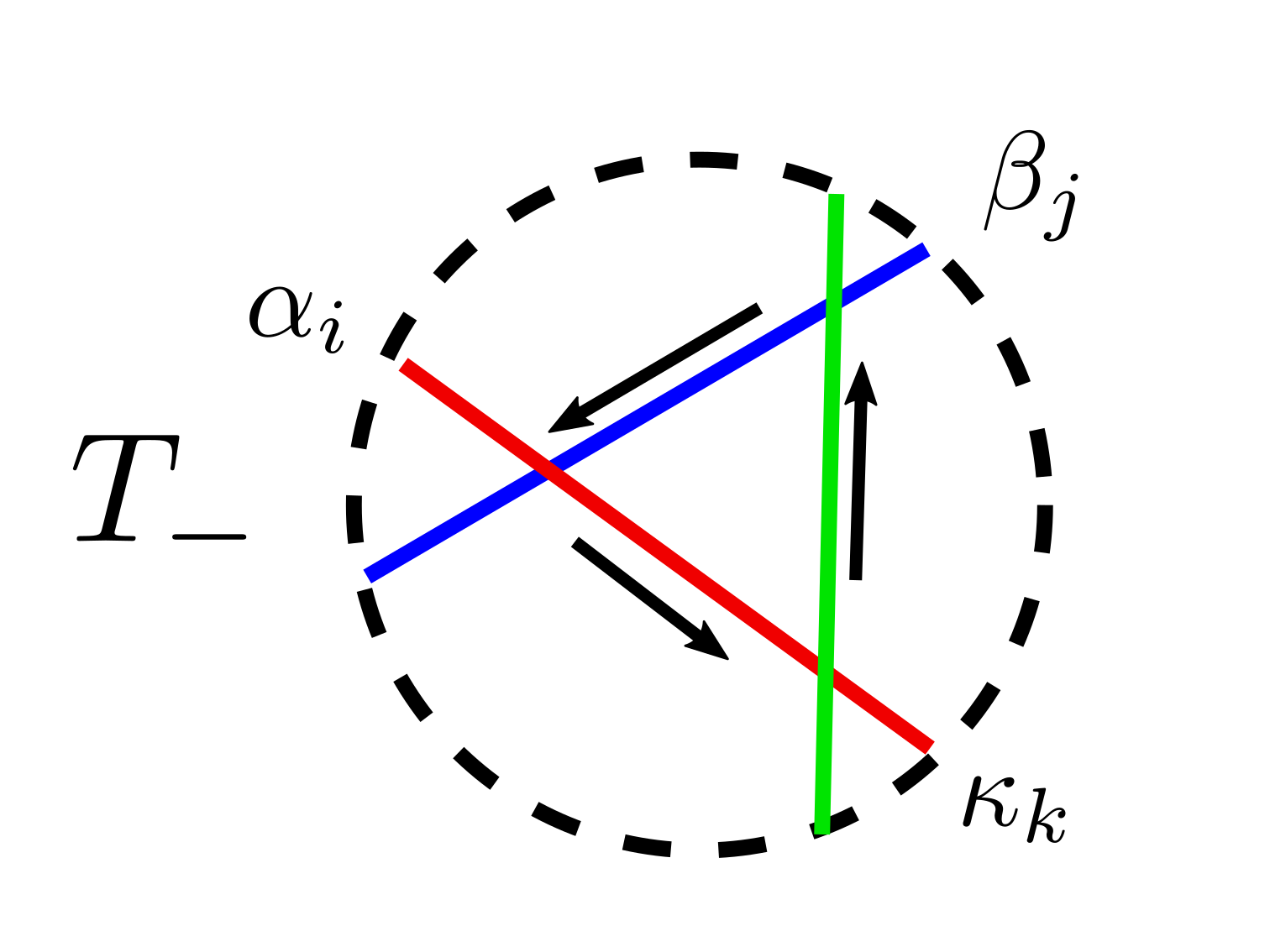}
\]
We focus on the first case, the second being exactly analogous in a manner that we will remark on near the end of the proof.

Proceeding, we write the local contribution of each of these regions to their respective brackets. 
\[
\includegraphics[width=.25\linewidth]{three_point_move_p_a.png}
\begin{tikzpicture}
  \draw (-.8,0) node (=>) {$\mapsto$};
  \draw (0,-1.2) node (Sp2) { };

  \draw (1,0) node (Db) {$\Delta_\beta$};
  \draw (2,1) node (Da) {$\Delta_\alpha$};
  \draw (2,-1) node (Dc) {$\Delta_\kappa$};

  \draw (1,1) node (Pab) {$\bullet$};
  \draw (1,-1) node (Pbc) {$\bullet$};
  \draw (2,0) node (Pca) {$\bullet$};

  \draw (3,1) node (in_a) {};
  \draw (0,0) node (in_b) {};
  \draw (3,-1) node (in_c) {};

 \draw[->] (in_a)--(Da);
 \draw[->] (in_b)--(Db);
 \draw[->] (in_c)--(Dc);

 \draw[->] (Da)--(Pab);
 \draw[->] (Da)--(Pca);
 \draw[->] (Db)--(Pab);
 \draw[->] (Db)--(Pbc);
 \draw[->] (Dc)--(Pbc);
 \draw[->] (Dc)--(Pca);
 \end{tikzpicture}
\]
\[
\includegraphics[width=.25\linewidth]{three_point_move_n_a.png}
\begin{tikzpicture}
  \draw (-1.8,0) node (=>) {$\mapsto$};
  \draw (0,-1.2) node (Sp2) { };

  \draw (0,0) node (Db) {$\Delta_\beta$};
  \draw (2.5,1) node (Da) {$\Delta_\alpha$};
  \draw (2.5,-1) node (Dc) {$\Delta_\kappa$};

  \draw (1,1) node (Pab) {$\bullet$};
  \draw (1,-1) node (Pbc) {$\bullet$};
  \draw (2,0) node (Pca) {$\bullet$};

  \draw (3.5,1) node (in_a) {};
  \draw (-1,0) node (in_b) {};
  \draw (3.5,-1) node (in_c) {};

 \draw[->] (in_a)--(Da);
 \draw[->] (in_b)--(Db);
 \draw[->] (in_c)--(Dc);

 \draw[->] (Da) to [out=180,in=90,looseness=1] (Pca);
 \draw[->] (Da) to [out=200,in=0,looseness=1] (Pab);
 \draw[->] (Db) to [out=-20,in=-90,looseness=1] (Pab);
 \draw[->] (Db) to [out=20,in=90,looseness=1] (Pbc);
 \draw[->] (Dc) to [out=160,in=0,looseness=1] (Pbc);
 \draw[->] (Dc) to [out=180,in=-90,looseness=1] (Pca);

  \draw (3.7,0) node (=1) {$=$};

  \draw (8,1) node (Sa2) {$S_\alpha$};
  \draw (5,0) node (Sb2) {$S_\beta$};
  \draw (8,-1) node (Sc2) {$S_\kappa$};

  \draw (6,0) node (Db2) {$\Delta_\beta$};
  \draw (7,1) node (Da2) {$\Delta_\alpha$};
  \draw (7,-1) node (Dc2) {$\Delta_\kappa$};

  \draw (6,1) node (Pab2) {$\bullet$};
  \draw (6,-1) node (Pbc2) {$\bullet$};
  \draw (7,0) node (Pca2) {$\bullet$};

  \draw (9,1) node (in_a2) {};
  \draw (4,0) node (in_b2) {};
  \draw (9,-1) node (in_c2) {};

 \draw[->] (in_a2)--(Sa2);
 \draw[->] (Sa2)--(Da2);

 \draw[->] (in_b2)--(Sb2);
 \draw[->] (Sb2)--(Db2);

 \draw[->] (in_c2)--(Sc2);
 \draw[->] (Sc2)--(Dc2);

 \draw[->] (Da2)--(Pab2);
 \draw[->] (Da2)--(Pca2);
 \draw[->] (Db2)--(Pab2);
 \draw[->] (Db2)--(Pbc2);
 \draw[->] (Dc2)--(Pbc2);
 \draw[->] (Dc2)--(Pca2);
 \end{tikzpicture}
\]
Now we simply observe that the equality of these two tensor sub-diagrams is implied by the Hopf triplet axioms, specifically Definition \ref{def:hopf_triplet}(b). This is due to Lemma \ref{lem:fundamental_triplet_lemma}, more precisely the equivalence between Lemma \ref{lem:fundamental_triplet_lemma}(a) and Lemma \ref{lem:fundamental_triplet_lemma}(c).

For the case of the second local model, we use the same argument and appeal to the equivalence of the conditions Lemma \ref{lem:fundamental_triplet_lemma}(a) and Lemma \ref{lem:fundamental_triplet_lemma}(d).

\vspace{5pt}

\noindent \emph{(c) - Connect Sum.} If $T$ and $T'$ are two trisection diagrams, then the oriented connect sum $T \# T'$ has curve sets $\alpha \sqcup \alpha'$, $\beta \sqcup \beta'$ and $\kappa \sqcup \kappa'$. The set of intersections $\mathcal{I}(T \# T')$ is the disjoint union $\mathcal{I}(T) \sqcup \mathcal{I}(T')$ of the intersections of $T$ and $T'$, and the signs of the intersections remain unchanged. Thus the tensor diagram $\langle T \# T'\rangle_{\mathcal{H}}$ is the disjoint union of the diagrams $\langle T\rangle_{\mathcal{H}}$ and $\langle T'\rangle_{\mathcal{H}}$, and from Notation \ref{not:general_tensor_diagrams}(b) we deduce that $\langle T \# T'\rangle_{\mathcal{H}} = \langle T\rangle_{\mathcal{H}} \otimes \langle T'\rangle_{\mathcal{H}} = \langle T\rangle_{\mathcal{H}} \cdot \langle T'\rangle_{\mathcal{H}} \in k$

\vspace{5pt}

\noindent \emph{(d) - Handle Slides.} Let $T$ denote a trisection, and let $\gamma \in \{\alpha,\beta,\kappa\}$. Furthermore, let $\gamma_i$ and $\gamma_j$ be distinct $\gamma$-curves and $\xi$ be an arc connecting $\gamma_i$ to $\gamma_j$ in $\Sigma$. Finally, let $T'$ denote the trisection acquired by a handle-slide of $\gamma_i$ over $\gamma_j$ via $\xi$. Before we proceed, we fix some additional notation. 

First, fix orientations $o_\alpha,o_\beta$ and $o_\kappa$ of the curves, such that the orientation $o_j$ on $\gamma_j$ is induced by the orientation $o_i$ on $\gamma_i$ and the arc $\xi$, in the following sense. Consider the surface $\Sigma - (\gamma_i \cup \xi \cup \gamma_j)$, which is the interior of a (topological) compact surface with two circle boundary components $C$ and $C'$. $C$ contains a copy of $\gamma_i - \gamma_i \cap \xi$ and a copy of $\gamma_j - \gamma_j \cap \xi$. Any orientation $o$ of $\gamma_i$ thus induces orientations $C$, $\gamma_j - \gamma_j \cap \xi$ and therefore $\gamma_j$. The orientation induced by $o$ and $\xi$ is simply $o'$. 

Second, label the curves intersecting $\gamma_i$ by $\nu_1,\dots,\nu_a$ and label the curves intersecting $\gamma_j$ by $\eta_1,\dots,\eta_b$. Here we use the orderings such that the points $\xi \cap \gamma_i,\vu_1 \cap \gamma_i,\dots,\nu_a \cap \gamma_i$ and $\xi \cap \gamma_j, \eta_1 \cap \gamma_j, \dots,\eta_b \cap \gamma_j$ occur in cyclic order about $\gamma_i$ and $\gamma_j$, respectively. Also define $s:\{1,\dots,a\} \to \{0,1\}$ to be $0$ if $\op{sgn}(\gamma_i \cap \vu_l) = +$ and $1$ if $\op{sgn}(\gamma_i \cap \vu_l) = -$, and define $t:\{1,\dots,b\} \to \{0,1\}$ similarly for intersections of $\gamma_j$.

Under this setup, the (curve oriented) trisection diagrams $T$ and $T'$, along with the corresponding tensor diagrams $\langle T\rangle_{\mathcal{H}}$ and $\langle T'\rangle_{\mathcal{H}}$, are the following.

\[
\raisebox{20pt}{\includegraphics[width=.27\linewidth]{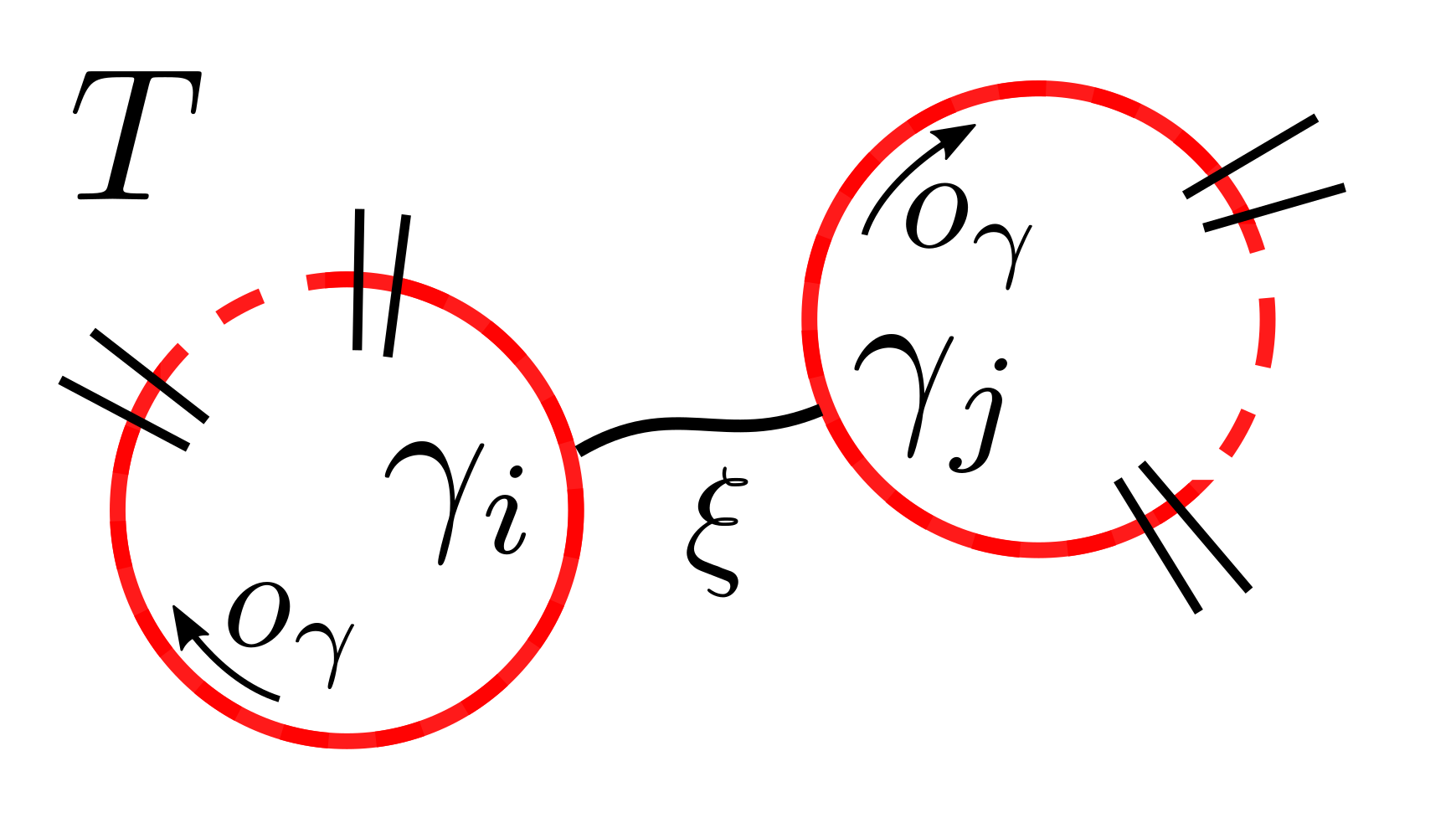}}
\begin{tikzpicture}
  \draw (-1.4,0) node (Sp1) { };
  \draw (-1.4,0) node (=>) {$\mapsto$};
  \draw (0,-1.7) node (Sp2) { };

  \draw (0,1.6) node (C1) {$C_\gamma$};
  \draw (1,1.4) node (D1) {$\Delta_\gamma$};

  \draw (-.6,-.3) node (P1a) {$\bullet$};
  \draw (.1,-.3) node (P1b) {$\bullet$};
  \draw (.8,-.3) node (dots1c) {$\dots$};
  \draw (1.5,-.3) node (P1d) {$\bullet$};

  \draw (-.6,.5) node (S1a) {$S_\gamma^{s_1}$};
  \draw (.1,.5) node (S1b) {$S_\gamma^{s_2}$};
  \draw (.8,.5) node (dots1c2) {$\dots$};
  \draw (1.5,.5) node (S1d) {$S_\gamma^{s_a}$};

  \draw (2,-1.7) node (s1a) {};
  \draw (2,-1.6) node (s1b) {};
  \draw (2,-1.5) node (s1d) {};

  \draw (7,1.6) node (C2) {$C_\gamma$};
  \draw (8,1.4) node (D2) {$\Delta_\gamma$};

  \draw (4.6,-.3) node (P2a) {$\bullet$};
  \draw (5.8,-.3) node (P2b) {$\bullet$};
  \draw (7,-.3) node (dots2c1) {$\dots$};
  \draw (8.2,-.3) node (P2d) {$\bullet$};

  \draw (4.6,.5) node (S2a) {$S_\gamma^{t_1}$};
  \draw (5.8,.5) node (S2b) {$S_\gamma^{t_2}$};
  \draw (7,.5) node (dots2c2) {$\dots$};
  \draw (8.2,.5) node (S2d) {$S_\gamma^{t_b}$};

  \draw (4.6,-1.1) node (D2a) {$\Delta_{\eta_1}$};
  \draw (5.8,-1.1) node (D2b) {$\Delta_{\eta_2}$};
  \draw (7,-1.1) node (dots2c) {$\dots$};
  \draw (8.2,-1.1) node (D2d) {$\Delta_{\eta_b}$};

  \node at (3,-1) [draw,rectangle,dashed,minimum height=45pt] (A) {$A$};

  %Boxed piece
  \draw (-1,1.4) node (R1) {$R_H$};

  \draw[-,dashed] (-.5,2)--(8.7,2);
  \draw[-,dashed] (8.7,2)--(8.7,-1.7);
  \draw[-,dashed] (8.7,-1.7)--(4.1,-1.7);
  \draw[-,dashed] (4.1,-1.7)--(4.1,1);
  \draw[-,dashed] (4.1,1)--(-.5,1);
  \draw[-,dashed] (-.5,1)--(-.5,2);

  \draw[->] (C1)--(D1);
  
  \draw[->] (D1)--(S1a);
  \draw[->] (D1)--(S1b);
  \draw[->] (D1)--(S1d);

  \draw[->] (S1a)--(P1a);
  \draw[->] (S1b)--(P1b);
  \draw[->] (S1d)--(P1d);

  \draw[->] (1.6,-2) to [in=270,out=180,looseness=1] (P1a);
  \draw[->] (1.6,-1.9) to [in=270,out=180,looseness=1] (P1b);
  \draw[->] (1.6,-1.8) to [in=270,out=180,looseness=1,dotted] (dots1c);
  \draw[->] (1.6,-1.7) to [in=270,out=180,looseness=1] (P1d);

  \draw[-] (1.6,-2) to [out=0,in=180,looseness=1] (2.7,-1.6);
  \draw[-] (1.6,-1.9) to [out=0,in=180,looseness=1] (2.7,-1.2);
  \draw[-] (1.6,-1.8) to [out=0,in=180,looseness=1,dotted] (2.7,-.8);
  \draw[-] (1.6,-1.7) to [out=0,in=180,looseness=1] (2.7,-.4);

  \draw[->] (C2)--(D2);

  \draw[->] (D2)--(S2a);
  \draw[->] (D2)--(S2b);
  \draw[->] (D2)--(S2d);

  \draw[->] (S2a)--(P2a);
  \draw[->] (S2b)--(P2b);
  \draw[->] (S2d)--(P2d);

  \draw[<-] (P2a)--(D2a);
  \draw[<-] (P2b)--(D2b);
  \draw[<-] (P2d)--(D2d);

  \draw[->] (4,-2) to [in=270,out=0,looseness=.6] (D2d);
  \draw[->] (4,-1.9) to [in=270,out=0,looseness=.6] (dots2c);
  \draw[->] (4,-1.8) to [in=260,out=0,looseness=.6] (D2b);
  \draw[->] (4,-1.7) to [in=250,out=0,looseness=.6] (D2a);

  \draw[-] (4,-2) to [out=180,in=0,looseness=1] (3.3,-1.6);
  \draw[-] (4,-1.9) to [out=180,in=0,looseness=1] (3.3,-1.2);
  \draw[-] (4,-1.8) to [out=180,in=0,looseness=1,dotted] (3.3,-.8);
  \draw[-] (4,-1.7) to [out=180,in=0,looseness=1] (3.3,-.4);

  %\draw[->] (D1) to [out=30,in=180,looseness=1.5] (P1);

 \end{tikzpicture}
\] \[
\raisebox{20pt}{\includegraphics[width=.27\linewidth]{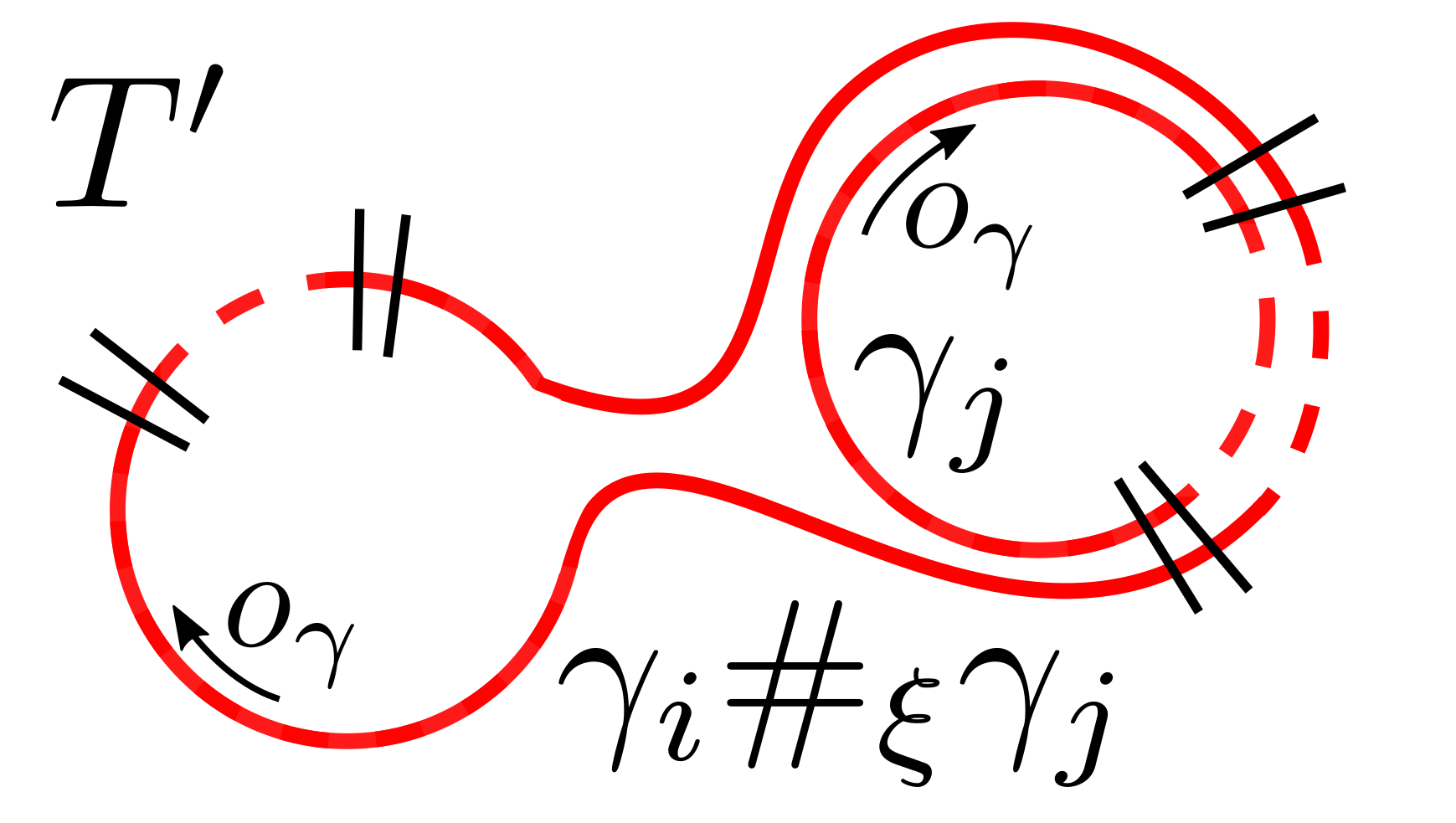}}
\begin{tikzpicture}

  \draw (-1.4,0) node (Sp1) { };
  \draw (-1.4,0) node (=>) {$\mapsto$};
  \draw (0,-1.7) node (Sp2) { };

  \draw (0,1.6) node (C1) {$C_\gamma$};
  \draw (1,1.4) node (D1) {$\Delta_\gamma$};

  \draw (-.6,-.3) node (P1a) {$\bullet$};
  \draw (.1,-.3) node (P1b) {$\bullet$};
  \draw (.8,-.3) node (dots1c) {$\dots$};
  \draw (1.5,-.3) node (P1d) {$\bullet$};

  \draw (-.6,.5) node (S1a) {$S_\gamma^{s_1}$};
  \draw (.1,.5) node (S1b) {$S_\gamma^{s_2}$};
  \draw (.8,.5) node (dots1c2) {$\dots$};
  \draw (1.5,.5) node (S1d) {$S_\gamma^{s_a}$};

  \draw (2,-1.7) node (s1a) {};
  \draw (2,-1.6) node (s1b) {};
  \draw (2,-1.5) node (s1d) {};

  \draw (7,1.6) node (C2) {$C_\gamma$};
  \draw (8,1.4) node (D2) {$\Delta_\gamma$};

  \draw (4.6,1.4) node (D3) {$\Delta_\gamma$};

  \draw (4.6,-1.1) node (D2a) {$\Delta_{\eta_1}$};
  \draw (5.8,-1.1) node (D2b) {$\Delta_{\eta_2}$};
  \draw (7,-1.1) node (dots2c) {$\dots$};
  \draw (8.2,-1.1) node (D2d) {$\Delta_{\eta_b}$};

  \draw (4.8,-.3) node (P2a) {$\bullet$};
  \draw (6,-.3) node (P2b) {$\bullet$};
  \draw (7,-.3) node (dots2c1) {$\dots$};
  \draw (8.4,-.3) node (P2d) {$\bullet$};

  \draw (4.4,-.3) node (P3a) {$\bullet$};
  \draw (5.6,-.3) node (P3b) {$\bullet$};
  \draw (8,-.3) node (P3d) {$\bullet$};

  \draw (4.6,.5) node (F2a) {$E_{1}$};
  \draw (5.8,.5) node (F2b) {$E_{2}$};
  \draw (7,.5) node (dots2c2) {$\dots$};
  \draw (8.2,.5) node (F2d) {$E_{b}$};

  \node at (3,-1) [draw,rectangle,dashed,minimum height=45pt] (A) {$A$};

  %Boxed piece
  \draw (-1,1.4) node (R1) {$R'_H$};

  \draw[-,dashed] (-.5,2)--(8.7,2);
  \draw[-,dashed] (8.7,2)--(8.7,-1.7);
  \draw[-,dashed] (8.7,-1.7)--(4.1,-1.7);
  \draw[-,dashed] (4.1,-1.7)--(4.1,1);
  \draw[-,dashed] (4.1,1)--(-.5,1);
  \draw[-,dashed] (-.5,1)--(-.5,2);

  \draw[->] (C1)--(D1);
  
  \draw[->] (D1)--(S1a);
  \draw[->] (D1)--(S1b);
  \draw[->] (D1)--(S1d);

  \draw[->] (S1a)--(P1a);
  \draw[->] (S1b)--(P1b);
  \draw[->] (S1d)--(P1d);

  \draw[->] (1.6,-2) to [in=270,out=180,looseness=1] (P1a);
  \draw[->] (1.6,-1.9) to [in=270,out=180,looseness=1] (P1b);
  \draw[->] (1.6,-1.8) to [in=270,out=180,looseness=1,dotted] (dots1c);
  \draw[->] (1.6,-1.7) to [in=270,out=180,looseness=1] (P1d);

  \draw[-] (1.6,-2) to [out=0,in=180,looseness=1] (2.7,-1.6);
  \draw[-] (1.6,-1.9) to [out=0,in=180,looseness=1] (2.7,-1.2);
  \draw[-] (1.6,-1.8) to [out=0,in=180,looseness=1,dotted] (2.7,-.8);
  \draw[-] (1.6,-1.7) to [out=0,in=180,looseness=1] (2.7,-.4);

  \draw[->] (C2)--(D2);

  \draw[->] (D1)--(D3);

  \draw[->] (D2)--(F2a);
  \draw[->] (D2)--(F2b);
  \draw[->] (D2)--(F2d);

  \draw[->] (D3)--(F2a);
  \draw[->] (D3)--(F2b);
  \draw[->] (D3)--(F2d);

  \draw[->] (F2a)--(P2a);
  \draw[->] (F2b)--(P2b);
  \draw[->] (F2d)--(P2d);

  \draw[->] (F2a)--(P3a);
  \draw[->] (F2b)--(P3b);
  \draw[->] (F2d)--(P3d);

  \draw[<-] (P2a)--(D2a);
  \draw[<-] (P2b)--(D2b);
  \draw[<-] (P2d)--(D2d);

  \draw[<-] (P3a)--(D2a);
  \draw[<-] (P3b)--(D2b);
  \draw[<-] (P3d)--(D2d);

  \draw[->] (4,-2) to [in=270,out=0,looseness=.6] (D2d);
  \draw[->] (4,-1.9) to [in=270,out=0,looseness=.6] (dots2c);
  \draw[->] (4,-1.8) to [in=260,out=0,looseness=.6] (D2b);
  \draw[->] (4,-1.7) to [in=250,out=0,looseness=.6] (D2a);

  \draw[-] (4,-2) to [out=180,in=0,looseness=1] (3.3,-1.6);
  \draw[-] (4,-1.9) to [out=180,in=0,looseness=1] (3.3,-1.2);
  \draw[-] (4,-1.8) to [out=180,in=0,looseness=1,dotted] (3.3,-.8);
  \draw[-] (4,-1.7) to [out=180,in=0,looseness=1] (3.3,-.4);

  %\draw[->] (D1) to [out=30,in=180,looseness=1.5] (P1);

 \end{tikzpicture}
\] 

Let us elaborate on the various notational components of the two tensor diagrams above. Above, $A$ specifies the contribution of intersections other than $\nu_l$ and $\eta_l$ (for all $l$). This contribution is the same in both $\langle T\rangle_{\mathcal{H}}$ and $\langle T'\rangle_{\mathcal{H}}$. The symbols $R_H$ and $R'_H$ respectively denote the tensors formed by the circled (or rather, boxed) regions in $\langle T\rangle_{\mathcal{H}}$ and $\langle T'\rangle_{\mathcal{H}}$. By convention, the diagram $\to \Delta_{\eta_i} \to$ simply denotes the identity (see Notation \ref{not:product_coproduct_abbreviated_notation}). The tensor $S^l_\gamma$ denotes $l$-th powers with respect to composition of the tensor $S_\gamma$.

The tensors $E_k$ for $k \in \{1,\dots,b\}$ are abbreviations, and require some careful elaboration. First, fix the notation $\op{col}(\eta) \in \{\alpha,\beta,\kappa\}$ for the color of a curve $\eta$. Given another curve $\xi$, we say that $\op{col}(\xi) - \op{col}(\eta) \equiv 1 \!\mod 3$ if the color of $\xi$ occurs to the right of the color of $\eta$ with respect to the standard cyclic ordering $(\alpha,\beta,\kappa)$ of the colors. So for instance, $\op{col}(\beta_j) - \op{col}(\alpha_i) \equiv 1$. Likewise, $\op{col}(\xi) - \op{col}(\eta) \equiv -1 \mod 3$ if the color of $\xi$ occurs to the left of the color of $\eta$, and $\op{col}(\xi) - \op{col}(\eta) \equiv 0 \mod 3$ if the colors are equal. Using this notation, we define $E_k$ case by case depending on the colors of $\gamma$ and $\eta_j$, and the sign of $\gamma \cap \eta_j$.
\[
\begin{tikzpicture}

  \draw (7,0) node (E1) {$E_k$};
  \draw (8.5,0) node (=0) {$:=$};

  \draw[->] (6.2,-.3)--(E1);
  \draw[->] (6.2,.3)--(E1);
  \draw[->] (E1)--(7.8,-.3);
  \draw[->] (E1)--(7.8,.3);

   \draw[->] (9,-.3)--(10.8,.3);
  \draw[->] (9,.3)--(10.8,-.3);

  \draw (15,0) node (lb1) {if $\op{sgn}(\gamma_j \cap \eta_k) = +$ and $\op{col}(\eta_k) - \op{col}(\gamma_j) \equiv 1$};

 \end{tikzpicture}
\]
\[
\begin{tikzpicture}

  \draw (7,0) node (E1) {$E_k$};
  \draw (8.5,0) node (=0) {$:=$};

  \draw[->] (6.2,-.3)--(E1);
  \draw[->] (6.2,.3)--(E1);
  \draw[->] (E1)--(7.8,-.3);
  \draw[->] (E1)--(7.8,.3);

  \draw (9.8,-.3) node (Sa) {$S_\gamma$};
  \draw (9.8,.3) node (Sb) {$S_\gamma$};

  \draw[->] (9,-.3)--(Sa);
  \draw[->] (9,.3)--(Sb);

  \draw[->] (Sa)--(10.8,-.3);
  \draw[->] (Sb)--(10.8,.3);

  \draw (15,0) node (lb1) {if $\op{sgn}(\gamma_j \cap \eta_k) = -$ and $\op{col}(\eta_k) - \op{col}(\gamma_j) \equiv 1$};

 \end{tikzpicture}
\]
\[
\begin{tikzpicture}

  \draw (7,0) node (E1) {$E_k$};
  \draw (8.5,0) node (=0) {$:=$};

  \draw[->] (6.2,-.3)--(E1);
  \draw[->] (6.2,.3)--(E1);
  \draw[->] (E1)--(7.8,-.3);
  \draw[->] (E1)--(7.8,.3);

  \draw[->] (9,-.3)--(10.8,-.3);
  \draw[->] (9,.3)--(10.8,.3);

  \draw (15,0) node (lb1) {if $\op{sgn}(\gamma_j \cap \eta_k) = +$ and $\op{col}(\eta_k) - \op{col}(\gamma_j) \equiv -1$};

 \end{tikzpicture}
\]
\[
\begin{tikzpicture}

  \draw (7,0) node (E1) {$E_k$};
  \draw (8.5,0) node (=0) {$:=$};

  \draw[->] (6.2,-.3)--(E1);
  \draw[->] (6.2,.3)--(E1);
  \draw[->] (E1)--(7.8,-.3);
  \draw[->] (E1)--(7.8,.3);

  \draw (9.8,-.3) node (Sa) {$S_\gamma$};
  \draw (9.8,.3) node (Sb) {$S_\gamma$};

  \draw[->] (9,-.3)--(Sa);
  \draw[->] (9,.3)--(Sb);

  \draw[->] (Sa) to [out=0,in=180,looseness=1] (10.8,.3);
  \draw[->] (Sb) to [out=0,in=180,looseness=1] (10.8,-.3);

  \draw (15,0) node (lb1) {if $\op{sgn}(\gamma_j \cap \eta_k) = -$ and $\op{col}(\eta_k) - \op{col}(\gamma_j) \equiv -1$};

 \end{tikzpicture}
\]
The presence of these cases comes from the relative order of the output legs of $\Delta_{\eta_k}$ corresponding to the original intersection $\gamma_i \cap \eta_k$ and the cloned intersection $(\gamma_i \# \gamma_j) \cap \eta_k$, and the relation of this order to the sign of the original intersection.

Returning to the proof of (d), to show that the tensor diagrams $\langle T\rangle_{\mathcal{H}}$ and $\langle T'\rangle_{\mathcal{H}}$ specify the same scalar, it suffices to demonstrate that $R_H$ and $R'_H$ denote the same tensor. To accomplish this, we first observe the following identities. 
\begin{equation} \label{eqn:handle_slide_id_1}
\begin{tikzpicture}

  \draw (0,0) node (S1) {$S_\gamma^{t_l}$};
  \draw (1,0) node (P1) {$\bullet$};
  \draw (2,0) node (D1) {$\Delta_{\eta_l}$};

  \draw[->] (-.8,0)--(S1);
  \draw[->] (S1)--(P1);
  \draw[<-] (P1)--(D1);
  \draw[<-] (D1)--(2.8,0);

  \draw (3.5,0) node (=1) {$=$};

  \draw (5,0) node (M3) {$M_\gamma$};
  \draw (6,0) node (P3) {$\bullet$};
  \draw (7,0) node (S3) {$S_{\eta_l}^{t_l}$};

  \draw[->] (4.2,0)--(M3);
  \draw[->] (M3)--(P3);
  \draw[<-] (P3)--(S3);
  \draw[<-] (S3)--(7.8,0);
 \end{tikzpicture}\end{equation}
 \begin{equation}\label{eqn:handle_slide_id_2}
 \begin{tikzpicture}

  \draw (0,0) node (E1) {$E_l$};
  \draw (1,.3) node (P1a) {$\bullet$};
  \draw (1,-.3) node (P1b) {$\bullet$};
  \draw (2,0) node (D1) {$\Delta_{\eta_l}$};

  \draw[->] (-.8,.3)--(E1);
  \draw[->] (-.8,-.3)--(E1);
  \draw[->] (E1)--(P1a);
  \draw[->] (E1)--(P1b);
  \draw[<-] (P1a)--(D1);
  \draw[<-] (P1b)--(D1);
  \draw[<-] (D1)--(2.8,0);

  \draw (3.5,0) node (=1) {$=$};

  \draw (5,0) node (M2) {$M_\gamma$};
  \draw (6,0) node (P2) {$\bullet$};
  \draw (7,0) node (S2) {$S_{\eta_l}^{t_l}$};

  \draw[->] (4.2,.3)--(M2);
  \draw[->] (4.2,-.3)--(M2);
  \draw[->] (M2)--(P2);
  \draw[<-] (P2)--(S2);
  \draw[<-] (S2)--(7.8,0);
 
 \end{tikzpicture}\end{equation}
 \begin{equation}\label{eqn:handle_slide_id_3}
 \begin{tikzpicture}

  \draw (1,-.6) node (C1) {$C_\gamma$};
  \draw (0,-1) node (D1) {$\Delta_\gamma$};
  \draw (-.6,-1) node (dots1) {$\dots$};

  \draw (3,-.6) node (C2) {$C_\gamma$};
  \draw (4,-1) node (D2) {$\Delta_\gamma$};

  \draw (0,-2) node (Ma) {$M_\gamma$};
  \draw (1,-2) node (Mb) {$M_\gamma$};
  \draw (2,-2) node (Mc) {$M_\gamma$};
  \draw (3,-2) node (dotsd) {$\dots$};
  \draw (4,-2) node (Me) {$M_\gamma$};

  \draw[->] (D1)--(-.5,-1.6);
  \draw[->] (D1)--(-.7,-1.3);

  \draw[->] (C1)--(D1);
  \draw[->] (D1)--(Ma);
  \draw[->] (D1)--(Mb);
  \draw[->] (D1)--(Mc);
  \draw[->] (D1)--(Me);

  \draw[->] (C2)--(D2);
  \draw[->] (D2)--(Ma);
  \draw[->] (D2)--(Mb);
  \draw[->] (D2)--(Mc);
  \draw[->] (D2)--(Me);

  \draw[->] (Ma)--(0,-2.7);
  \draw[->] (Mb)--(1,-2.7);
  \draw[->] (Mc)--(2,-2.7);
  \draw[->] (dotsd)--(3,-2.7);
  \draw[->] (Me)--(4,-2.7);

  \draw (6,-1.5) node (=) {$=$};

  \draw (9,-.6) node (C1) {$C_\gamma$};
  \draw (8,-1) node (D1) {$\Delta_\gamma$};
  \draw (7.4,-1) node (dots1) {$\dots$};

  \draw (11,-.6) node (C2) {$C_\gamma$};
  \draw (12,-1) node (D2) {$\Delta_\gamma$};

  \draw (8,-2) node (Ma) {$M_\gamma$};
  \draw (9,-2) node (Mb) {$M_\gamma$};
  \draw (10,-2) node (Mc) {$M_\gamma$};
  \draw (11,-2) node (dotsd) {$\dots$};
  \draw (12,-2) node (Me) {$M_\gamma$};

  \draw[->] (D1)--(7.5,-1.6);
  \draw[->] (D1)--(7.3,-1.3);

  \draw[->] (C1)--(D1);

  \draw[->] (C2)--(D2);
  \draw[->] (D2)--(Ma);
  \draw[->] (D2)--(Mb);
  \draw[->] (D2)--(Mc);
  \draw[->] (D2)--(Me);

  \draw[->] (Ma)--(8,-2.7);
  \draw[->] (Mb)--(9,-2.7);
  \draw[->] (Mc)--(10,-2.7);
  \draw[->] (dotsd)--(11,-2.7);
  \draw[->] (Me)--(12,-2.7);
 \end{tikzpicture}\end{equation} 
Here we emphasize again that $\to M_\gamma \to$ and $\to \Delta_{\eta_l} \to$ denote, by convention, the identity tensor (see Notation \ref{not:product_coproduct_abbreviated_notation}). The equations (\ref{eqn:handle_slide_id_1})-(\ref{eqn:handle_slide_id_3}) are consequences of the diagrammatic Hopf algebra and Hopf triplet axioms from Definition \ref{def:hopf_algebra} and \ref{def:hopf_triplet}. We will explain them in detail in Lemma \ref{lem:handle_slide_lemma} below.

We may now apply (\ref{eqn:handle_slide_id_1}), (\ref{eqn:handle_slide_id_3}) and (\ref{eqn:handle_slide_id_2}) in that order to perform the following manipulation transforming $R_H$ into $R'_H$. This finishes the proof of the handle slide property.
\[
\begin{tikzpicture}

  \draw (0,1.6) node (C1) {$C_\gamma$};
  \draw (0,-.3) node (D1) {$\Delta_\gamma$};
  \draw (-.3,-1.2) node (dots1) {$\dots$};

  \draw (3.8,1.6) node (C2) {$C_\gamma$};
  \draw (4.8,1.4) node (D2) {$\Delta_\gamma$};

  \draw (1.4,-.3) node (P2a) {$\bullet$};
  \draw (2.6,-.3) node (P2b) {$\bullet$};
  \draw (3.8,-.3) node (dots2c1) {$\dots$};
  \draw (5,-.3) node (P2d) {$\bullet$};

  \draw (1.4,-1.1) node (D2a) {$\Delta_{\eta_1}$};
  \draw (2.6,-1.1) node (D2b) {$\Delta_{\eta_2}$};
  \draw (3.8,-1.1) node (dots2c) {$\dots$};
  \draw (5,-1.1) node (D2d) {$\Delta_{\eta_b}$};

  \draw (1.4,.5) node (S2a) {$S_\gamma^{t_1}$};
  \draw (2.6,.5) node (S2b) {$S_\gamma^{t_2}$};
  \draw (3.8,.5) node (dots2c2) {$\dots$};
  \draw (5,.5) node (S2d) {$S_\gamma^{t_b}$};

  \draw[->] (C1)--(D1);

  \draw[->] (D1)--(-.8,-.7);
  \draw[->] (D1)--(-.7,-1.1);
  \draw[->] (D1)--(0,-1.2);

  \draw[->] (C2)--(D2);

  \draw[->] (D2)--(S2a);
  \draw[->] (D2)--(S2b);
  \draw[->] (D2)--(S2d);

  \draw[->] (S2a)--(P2a);
  \draw[->] (S2b)--(P2b);
  \draw[->] (S2d)--(P2d);

  \draw[<-] (P2a)--(D2a);
  \draw[<-] (P2b)--(D2b);
  \draw[<-] (P2d)--(D2d);

  \draw[<-] (D2a)--(1.4,-1.7);
  \draw[<-] (D2b)--(2.6,-1.7);
  \draw[<-] (D2d)--(5,-1.7);

  %2nd_part

  \draw (6,.2) node (=1) {$=$};

  \draw (7.5,1.6) node (C1) {$C_\gamma$};
  \draw (7.5,-.3) node (D1) {$\Delta_\gamma$};
  \draw (7.2,-1.2) node (dots1) {$\dots$};

  \draw (11.3,1.6) node (C2) {$C_\gamma$};
  \draw (12.3,1.4) node (D2) {$\Delta_\gamma$};

  \draw (8.9,-.3) node (P2a) {$\bullet$};
  \draw (10.1,-.3) node (P2b) {$\bullet$};
  \draw (11.3,-.3) node (dots2c1) {$\dots$};
  \draw (12.5,-.3) node (P2d) {$\bullet$};

  \draw (8.9,.5) node (M2a) {$M_\gamma$};
  \draw (10.1,.5) node (M2b) {$M_\gamma$};
  \draw (11.3,.5) node (dots2c) {$\dots$};
  \draw (12.5,.5) node (M2d) {$M_\gamma$};

  \draw (8.9,-1.1) node (S2a) {$S_{\eta_1}^{t_1}$};
  \draw (10.1,-1.1) node (S2b) {$S_{\eta_2}^{t_2}$};
  \draw (11.3,-1.1) node (dots2c2) {$\dots$};
  \draw (12.5,-1.1) node (S2d) {$S_{\eta_b}^{t_b}$};

  \draw[->] (C1)--(D1);

  \draw[->] (D1)--(6.7,-.7);
  \draw[->] (D1)--(6.8,-1.1);
  \draw[->] (D1)--(7.5,-1.2);

  \draw[->] (C2)--(D2);

  \draw[->] (D2)--(M2a);
  \draw[->] (D2)--(M2b);
  \draw[->] (D2)--(M2d);

  \draw[->] (M2a)--(P2a);
  \draw[->] (M2b)--(P2b);
  \draw[->] (M2d)--(P2d);

  \draw[<-] (P2a)--(S2a);
  \draw[<-] (P2b)--(S2b);
  \draw[<-] (P2d)--(S2d);

  \draw[<-] (S2a)--(8.9,-1.7);
  \draw[<-] (S2b)--(10.1,-1.7);
  \draw[<-] (S2d)--(12.5,-1.7);

  \draw (13.4,.2) node (=1) {$=$};
 \end{tikzpicture}
\] 
\[
\begin{tikzpicture}

  %2nd_part

  \draw (6,.2) node (=1) {$=$};

  \draw (0,1.6) node (C1) {$C_\gamma$};
  \draw (0,-.3) node (D1) {$\Delta_\gamma$};
  \draw (-.3,-1.2) node (dots1) {$\dots$};

  \draw (3.8,1.6) node (C2) {$C_\gamma$};
  \draw (4.8,1.4) node (D2) {$\Delta_\gamma$};

  \draw (1.3,1.4) node (D3) {$\Delta_\gamma$};

  \draw (1.4,-.3) node (P2a) {$\bullet$};
  \draw (2.6,-.3) node (P2b) {$\bullet$};
  \draw (3.8,-.3) node (dots2c1) {$\dots$};
  \draw (5,-.3) node (P2d) {$\bullet$};

  \draw (1.4,.5) node (M2a) {$M_\gamma$};
  \draw (2.6,.5) node (M2b) {$M_\gamma$};
  \draw (3.8,.5) node (dots2c) {$\dots$};
  \draw (5,.5) node (M2d) {$M_\gamma$};

  \draw (1.4,-1.1) node (S2a) {$S_{\eta_1}^{t_1}$};
  \draw (2.6,-1.1) node (S2b) {$S_{\eta_2}^{t_2}$};
  \draw (3.8,-1.1) node (dots2c2) {$\dots$};
  \draw (5,-1.1) node (S2d) {$S_{\eta_b}^{t_b}$};

  \draw[->] (C1)--(D1);

  \draw[->] (D1)--(-.8,-.7);
  \draw[->] (D1)--(-.7,-1.1);
  \draw[->] (D1)--(0,-1.2);

  \draw[->] (C2)--(D2);

  \draw[->] (D1)--(D3);

  \draw[->] (D2)--(M2a);
  \draw[->] (D2)--(M2b);
  \draw[->] (D2)--(M2d);

  \draw[->] (M2a)--(P2a);
  \draw[->] (M2b)--(P2b);
  \draw[->] (M2d)--(P2d);

  \draw[<-] (P2a)--(S2a);
  \draw[<-] (P2b)--(S2b);
  \draw[<-] (P2d)--(S2d);

  \draw[->] (D3)--(M2a);
  \draw[->] (D3)--(M2b);
  \draw[->] (D3)--(M2d);

  \draw[<-] (S2a)--(1.4,-1.7);
  \draw[<-] (S2b)--(2.6,-1.7);
  \draw[<-] (S2d)--(5,-1.7);

%2nd part

  \draw (7.5,1.6) node (C1) {$C_\gamma$};
  \draw (7.5,-.3) node (D1) {$\Delta_\gamma$};
  \draw (7.2,-1.2) node (dots1) {$\dots$};

  \draw (11.3,1.6) node (C2) {$C_\gamma$};
  \draw (12.3,1.4) node (D2) {$\Delta_\gamma$};

  \draw (8.9,1.4) node (D3) {$\Delta_\gamma$};

  \draw (9.1,-.3) node (P2a) {$\bullet$};
  \draw (10.3,-.3) node (P2b) {$\bullet$};
  \draw (11.3,-.3) node (dots2c1) {$\dots$};
  \draw (12.7,-.3) node (P2d) {$\bullet$};

  \draw (8.7,-.3) node (P3a) {$\bullet$};
  \draw (9.9,-.3) node (P3b) {$\bullet$};
  \draw (12.3,-.3) node (P3d) {$\bullet$};

  \draw (8.9,-1.1) node (D2a) {$\Delta_{\eta_1}$};
  \draw (10.1,-1.1) node (D2b) {$\Delta_{\eta_{2}}$};
  \draw (11.3,-1.1) node (dots2c) {$\dots$};
  \draw (12.5,-1.1) node (D2d) {$\Delta_{\eta_b}$};

  \draw (8.9,.5) node (F2a) {$E_1$};
  \draw (10.1,.5) node (F2b) {$E_2$};
  \draw (11.3,.5) node (dots2c2) {$\dots$};
  \draw (12.5,.5) node (F2d) {$E_b$};

  \draw[->] (C1)--(D1);

  \draw[->] (D1)--(6.7,-.7);
  \draw[->] (D1)--(6.8,-1.1);
  \draw[->] (D1)--(7.5,-1.2);

  \draw[->] (C2)--(D2);

  \draw[->] (D1)--(D3);

  \draw[->] (D3)--(F2a);
  \draw[->] (D3)--(F2b);
  \draw[->] (D3)--(F2d);

  \draw[->] (D2)--(F2a);
  \draw[->] (D2)--(F2b);
  \draw[->] (D2)--(F2d);

  \draw[->] (F2a)--(P2a);
  \draw[->] (F2b)--(P2b);
  \draw[->] (F2d)--(P2d);

  \draw[->] (F2a)--(P3a);
  \draw[->] (F2b)--(P3b);
  \draw[->] (F2d)--(P3d);

  \draw[<-] (P2a)--(D2a);
  \draw[<-] (P2b)--(D2b);
  \draw[<-] (P2d)--(D2d);

  \draw[<-] (P3a)--(D2a);
  \draw[<-] (P3b)--(D2b);
  \draw[<-] (P3d)--(D2d);

  \draw[<-] (D2a)--(8.9,-1.7);
  \draw[<-] (D2b)--(10.1,-1.7);
  \draw[<-] (D2d)--(12.5,-1.7);

 \end{tikzpicture}
\]
Having proven the handle slide property, we have demonstrated properties (a)-(d) listed in the proposition statement and so concluded the proof of Proposition \ref{prop:properties_of_bracket}.\end{proof}

\begin{lemma} \label{lem:handle_slide_lemma} The formulas (\ref{eqn:handle_slide_id_1}), (\ref{eqn:handle_slide_id_2}) and (\ref{eqn:handle_slide_id_3}) in the proof of Proposition \ref{prop:properties_of_bracket}(d) are valid and follow from Definition \ref{def:hopf_algebra} and \ref{def:hopf_triplet}.
\end{lemma}

\begin{proof} For (\ref{eqn:handle_slide_id_1}), again recall that $\to M_\gamma \to$ and $\to \Delta_\gamma \to $ both denote the identity and $S^{t_l}_\gamma$ is either $\op{Id}$ and $S_\gamma$. Thus (\ref{eqn:handle_slide_id_1}) is (in the non-trivial case) simply the fact that skew pairing intertwine antipodes.

Next we handle (\ref{eqn:handle_slide_id_2}). Assume (without loss of generality) that $\gamma_j = \alpha_j$ is an $\alpha$ curve, so that $\op{col}(\eta_k) - \op{col}(\gamma_j) \equiv 1 \mod 3$ implies $\eta_k = \beta_k$ is a $\beta$ curve, and similarly $\op{col}(\eta_k) - \op{col}(\gamma_j) \equiv - 1 \mod 3$ implies $\eta_k = \kappa_k$ is a $\kappa$ curve. Then the identity (\ref{eqn:handle_slide_id_2}) for the four cases for $E_i$ translates to the following identities.

\[\begin{tikzpicture}

  \draw (-2,.3) node (i1a) {};
  \draw (-2,-.3) node (i1b) {};
  \draw (.6,.3) node (P1a) {$\langle-\rangle_{\alpha\beta}$};
  \draw (.6,-.3) node (P1b) {$\langle-\rangle_{\alpha\beta}$};
  \draw (2,0) node (D1) {$\Delta_\beta$};
  \draw (2.8,0) node (o1) {};

  \draw[->] (i1b)--(P1a);
  \draw[->] (i1a)--(P1b);
  \draw[<-] (P1a)--(D1);
  \draw[<-] (P1b)--(D1);
  \draw[<-] (D1)--(o1);

  \draw (3.5,0) node (=1) {$=$};

  \draw (4,.3) node (i2a) {};
  \draw (4,-.3) node (i2b) {};
  \draw (5,0) node (M2) {$M_\alpha$};
  \draw (6.4,0) node (P2) {$\langle-\rangle_{\alpha\beta}$};
  \draw (8.4,0) node (o2) {};

  \draw[->] (i2a)--(M2);
  \draw[->] (i2b)--(M2);
  \draw[->] (M2)--(P2);
  \draw[<-] (P2)--(o2);

 \end{tikzpicture}\]

 \[\begin{tikzpicture}

  \draw (-2,.3) node (i1a) {};
  \draw (-2,-.3) node (i1b) {};
  \draw (-1,.3) node (S1a) {$S_\alpha$};
  \draw (-1,-.3) node (S1b) {$S_\alpha$};
  \draw (.6,.3) node (P1a) {$\langle-\rangle_{\alpha\beta}$};
  \draw (.6,-.3) node (P1b) {$\langle-\rangle_{\alpha\beta}$};
  \draw (2,0) node (D1) {$\Delta_\beta$};
  \draw (2.8,0) node (o1) {};

  \draw[->] (i1a)--(S1a);
  \draw[->] (i1b)--(S1b);
  \draw[->] (S1a) to [out=0,in=180,looseness=1] (P1b);
  \draw[->] (S1b) to [out=0,in=180,looseness=1] (P1a);
  \draw[<-] (P1a)--(D1);
  \draw[<-] (P1b)--(D1);
  \draw[<-] (D1)--(o1);

  \draw (3.5,0) node (=1) {$=$};

  \draw (4,.3) node (i2a) {};
  \draw (4,-.3) node (i2b) {};
  \draw (5,0) node (M2) {$M_\alpha$};
  \draw (6.4,0) node (P2) {$\langle-\rangle_{\alpha\beta}$};
  \draw (7.7,0) node (S2) {$S_\beta$};
  \draw (8.4,0) node (o2) {};

  \draw[->] (i2a)--(M2);
  \draw[->] (i2b)--(M2);
  \draw[->] (M2)--(P2);
  \draw[->] (P2)--(S2);
  \draw[<-] (S2)--(o2);
 
 \end{tikzpicture}\]

 \[\begin{tikzpicture}

  \draw (-2,.3) node (i1a) {};
  \draw (-2,-.3) node (i1b) {};
  \draw (.6,.3) node (P1a) {$\langle-\rangle_{\kappa\alpha}$};
  \draw (.6,-.3) node (P1b) {$\langle-\rangle_{\kappa\alpha}$};
  \draw (2,0) node (D1) {$\Delta_\kappa$};
  \draw (2.8,0) node (o1) {};

  \draw[->] (i1a)--(P1a);
  \draw[->] (i1b)--(P1b);
  \draw[<-] (P1a)--(D1);
  \draw[<-] (P1b)--(D1);
  \draw[<-] (D1)--(o1);

  \draw (3.5,0) node (=1) {$=$};

  \draw (4,.3) node (i2a) {};
  \draw (4,-.3) node (i2b) {};
  \draw (5,0) node (M2) {$M_\alpha$};
  \draw (6.4,0) node (P2) {$\langle-\rangle_{\kappa\alpha}$};
  \draw (8.4,0) node (o2) {};

  \draw[->] (i2a)--(M2);
  \draw[->] (i2b)--(M2);
  \draw[->] (M2)--(P2);
  \draw[<-] (P2)--(o2);
 
 \end{tikzpicture}\]

 \[\begin{tikzpicture}

  \draw (-2,.3) node (i1a) {};
  \draw (-2,-.3) node (i1b) {};
  \draw (-1,.3) node (S1a) {$S_\alpha$};
  \draw (-1,-.3) node (S1b) {$S_\alpha$};
  \draw (.6,.3) node (P1a) {$\langle-\rangle_{\kappa\alpha}$};
  \draw (.6,-.3) node (P1b) {$\langle-\rangle_{\kappa\alpha}$};
  \draw (2,0) node (D1) {$\Delta_\kappa$};
  \draw (2.8,0) node (o1) {};

  \draw[->] (i1a)--(S1a);
  \draw[->] (i1b)--(S1b);
  \draw[->] (S1a) to [out=0,in=180,looseness=1] (P1a);
  \draw[->] (S1b) to [out=0,in=180,looseness=1] (P1b);
  \draw[<-] (P1a)--(D1);
  \draw[<-] (P1b)--(D1);
  \draw[<-] (D1)--(o1);

  \draw (3.5,0) node (=1) {$=$};

  \draw (4,.3) node (i2a) {};
  \draw (4,-.3) node (i2b) {};
  \draw (5,0) node (M2) {$M_\alpha$};
  \draw (6.4,0) node (P2) {$\langle-\rangle_{\kappa\alpha}$};
  \draw (7.7,0) node (S2) {$S_\kappa$};
  \draw (8.4,0) node (o2) {};

  \draw[->] (i2a)--(M2);
  \draw[->] (i2b)--(M2);
  \draw[->] (M2)--(P2);
  \draw[->] (P2)--(S2);
  \draw[<-] (S2)--(o2);
 
 \end{tikzpicture}\]
 The first and third identities follow from the fact that the pairings induce Hopf algebra morphisms $H_\alpha \to H_\beta^{*,\op{cop}}$ and $H_\kappa \to H_\alpha^{*,\op{cop}}$ (see Definition \ref{def:hopf_triplet}(a)). The second and the fourth also follow from this fact, after commuting the antipodes past the pairings and using the anti-homomorphism property of the antipodes.

Finally, we address (\ref{eqn:handle_slide_id_3}). This is just a fact about Hopf algebras, so let $H = (H,M,\eta,\Delta,\epsilon,S)$ be a Hopf algebra. Then we compute as follows.

\[\begin{tikzpicture}

  \draw (-.2,-.6) node (C1) {$C$};
  \draw (0,-1) node (D1) {$\Delta$};
  \draw (-1.6,-1) node (dots1) {$\dots$};

  \draw (-1,-1) node (D3) {$\Delta$};

  \draw (3,-.6) node (C2) {$C$};
  \draw (4,-1) node (D2) {$\Delta$};

  \draw (0,-2) node (Ma) {$M$};
  \draw (1,-2) node (Mb) {$M$};
  \draw (2,-2) node (Mc) {$M$};
  \draw (3,-2) node (dotsd) {$\dots$};
  \draw (4,-2) node (Me) {$M$};

  \draw[->] (D3)--(-1.5,-1.6);
  \draw[->] (D3)--(-1.7,-1.3);

  \draw[->] (C1)--(D3);
  \draw[->] (D3)--(D1);
  \draw[->] (D1)--(Ma);
  \draw[->] (D1)--(Mb);
  \draw[->] (D1)--(Mc);
  \draw[->] (D1)--(Me);

  \draw[->] (C2)--(D2);
  \draw[->] (D2)--(Ma);
  \draw[->] (D2)--(Mb);
  \draw[->] (D2)--(Mc);
  \draw[->] (D2)--(Me);

  \draw[->] (Ma)--(0,-2.7);
  \draw[->] (Mb)--(1,-2.7);
  \draw[->] (Mc)--(2,-2.7);
  \draw[->] (dotsd)--(3,-2.7);
  \draw[->] (Me)--(4,-2.7);

  \draw (5,-1.5) node (=) {$=$};
 \end{tikzpicture}\begin{tikzpicture}

  \draw (-.2,-.6) node (C1) {$C$};
  \draw (1,-1) node (M1) {$M$};
  \draw (-1.6,-1) node (dots1) {$\dots$};

  \draw (-1,-1) node (D3) {$\Delta$};

  \draw (3,-.6) node (C2) {$C$};
  \draw (4,-1) node (D2) {$\Delta$};

  \draw (0,-2) node (Ma) {$M$};
  \draw (1,-2) node (Mb) {$M$};
  \draw (2,-2) node (Mc) {$M$};
  \draw (3,-2) node (dotsd) {$\dots$};
  \draw (4,-2) node (Me) {$M$};

  \draw[->] (D3)--(-1.5,-1.6);
  \draw[->] (D3)--(-1.7,-1.3);

  \draw[->] (C1)--(D3);
  \draw[->] (D3)--(M1);
  \draw[->] (M1)--(D2);

  \draw[->] (C2)--(M1);
  \draw[->] (D2)--(Ma);
  \draw[->] (D2)--(Mb);
  \draw[->] (D2)--(Mc);
  \draw[->] (D2)--(Me);

  \draw[->] (Ma)--(0,-2.7);
  \draw[->] (Mb)--(1,-2.7);
  \draw[->] (Mc)--(2,-2.7);
  \draw[->] (dotsd)--(3,-2.7);
  \draw[->] (Me)--(4,-2.7);

  \draw (5,-1.5) node (=) {$=$};
 \end{tikzpicture}\]

 \[\begin{tikzpicture}

  \draw (-.2,-.6) node (C1) {$C$};
  \draw (1,-1) node (e1) {$\epsilon$};
  \draw (-1.6,-1) node (dots1) {$\dots$};

  \draw (-1,-1) node (D3) {$\Delta$};

  \draw (3,-.6) node (C2) {$C$};
  \draw (4,-1) node (D2) {$\Delta$};

  \draw (0,-2) node (Ma) {$M$};
  \draw (1,-2) node (Mb) {$M$};
  \draw (2,-2) node (Mc) {$M$};
  \draw (3,-2) node (dotsd) {$\dots$};
  \draw (4,-2) node (Me) {$M$};

  \draw[->] (D3)--(-1.5,-1.6);
  \draw[->] (D3)--(-1.7,-1.3);

  \draw[->] (C1)--(D3);
  \draw[->] (D3)--(e1);

  \draw[->] (C2)--(D2);
  \draw[->] (D2)--(Ma);
  \draw[->] (D2)--(Mb);
  \draw[->] (D2)--(Mc);
  \draw[->] (D2)--(Me);

  \draw[->] (Ma)--(0,-2.7);
  \draw[->] (Mb)--(1,-2.7);
  \draw[->] (Mc)--(2,-2.7);
  \draw[->] (dotsd)--(3,-2.7);
  \draw[->] (Me)--(4,-2.7);
  \draw (5,-1.5) node (=) {$=$};
 \end{tikzpicture}
 \begin{tikzpicture}

  \draw (-.2,-.6) node (C1) {$C$};
  \draw (-1.6,-1) node (dots1) {$\dots$};

  \draw (-1,-1) node (D3) {$\Delta$};

  \draw (3,-.6) node (C2) {$C$};
  \draw (4,-1) node (D2) {$\Delta$};

  \draw (0,-2) node (Ma) {$M$};
  \draw (1,-2) node (Mb) {$M$};
  \draw (2,-2) node (Mc) {$M$};
  \draw (3,-2) node (dotsd) {$\dots$};
  \draw (4,-2) node (Me) {$M$};

  \draw[->] (D3)--(-1.5,-1.6);
  \draw[->] (D3)--(-1.7,-1.3);

  \draw[->] (C1)--(D3);

  \draw[->] (C2)--(D2);
  \draw[->] (D2)--(Ma);
  \draw[->] (D2)--(Mb);
  \draw[->] (D2)--(Mc);
  \draw[->] (D2)--(Me);

  \draw[->] (Ma)--(0,-2.7);
  \draw[->] (Mb)--(1,-2.7);
  \draw[->] (Mc)--(2,-2.7);
  \draw[->] (dotsd)--(3,-2.7);
  \draw[->] (Me)--(4,-2.7);
 \end{tikzpicture}\]
 Here we use the bialgebra property of Hopf algebras (see Definition \ref{def:hopf_algebra}(c)) on the first line and the fact that the cotrace $C$ is a cointegral (see Definition \ref{def:trace_cotrace} and Lemma \ref{lem:trace_is_integral}) to move from the first to second line. This concludes the proof of the Lemma. \end{proof}

\subsection{Main Definition and Properties} \label{subsec:main_definition_and_properties} The definition of the invariant itself is straightforward, as it is simply a normalization of the bracket discussed in \S \ref{subsec:trisection_bracket}.  However, the invariant requires a small constraint on the Hopf triplets we use, which is easily verified in practice.

\begin{definition}[Admissible] Let $\mathcal{H}$ be a Hopf triplet over field $k$ of characteristic zero. 
We say that $H$ is \emph{trisection admissible} if the equation $\xi^3 - \langle T_{\op{st}}\rangle_{\mathcal{H}} = 0$ admits a solution $\xi \in k^\times$ in the group of units $k^\times$ of $k$, where $T_{\op{st}}$ is the standard genus-$3$ (stabilizing) trisection of $S^4$ (see Figure \ref{fig:stabilized_sphere_trisection}). 
\end{definition}

\begin{definition} (Trisection Invariant) \label{def:trisection_kuperberg_invariant}
Let $\mathcal{H}$ be a trisection admissible Hopf triplet and fix a root $\xi \in k^\times$ of $\langle T_{\op{st}}\rangle_{\mathcal{H}}$.
Let $X$ be a smooth, closed, oriented 4-manifold and $T$ be a trisection diagram for $X$.  The \emph{trisection invariant} $\tau_{\mathcal{H},\xi}(X;T) \in k$ is defined to be 
\begin{align}
\tau_{\mathcal{H},\xi}(X;T) = \xi^{-g(T)} \langle T \rangle_{\mathcal{H}}\,.
\end{align}
\end{definition}
As a consequence of Proposition \ref{prop:properties_of_bracket}, we have the following invariance result, which is one of the main theorems of the paper.

\begin{theorem} \label{thm:main_invariance_thm} (Invariance) The trisection invariant $\tau_{\mathcal{H},\xi}(X;T)$ is invariant under all trisection moves and isotopy of $T$.
\end{theorem}
\begin{proof} The genus $g(T)$ is obviously invariant under isotopy and handle-slides, and is additive under connect sum. Thus Proposition \ref{prop:properties_of_bracket}(b) and (e) imply that $\tau_{\mathcal{H}}(X;T)$ is invariant under isotopies and handle-slides. For stabilizations, we observe that:
\[\xi^{-g(T \# T_{\op{st}})} \langle T \# T_{\op{st}} \rangle_{\mathcal{H}} = \xi^{-g(T)} \langle T_{\op{st}}\rangle^{-1} \langle T \rangle_{\mathcal{H}} \langle T_{\op{st}}\rangle_{\mathcal{H}} = \xi^{-g(T)}\langle T \rangle_{\mathcal{H}}\]
Here we use the multiplicativity of $\langle -\rangle_{\mathcal{H}}$ under connect sum, see Proposition \ref{prop:properties_of_bracket}(d). Thus $\tau_{\mathcal{H},\xi}(X;T \# T_{\op{st}}) = \tau_{\mathcal{H},\xi}(X;T)$, and  $\tau_{\mathcal{H},\xi}$ is invariant under stabilization.\end{proof}

\begin{corollary} \label{cor:invariance}  $\tau_{\mathcal{H},\xi}(X) := \tau_{\mathcal{H},\xi}(X;T)$ is an oriented diffeomorphism invariant.
\end{corollary}

Moreover, the connect sum property of the trisection bracket implies the same property for $\tau_{\mathcal{H},\xi}(X)$.

\begin{proposition}[Connect Sum] Let $X$ and $X'$ be smooth, closed $4$-manifolds and let $\mathcal{H}$ be a Hopf triplet. Then $\tau_{\mathcal{H},\xi}(X \# X') = \tau_{\mathcal{H},\xi}(X)\tau_{\mathcal{H},\xi}(X')$.
\end{proposition}

In this paper, we will only compute examples of this invariant for Hopf triplets over $k = \R$ or $\C$. Thus we will use the following abbreviation for the rest of the paper. 

\begin{convention} \label{con:choice_of_root_over_R_or_C} Let $\mathcal{H}$ be a trisection admissible Hopf triplet over $\R$ or $\C$ such that $\langle T_{\op{st}}\rangle_{\mathcal{H}} \in \R_+$ is positive and real. We fix the convention that
\[
\tau_{\mathcal{H}}(X) := \tau_{\mathcal{H},\xi}(X)
\]
where $\xi$ is the unique real cube root of $\langle T_{\op{st}}\rangle_{\mathcal{H}}$.
\end{convention}

In all of the cases calculated in \S \ref{sec:examples_and_calculations} and the setting of \S \ref{sec:CY_dichro}, the conditions for Convention \ref{con:choice_of_root_over_R_or_C} hold.

\begin{remark}[Pseudotrisections] \label{rem:pseudotrisections} Consider a triple $PT = (\Sigma, \alpha, \beta, \gamma)$ which satisfies all of the properties of an oriented trisection diagram in Definition \ref{def:trisection_diagram} except for (b)(ii).  This is called a psuedotrisection diagram~\cite{lambert2019trisections}, and does not generally correspond to a trisection of a 4-manifold.  Applying $\tau_{\mathcal{H}}$ to a psuedotrisection diagram $PT$ outputs a scalar which is invariant under trisection moves of $PT$.
\end{remark}

\subsection{3-Manifold Invariant from Hopf doublets} \label{subsec:generalized_kuperberg_invt}
Motivated by the trisection invariant, we provide a version of the Kuperberg invariant of 3-manifolds using involutory Hopf doublets. 

\begin{definition}[Kuperberg Invariant from Hopf doublets] Let $\mathcal{H}$ be an involutory Hopf doublet, $Y$ be a $3$-manifold and $S$ be a Heegaard diagram for $Y$. 

The \emph{(generalized) Kuperberg bracket} $\langle S\rangle_{\mathcal{H}}$ is defined as in Definition \ref{def:trisection_bracket}, using only the $\alpha$ and $\beta$ curves in that discussion. The \emph{Kuperberg invariant} $\tau_{\mathcal{H}}(Y;S)$ for a Hopf doublet $\mathcal{H}$ with $\langle S_{\op{st}}\rangle_{\mathcal{H}} \neq 0$ is the normalization
\[
\tau_{\mathcal{H}}(Y;S) := \langle S_{\op{st}}\rangle_{\mathcal{H}}^{-g(S)} \cdot \langle S\rangle_{\mathcal{H}}
\]
Here $S_{\op{st}}$ denotes the standard (stabilizing) genus $1$ Heegaard splitting of $S^3$.
\end{definition}

Using the arguments in \S \ref{subsec:trisection_bracket} (more specifically, Proposition \ref{prop:properties_of_bracket}) and \S \ref{subsec:main_definition_and_properties}, we can prove an invariance theorem.

\begin{theorem} The  Kuperberg invariant $\tau_{\mathcal{H}}(Y) := \tau_{\mathcal{H}}(Y;S)$ is independent of the choice of Heegaard splitting, and satisfies $\tau_{\mathcal{H}}(Y \# Y') = \tau_{\mathcal{H}}(Y) \cdot \tau_{\mathcal{H}}(Y')$.
\end{theorem}

\noindent The original Kuperberg invariant $\#(Y,H)$ can be recovered by applying this construction to the standard doublet $\mathcal{H} = (H,H^{*,\op{op}},\langle-\rangle)$ of a Hopf algebra $H$.

\vspace{3pt}

The invariant $\tau_{\mathcal{H}}$ of a general doublet $\mathcal{H} = (H_{\alpha}, H_{\beta}, \langle - \rangle)$ is, in fact, equivalent to that of a single Hopf algebra by the following construction. Consider the Hopf algebra ideals
\[
I_\alpha :=\op{ker}(H_\alpha \to H^{*,\op{cop}}_\beta) \qquad\text{and}\qquad I_\beta := \op{ker}(H_\beta \to H^{*,\op{op}}_\beta)
\]
Furthermore, consider the quotient Hopf algebras $G_\alpha := H_\alpha/I_\alpha$ and $G_\beta := H_\beta/I_\beta$. 

\begin{proposition} \label{prop:generalized_Kup_is_Kup} For any closed $3$-manifold $Y$, we have $\tau_{\mathcal{H}}(Y) = \#(Y,G_\alpha)$.
\end{proposition}

\begin{proof} The pairing on $H_{\alpha} \otimes H_{\beta}$ descends to a unique pairing $[ - ]$ on $G_\alpha \otimes G_\beta$ satisfying
\begin{equation} \label{eqn:degenerate_pairing_id}\rightarrow \pi_\alpha \rightarrow [-] \leftarrow \pi_\beta \leftarrow \qquad = \qquad \rightarrow \langle-\rangle \leftarrow\end{equation}
Here $\pi_\gamma:H_\gamma \to G_\gamma$ for $\gamma \in \{\alpha,\beta\}$ denotes the natural Hopf morphism induced to the quotient. We can form the Hopf doublet
\[\mathcal{G} := (G_\alpha, G_\beta, [ - ])\] which is isomorphic to the Hopf doublet associated to $G_\alpha$ since $[-]$ is non-degenerate. It suffices to show that
\[\langle S\rangle_{\mathcal{H}} = \langle S\rangle_{\mathcal{G}} \qquad\text{for any Heegaard diagram $S$}\]

Thus, consider the tensor diagram defining the bracket $\langle S\rangle_{\mathcal{H}}$ using cointegrals $C_\gamma \to$ for $\gamma \in \{\alpha,\beta\}$. Perform the substitution (\ref{eqn:degenerate_pairing_id}) at every copy of $\langle-\rangle$ appearing in the diagram. Since $\pi_\gamma$ is a Hopf morphism, we may commute $\to \pi_\gamma \to$ past all coproducts $\Delta_\gamma$, replacing those coproducts with coproducts in $G_\gamma$. The resulting diagram is precisely the bracket $\langle S\rangle_{\mathcal{G}}$ with cointegrals
\[C_\gamma \to \pi_\gamma \to \qquad\text{for}\quad \gamma \in \{\alpha,\beta\}\]
The bracket $\langle S\rangle_{\mathcal{G}}$ can be calculated with respect to any choice of cointegrals, so this concludes the proof.\end{proof}

\begin{remark} In the setting of the trisection invariant, it is not clear to the authors whether the three pairings (or any one of them) in a triplet can be assumed to be non-degenerate. \end{remark}

\begin{remark} We expect that a similar treatment is possible for the invariant incorporating $\op{Spin}^c$-structures (see \cite{lopez2019kuperberg}) and the invariant for non-involutory Hopf algebras (see \cite{kuperberg1996noninvolutory}). \end{remark}

\section{Examples and Calculations} \label{sec:examples_and_calculations} The trisection invariant formulated in \S \ref{sec:trisection_kuperberg_invt} is very explicit and computer friendly. To demonstrate this, we will now perform some example calculations of the trisection invariant. 

We start (\S \ref{subsec:trisection_diagrams_of_examples}) by providing a menagerie of trisection diagrams and bracket tensor diagrams for various examples of $4$-manifolds. For more difficult trisection diagrams, we wrote a Python script (available at \cite{pythonscript2019}) to calculate the invariant using a simple combinatorial description of the trisection diagram in use (\S \ref{subsec:computational_methods_and_scripting}). We then compute the trisection invariant for triplets arising from cyclic group algebras (\ref{subsec:cyclic_triplets_and_kashaev}), in particular demonstrating that they coincide with a slight modification of Kashaev's invariant. Finally, we discuss computations for a class of Hopf triplet whose corresponding invariants do not coincide with Crane-Yetter or dichromatic invariants via Theorem \ref{thm:trisection_vs_crane_yetter_informal} (\S \ref{subsec:triplets_from_8d_algebra}).

\subsection{Trisection and Tensor Diagrams of Examples} \label{subsec:trisection_diagrams_of_examples} Here is a list of trisection diagrams for a number of standard (and some exotic) $4$-manifolds, along with the corresponding tensor diagram for the trisection bracket. These diagrams are drawn primarily from \cite{gk2016}, \cite{co2017lefschetztrisections} and \cite{lcm2018rationalsurfaces}.

\begin{remark}[Efficiency of a Diagram] \label{rmk:efficient_diagrams} Consider a simply connected, closed $4$-manifold $X$. The genus $g(T)$ of any trisection $T$ admits a lower bound of the form $g(T) \ge b_2(X)$ where $b_2(X)$ is the rank of $H_2(X)$. An \emph{efficient trisection} $T$ is a trisection for which this lower bound is an equality, i.e. $g(T) = b_2(X)$ (see \cite{lcm2018rationalsurfaces}). 

Many of the trisection diagrams given in this section are efficient in this sense, making them particularly suitable for computations of the trisection invariant.
\end{remark}

The first four examples that we introduce here, namely $\C P^2$, $S^1 \times S^3$ and sphere bundles over $S^2$, are all very simple and provide easy sources of example calculations.

\begin{example}[Projective Space $\C P^2$] \label{ex:CP2_trisection} Complex projective space admits a standard $(1,0)$-trisection, which can be written as in Figure \ref{fig:CP2_trisection_1} below.
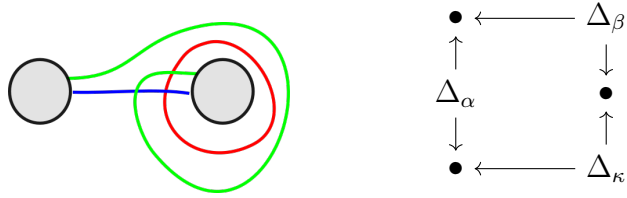
\begin{figure}[h]
\centering
\includegraphics[width=.25\textwidth]{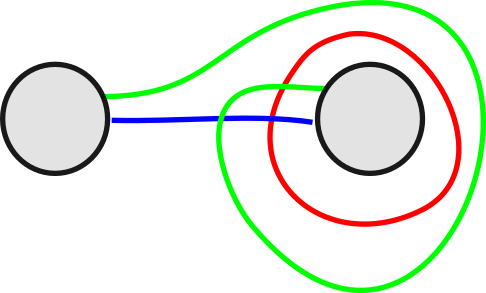} \qquad\qquad
\begin{tikzpicture}
  \node at (0,0) (Da) {$\Delta_\alpha$};
  \node at (2,1) (Db) {$\Delta_\beta$};
  \node at (2,-1) (Dc) {$\Delta_\kappa$};
  \node at (0,1) (Pab) {$\bullet$};
  \node at (0,-1) (Pca) {$\bullet$};
  \node at (2,0) (Pbc) {$\bullet$};

  \draw[->] (Da)--(Pab);
  \draw[->] (Da)--(Pca);
  \draw[->] (Db)--(Pab);
  \draw[->] (Db)--(Pbc);
  \draw[->] (Dc)--(Pca);
  \draw[->] (Dc)--(Pbc);
  \end{tikzpicture}
\caption{A trisection diagram for $\C P^2$.}
\label{fig:CP2_trisection_1}
\end{figure}
\end{example}

\begin{example}[$S^1 \times S^3$] \label{ex:S1xS3_trisection} Another very simple trisection is that of the product manifold $S^1 \times S^3$, which admits a $(1,0)$-trisection as in Figure \ref{fig:S1xS3_trisection_1} below.
\begin{figure}[h]
\centering
\includegraphics[width=.25\textwidth]{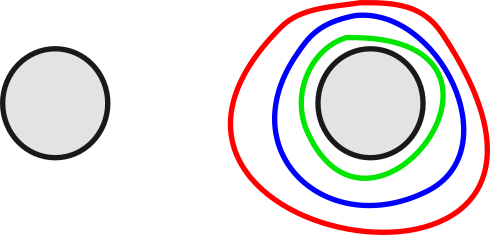} \qquad\qquad
\begin{tikzpicture}
  \node at (0,1) (Ca) {$C_\alpha$};
  \node at (1,1) (Cb) {$C_\beta$};
  \node at (2,1) (Cc) {$C_\kappa$};
  \node at (0,0) (Ea) {$\epsilon_\alpha$};
  \node at (1,0) (Eb) {$\epsilon_\beta$};
  \node at (2,0) (Ec) {$\epsilon_\kappa$};

  \draw[->] (Ca)--(Ea);
  \draw[->] (Cb)--(Eb);
  \draw[->] (Cc)--(Ec);
  \end{tikzpicture}
\caption{A trisection diagram for $S^1 \times S^3$.}
\label{fig:S1xS3_trisection_1}
\end{figure}
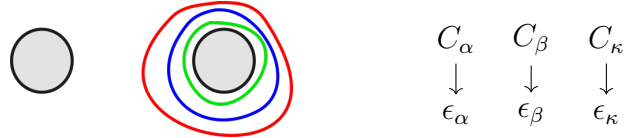
\end{example}

\begin{example}[$S^2 \times S^2$] \label{ex:S2xS2_trisection} The product $S^2 \times S^2$ of two $2$-spheres admits a genus $2$ trisection diagram as in Figure \ref{fig:S2xS2_trisection} below.
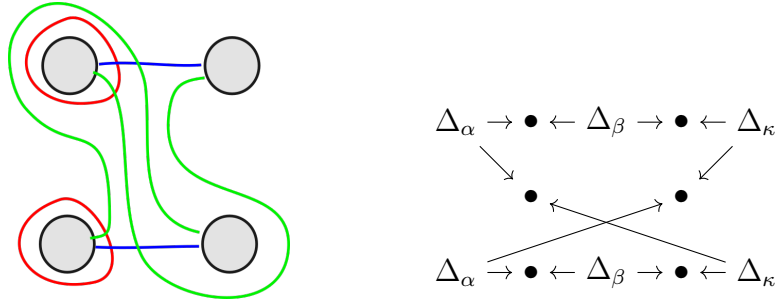
\begin{figure}[h]
\centering
\includegraphics[width=.25\textwidth]{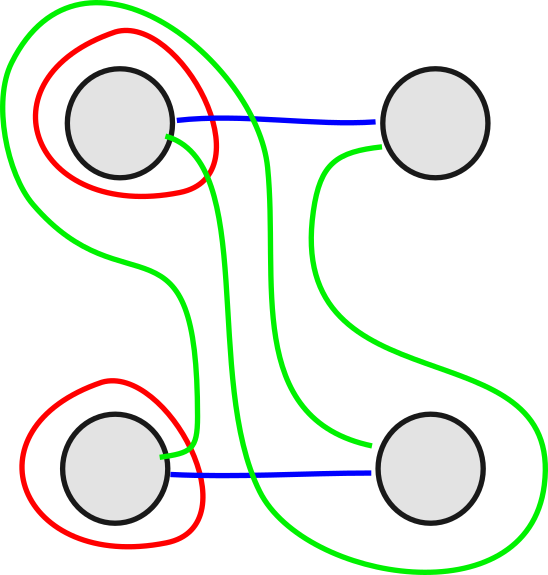} \qquad \qquad
\begin{tikzpicture}
  \node at (0,1) (Da1) {$\Delta_\alpha$};
  \node at (0,-1) (Da2) {$\Delta_\alpha$};
  \node at (2,1) (Db1) {$\Delta_\beta$};
  \node at (2,-1) (Db2) {$\Delta_\beta$};
  \node at (4,1) (Dc1) {$\Delta_\kappa$};
  \node at (4,-1) (Dc2) {$\Delta_\kappa$};
  \node at (1,1) (Pab1) {$\bullet$};
  \node at (1,-1) (Pab2) {$\bullet$};
  \node at (3,1) (Pbc1) {$\bullet$};
  \node at (3,-1) (Pbc2) {$\bullet$};
  \node at (1,0) (Pca1) {$\bullet$};
  \node at (3,0) (Pca2) {$\bullet$};

  \draw[->] (Da1)--(Pab1);
  \draw[->] (Da2)--(Pab2);
  \draw[->] (Db1)--(Pab1);
  \draw[->] (Db2)--(Pab2);

  \draw[->] (Db1)--(Pbc1);
  \draw[->] (Db2)--(Pbc2);
  \draw[->] (Dc1)--(Pbc1);
  \draw[->] (Dc2)--(Pbc2);

  \draw[->] (Da1)--(Pca1);
  \draw[->] (Da2)--(Pca2);
  \draw[->] (Dc2)--(Pca1);
  \draw[->] (Dc1)--(Pca2);
  \end{tikzpicture}
\caption{A trisection diagram for $S^2 \times S^2$.}
\label{fig:S2xS2_trisection}
\end{figure}
\end{example}

\begin{example}[$S^2 \tilde{\times} S^2$] \label{ex:S2tildexS2_trisection} The twisted product $S^2 \tilde{\times} S^2$ (that is, the total space of the non-trivial oriented sphere bundle over $S^2$) admits a genus $2$ trisection diagram as in Figure \ref{fig:twS2xS2_trisection} below.
\begin{figure}[h]
\centering
\includegraphics[width=.25\textwidth]{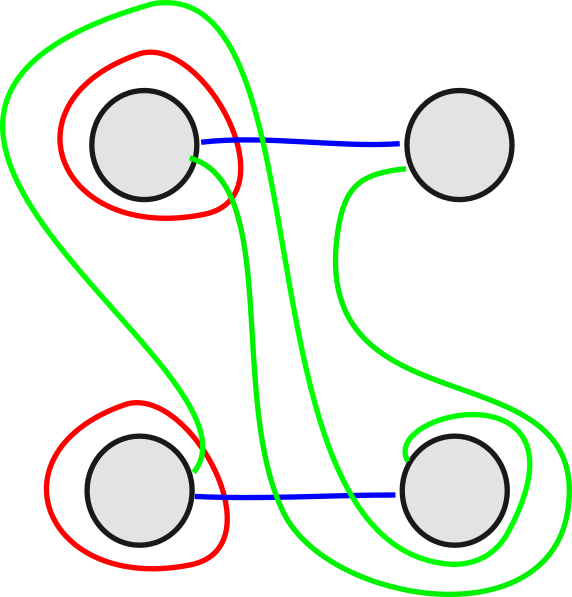} \qquad \qquad
\begin{tikzpicture}
  \node at (0,1) (Da1) {$\Delta_\alpha$};
  \node at (0,-1) (Da2) {$\Delta_\alpha$};
  \node at (2,1) (Db1) {$\Delta_\beta$};
  \node at (2,-1) (Db2) {$\Delta_\beta$};
  \node at (4,1) (Dc1) {$\Delta_\kappa$};
  \node at (4,-1) (Dc2) {$\Delta_\kappa$};
  \node at (1,1) (Pab1) {$\bullet$};
  \node at (1,-1) (Pab2) {$\bullet$};
  \node at (3,1) (Pbc1) {$\bullet$};
  \node at (3,-1) (Pbc2) {$\bullet$};
  \node at (1,0) (Pca1) {$\bullet$};
  \node at (3,0) (Pca2) {$\bullet$};

  \node at (2,.3) (Pbc3) {$\bullet$};

  \draw[->] (Da1)--(Pab1);
  \draw[->] (Da2)--(Pab2);
  \draw[->] (Db1)--(Pab1);
  \draw[->] (Db2)--(Pab2);

  \draw[->] (Db1)--(Pbc1);
  \draw[->] (Db2)--(Pbc2);
  \draw[->] (Dc1)--(Pbc1);
  \draw[->] (Dc2)--(Pbc2);

  \draw[->] (Da1)--(Pca1);
  \draw[->] (Da2)--(Pca2);
  \draw[->] (Dc2)--(Pca1);
  \draw[->] (Dc1)--(Pca2);

  \draw[->] (Dc1)--(Pbc3);
  \draw[->] (Db2)--(Pbc3);
  \end{tikzpicture}
\caption{A trisection diagram for $S^2 \tilde{\times} S^2$.}
\label{fig:twS2xS2_trisection}
\end{figure}
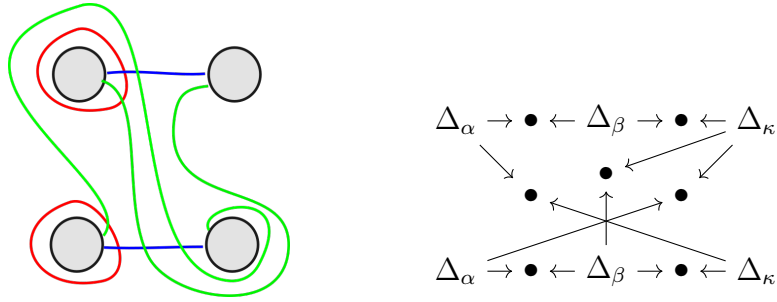
\end{example}

\begin{example}[$T^2 \times S^2$] \label{ex:T2xS2_trisection} The product $T^2 \times S^2$ admits a genus $4$ trisection diagram as in Figure \ref{fig:T2xS2_trisection} below. We omit the corresponding trisection bracket, as it is quite complicated in this case.
\begin{figure}[h]
\centering
\includegraphics[width=.8\textwidth]{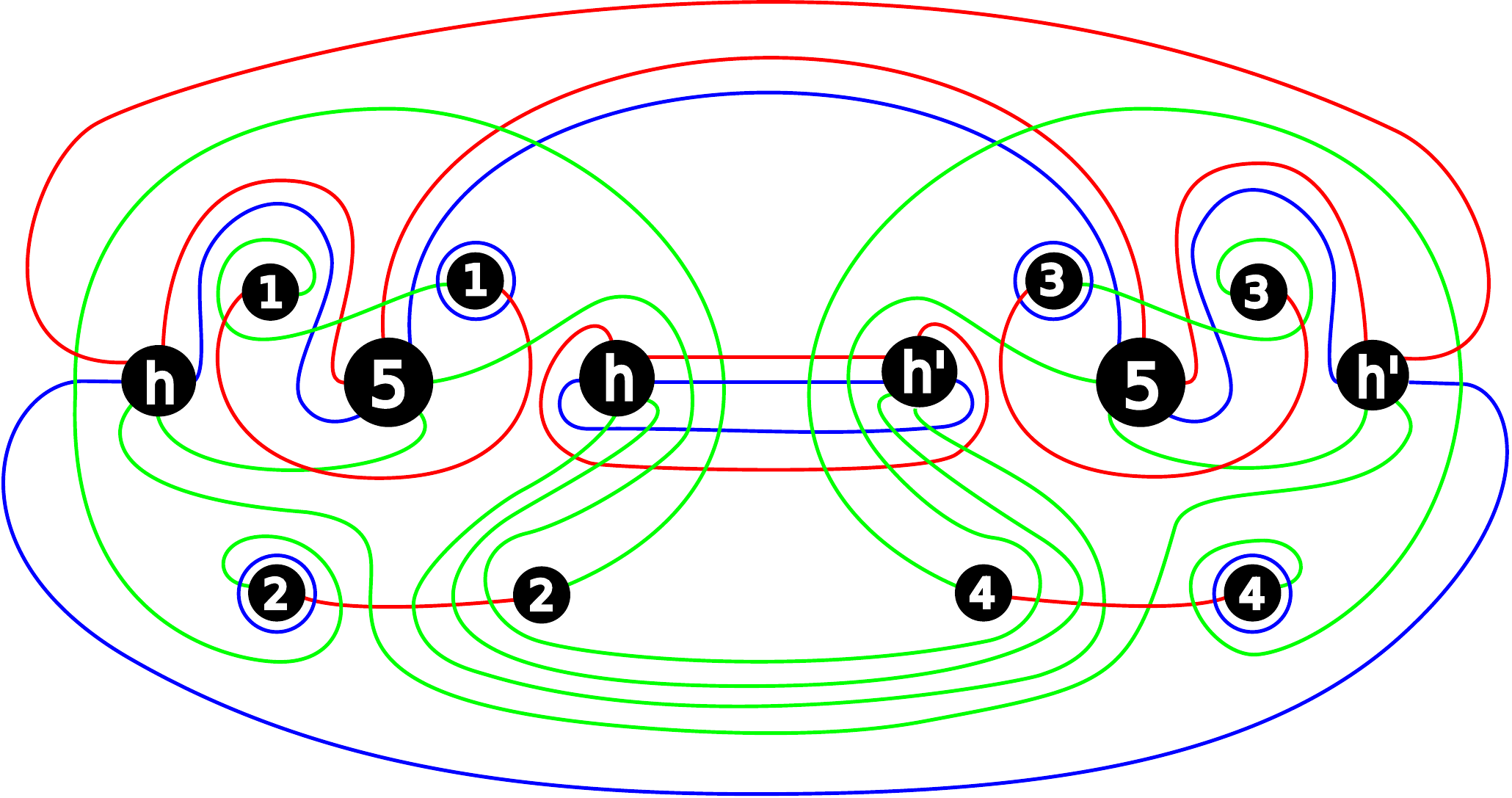}
\caption{A trisection diagram for $T^2 \times S^2$ (see \cite[Fig 19]{co2017lefschetztrisections}).}
\label{fig:T2xS2_trisection}
\end{figure}
\end{example}

The final (and most non-trivial) example of a trisection that we will include in this section is the following, of the Kummer surface. 

\begin{example}[Kummer Surface $K$] \label{ex:K3_trisection} The Kummer surface (or K3 surface) $K$ is the unique Calabi-Yau surface aside from $T^4$, up to deformation. By representing the Kummer surface as a certain branched cover of $\C P^2$, Lambert-Cole and Meier give an efficient trisection diagram for $K$ in \cite{lcm2018rationalsurfaces}. See Figure \ref{fig:K3_trisection}.
\begin{figure}[h]
\centering
\includegraphics[width=.7\textwidth]{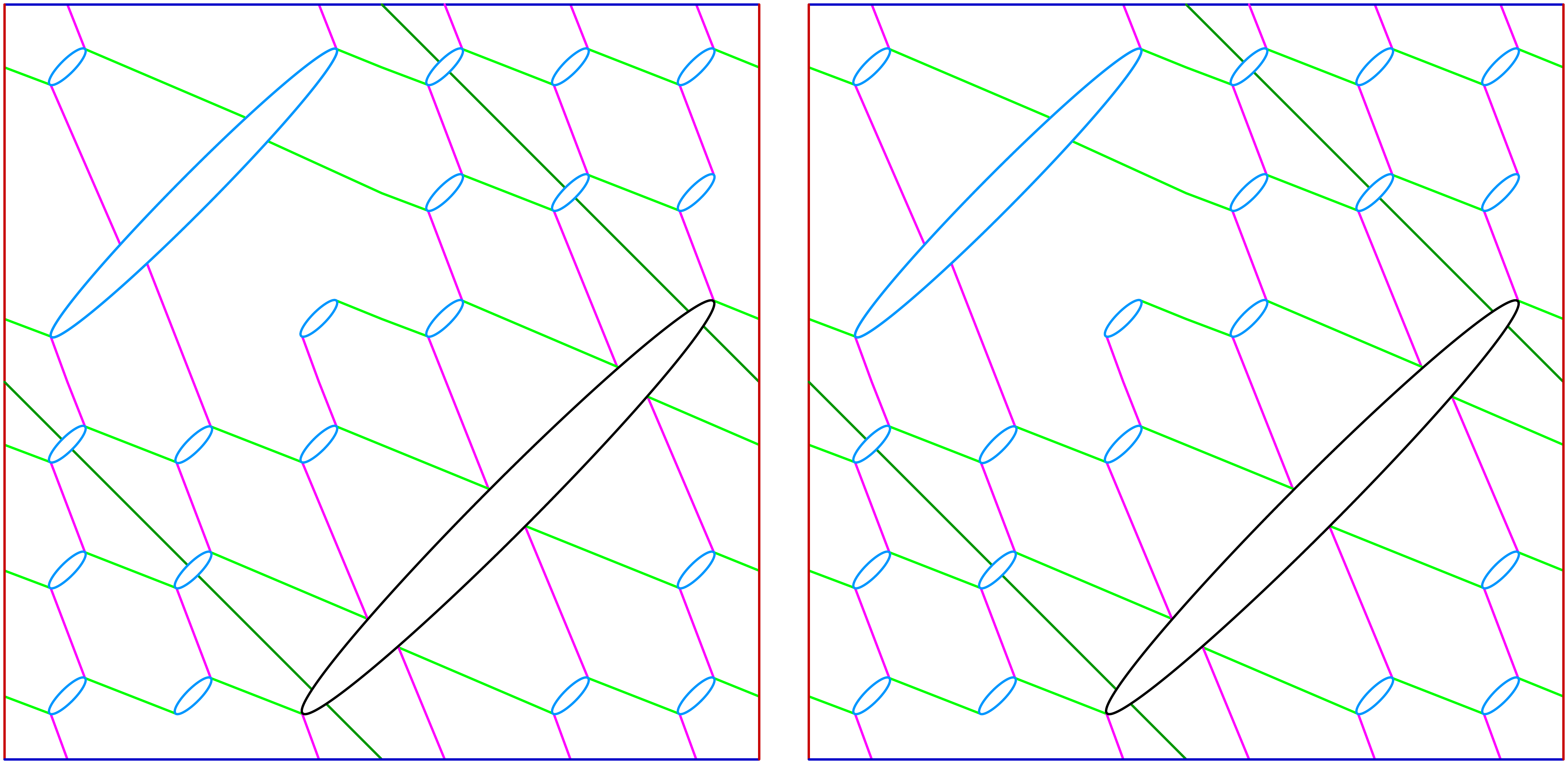} 
\caption{A trisection diagram for the Kummer surface $K$ (see \cite[Fig 16]{lcm2018rationalsurfaces}).}
\label{fig:K3_trisection}
\end{figure}
\end{example}

Although this trisection is currently beyond the computation abilities of our script \cite{pythonscript2019}, the diagram can be tabulated and easily stored as a trisection datum (see Definition \ref{def:trisection_datum} below) and is included in \cite{pythonscript2019}. Improvements in the efficiency of \cite{pythonscript2019} or enhancements to the properties of the invariant (e.g.~gluing formulae) may make calculations with this trisection tractable in the near future.

Note, however, that we can evaluate $\tau_{\mathcal{H}}(X)$ for $X$ simply connected so long as $\tau_{\mathcal{H}}(\mathbb{C}P^2) \not= 0$ and $\tau_{\mathcal{H}}(\overline{\mathbb{C}P^2}) \not= 0$.  Indeed, by Lemma 3.12 of~\cite{barenz2016dichromatic}, since our invariant is multiplicative under connect sums we would have
\begin{equation}
\label{eq:multformula1}
\tau_{\mathcal{H}}(X) = \left(\tau_{\mathcal{H}}(\mathbb{C}P^2) \, \tau_{\mathcal{H}}(\overline{\mathbb{C}P^2}) \right)^{-1 + \frac{\chi(X)}{2}} \left( \frac{\tau_{\mathcal{H}}(\mathbb{C}P^2)}{\tau_{\mathcal{H}}(\overline{\mathbb{C}P^2}) }\right)^{\frac{\sigma(X)}{2}}
\end{equation}
where $\chi$ is the Euler characteristic of $X$ and $\sigma$ is the signature.  Of course, not all trisection invariants are non-vanishing on $\mathbb{C}P^2$ and $\overline{\mathbb{C}P^2}$, see Section~\ref{subsec:cyclic_triplets_and_kashaev} for an example.

\subsection{Computational Methods} \label{subsec:computational_methods_and_scripting} Most trisection diagrams produce tensor diagram expressions for the corresponding trisection invariant that are too large and complicated to evaluate by hand. However, it is relatively straightforward to write computer code to calculate the invariant, essentially directly from the definition. Here we briefly outline how this is done.

\begin{definition}[Trisection Datum] \label{def:trisection_datum} A \emph{trisection datum} $D = (N, g, \sigma, I)$ consists of the following data.
\begin{itemize}
  \item[(a)] A genus $g \in \Z_{\ge 0}$ and an intersection number $N \in \Z_{\ge 0}$. The list $\{1,\dots,N\}$ is called the list of intersections.
  \item[(b)] A map $\sigma:\{1,\dots,N\} \to \{\pm 1\}$ or equivalently an ordered list of $N$ signs.
  \item[(c)] For each $\gamma \in \{\alpha,\beta,\kappa\}$ and $i \in \{1,\dots,g\}$, a list $I^\gamma_i = (i_1,\dots,i_m)$ of integers $1 \le i_j \le N$ where each intersection occurs exactly twice across all lists $I^\gamma_i$.
\end{itemize}

A trisection $T$ determines a trisection datum $D(T) = (N,g,\sigma,I)$ as so. First, order the intersections $\mathcal{I}(T)$ and the $\alpha/\beta/\kappa$-curves. Then define $N := \# \mathcal{I}(T)$ and $g := g(T)$, take the sign map $\sigma:\{1,\dots,N\} \simeq \mathcal{I}(T) \to \{\pm 1\}$ to be the intersection sign and form $I^\gamma_i$ by listing the intersections along each curve $\gamma_i$.
\end{definition}

A trisection datum is essentially a combinatorial data type containing all of the data necessary to calculate $\tau_{\mathcal{H}}(T)$, given the additional data of a Hopf triplet. The Hopf triplet itself can be stored as a set of tensors, namely the structure tensors of the three Hopf algebras and the pairing tensors.

The procedure for computing the trisection invariant with this data is essentially a direct application of the definition.

\begin{procedure} \label{proc:computing_trisection} Let $\mathcal{H} = (H_\alpha,H_\beta,H_\kappa,\langle-\rangle)$ be a Hopf triplet and $D = (N,g,\sigma,I)$ be a datum for a trisection $T$. The procedure for computing the trisection invariant $\tau_{\mathcal{H}}(T)$ goes like this.
\begin{itemize}
  \item[(a)] For each $\gamma \in \{\alpha,\beta,\kappa\}$ and $i \in \{1,\dots,g\}$, form a copy $\Delta_{\gamma,i}$ of the $0$-input, $k$-output coproduct tensor from the Hopf algebra $H_\gamma$.
  \item[(b)] Label the outputs of $\Delta_{\gamma,i}$ by the intersections in the list $I^\gamma_i$.
  \item[(c)] For each intersection $i$, find the two pairs $(\gamma,j)$ and $(\eta,k)$ so that $I^\gamma_j$ and $I^\eta_k$ contain the intersection $i$. Then contract $\Delta^\gamma_j$ and $\Delta^\eta_k$ using the pairing $\langle-\rangle_{\gamma\eta}$ if $\sigma(i) = +1$ or $\langle S-\rangle_{\gamma\eta} = -1$.
\end{itemize}
\end{procedure}

An implementation of this procedure as a Python script, written by the authors of this paper, can be found at \cite{pythonscript2019}. To conclude our discussion of computational methods, let us include a brief discussion of optimization.

\begin{remark}[Optimizations] Here is a list of optimizations that are useful in implementing Procedure \ref{proc:computing_trisection}.
\begin{itemize}
  \item[(a)] It is often more efficient to implement the structure tensors of $\mathcal{H}$ as sparse matrices, since (for instance) the structure tensors for many naturally occuring Hopf algebras (like group algebras) are sparse.
  \item[(b)] Relatedly, it is generally advantageous from an efficiency perspective to minimize the dimension of the vector-spaces being used in tensor calculations. This can be accomplished, for instance, by performing the contractions in step (c) in stages so that the maximum number of outputs are paired at each stage.
\end{itemize}
\end{remark}

\subsection{Cyclic Triplets and Kashaev's Invariant} \label{subsec:cyclic_triplets_and_kashaev} We now explain a first set of example calculations using the trisection diagrams of \S \ref{subsec:trisection_diagrams_of_examples} and the computational methods described in \S \ref{subsec:computational_methods_and_scripting}. Namely, we compute the trisection invariant for Hopf triplets living in the following simple family.

\begin{definition}[Cyclic Triplet] The \emph{cyclic Hopf triplet} $\mathcal{Z}[N]$ for a positive integer $N \ge 2$ is the (trisection admissible, involutory) Hopf triplet defined as follows.

Consider the Hopf algebra $\mathbb{C}[\mathbb{Z}/N]$ which is the group Hopf algebra of the cyclic group $\mathbb{Z}/N$. This Hopf algebra has a well-known quasi-triangular structure (c.f.~\cite{majid2000}), and the $R$-matrix can be written explicitly as follows.
\begin{equation}
R = \frac{1}{N} \sum_{k,\ell=0}^{N-1} \exp\left(2\pi i\, \frac{k \ell}{N}\right) [k] \otimes [\ell] \,,
\end{equation}
We may thus construct a triple $\mathcal{Z}[N] = (Z_\alpha,Z_\beta,Z_\kappa,\langle-\rangle)$ as in Example \ref{ex:basic_examples_of_triplets}. Namely, we define the constituent Hopf algebras by
\[Z_\alpha = \C[\mathbb{Z}/N]^*\,, \qquad Z_\beta = \C[\mathbb{Z}/N]\,, \qquad Z_\kappa = \C[\mathbb{Z}/N]^*\]
The pairings between $Z_\alpha = Z_\kappa$ and $Z_\beta$ are given by the dual pairing, while the final pairing is constructed using the $R$-matrix as in Example \ref{ex:basic_examples_of_triplets}(b). Note that we are omiting all applications of $(-)^{\op{op}}$ and $(-)^{\op{cop}}$, since all the Hopf algebras here are commutative and co-commutative and thus these operations have no effect.
\end{definition}

Empirically, the trisection invariant associated to $\mathcal{Z}[N]$ seems to be essentially equivalent to the numerical invariants arising from a family of simple 4D TQFTs introduced by Kashaev in \cite{kashaev2014asimple}. Let us briefly recount Kashaev's construction.

\begin{definition}[Kashaev Invariant] Let $X$ be a closed $4$-manifold and fix an integer $N \ge 2$. The \emph{Kashaev invariant} $\mathcal{K}_N(X)$ is defined as follows.

We start by fixing some auxiliary data. Let $V = \mathbb{C}^N$ be the standard $N$-dimensional Hilbert space wit the standard Hermitian inner product, and let $\{e_k\}_{k=0}^{N-1}$ and $\{\overline{e}_k\}_{k=0}^{N-1}$ denote the standard bases of $V$ and the dual basis of $V^*$ respectively. Let $Q$ denote the $5$ index tensor
\begin{equation} \label{eqn:Kashaev_Q_matrix}
Q := \frac{1}{\sqrt{N}} \sum_{k,\ell,m = 0}^{N-1} \op{exp}\left(2\pi i \,\frac{km}{N}\right) \cdot e_k \otimes \overline{e}_{k + \ell} \otimes e_\ell \otimes \overline{e}_{\ell + m} \otimes e_m 
\end{equation}
Note that the Hermitian conjugate tensor $Q^\dag$ of $Q$ may be written as follows.
\begin{equation} \label{eqn:Kashaev_Qbar_matrix}
Q^\dag := \frac{1}{\sqrt{N}} \sum_{k,\ell,m = 0}^{N-1} \op{exp}\left(-2\pi i \,\frac{km}{N}\right) \cdot \overline{e}_k \otimes e_{k + \ell} \otimes \overline{e}_\ell \otimes e_{\ell + m} \otimes \overline{e}_m
\end{equation}
Finally, choose an arbitrary triangulation $\mathcal{T}$ of $X$ and order the vertices of $\mathcal{T}$. To compute $\mathcal{K}_N(X)$, proceed as so. 

First, assign a copy $Q[\Delta]$ of either $Q$ or $Q^\dag$ to each $4$-dimensional simplex $\Delta \in \mathcal{T}$ in the triangulation $\mathcal{T}$. We use $Q$ if the orientation on $\Delta$ induced by $X$ agrees with that induced by the vertex order, and $Q^\dag$ if the orientations disagree. Then, label the five indices of $Q[\Delta]$ by the $3$-dimensional facets of $\partial\Delta$, using the dictionary ordering induced by the vertex ordering. Finally, contract all pairs of indices sharing a label on any pair of tensors $Q[\Delta]$ and $Q[\Delta']$.

The Kashaev invariant $\mathcal{K}_N(X) \in \C$ is defined to be the resulting scalar acquired by this final contraction times the normalization factor $N^{-|\mathcal{T}_0|}$, where $|\mathcal{T}_0|$ is the number of vertices of $\mathcal{T}$. \end{definition}

In Table 1 of \cite{kashaev2014asimple}, Kashaev presents a calculation of the Kashaev invariant for the spaces $S^4, \C P^2, S^2 \times S^2$, $S^1 \times S^3$ and $T^2 \times S^2$. By utilizing the diagrams presented in Examples \ref{ex:CP2_trisection}-\ref{ex:T2xS2_trisection} and (in some cases) the computation methods discussed in \S \ref{subsec:computational_methods_and_scripting}, we computed the following table comparing the Kashaev invariant $\mathcal{K}_N(-)$ to $\tau_{\mathcal{Z}[N]}(-)$ in these cases.

\begin{table}[h] \label{table:trisection_vs_Kashaev}
\begin{center}
{\renewcommand{\arraystretch}{1.2} %<- modify value to suit your needs
\begin{tabular}{ |c|c|c|c|c| } 
\hline
$X$ & $\chi(X)$ & $N^{\chi(X) + 1} \cdot \mathcal{K}_N(X)$ & $\tau_{\mathcal{Z}[N]}(X)$ \\
\hline \hline
$S^4$ & 2 & 1 & 1 \\ 
\hline
$S^2 \times S^2$ & 4 & $N^{-1}(3 + (-1)^N)/2$ &  $N^{-1}(3 + (-1)^N)/2 \quad (\dag)$ \\ 
\hline
$\mathbb{C}P^2$ & 3 & $N^{-1}\sum_{k=1}^N \omega_N^{k^2}$ & $N^{-1}\sum_{k=1}^N \omega_N^{k^2}$\\ 
\hline
$S^3 \times S^1$ & 0 & $N$ & $N$\\
\hline
$S^2 \times T^2$ & 0 & $N(3 + (-1)^N)/2$ & $N(3 + (-1)^N)/2 \quad (\dag\dag)$\\
\hline
\end{tabular}}
\end{center}
\caption{Comparing a family of trisection invariants, and the Kashaev invariants.}
\end{table}
Above, $\chi(X)$ denotes the Euler characteristic. Note that the formula $(\dag)$ for $\tau_{\mathcal{Z}[N]}(S^2 \times S^2)$ has been verified for $2 \le N \le 100$ and the formula $(\dag\dag)$ for $\tau_{\mathcal{Z}[N]}(T^2 \times S^2)$ has been verified for $2 \le N \le 4$. The remaining cases can be checked exactly. In light of these empirical results, we formulate the following conjecture.

\begin{conjecture} \label{conj:Kashaev_is_trisection} For all oriented closed $4$-manifolds $X$ and any $N \ge 2$, we have
\[\tau_{\mathcal{Z}[N]}(X) = N^{\chi(X) + 1} \cdot \mathcal{K}_N(X)\]\end{conjecture}
\noindent Due to Theorem \ref{thm:trisection_vs_crane_yetter_informal} (or the formal version Corollary \ref{cor:trisection=cy}) and the fact that $\mathcal{Z}[N]$ arises from the quasi-triangular Hopf algebra $\mathbb{C}[\mathbb{Z}/N]$ , Conjecture \ref{conj:Kashaev_is_trisection} 
would imply that the Kashaev invariants associated with an integer $N$ are equal to Crane-Yetter invariants based on $\text{Rep}(\mathbb{C}[\mathbb{Z}/N])$ up to Euler characteristics. This is also conjectured in \cite{williamson2016hamiltonian} when studying Hamiltonian models of the two theories.

\subsection{Triplets from the $8$-Dimensional Algebra} \label{subsec:triplets_from_8d_algebra} As a final computation for this section, we tabulate the value of the trisection invariant on some of the simple spaces in \S \ref{subsec:trisection_diagrams_of_examples} for a Hopf triplet that is beyond the purview of the dichromatic invariants via Theorem \ref{thm:trisection_vs_crane_yetter_informal} (and the more formally stated Theorem \ref{thm:tri=dichro}). That is, these invariants do \emph{not} arise from the Hopf triplet associated to a quasi-triangular Hopf algebra.

Each of the constituent Hopf algebras in the triplets of interest in this section will be isomorphic to a fixed Hopf algebra $H_8$, which admits the following description.

\begin{notation}[$8$d Hopf Algebra] \label{not:8d_Hopf_algebra} Denote by $H_8$ the unique semisimple Hopf algebra over $\C$ with $\op{dim}(H) = 8$ that is neither commutative nor cocommutative. 

More explicitly, $H_8$ may be presented as a quotient algebra $H_8 = \C\langle x,y,z\rangle/I$ of the free, unital associative algebra $\C\langle x,y,z\rangle$ generated by $3$ variables. The ideal $I$ in the quotient is generated by the relations
\begin{equation}
I = \langle xy - yx, xz - zy, yz - zx, x^2 -1, y^2 - 1, \ z^2 - \frac{1}{2}(1+x+y-xy)\rangle
\end{equation}
This defines the algebra structure on $H_8$ and implies that the set of elements $B = \{1,x,y,xy,z,xz,yz,xyz\}$ form a basis of $H_8$ as a $\C$ vector space. The coalgebra structure $(\Delta,\epsilon)$ can be specified as follows.
\begin{equation} \label{eqn:8_dim_coproduct}
\Delta(x) = x \otimes x,  \ \Delta(y) = y \otimes y, \ \Delta(z) = \frac{1}{2}(z \otimes z + yz \otimes z + z \otimes xz - yz \otimes xz)
\end{equation}
The coproduct of the remaining basis elements of $B$ can be deduced from (\ref{eqn:8_dim_coproduct}) and the bialgebra property. The counit may likewise be specified as follows.
\begin{equation}
\epsilon(w) = 1 \qquad \text{for} \qquad w \in B
\end{equation}
Finally, the antipode tensor $S$ can be specified by
\begin{equation}
S(w) = w \qquad \text{for} \qquad w \in \{x,y,z\}
\end{equation}
The antipode of the remaining basis elements of $B$ can be deduced from (\ref{eqn:8_dim_coproduct}) and the anti-homomorphism property of $S$.
\end{notation}

Next, we fix notation for a curated collection of skew pairings on $H_8$, each of which give the pair $(H_8,H_8)$ the structure of a Hopf doublet. 

\begin{notation}[Pairings] \label{not:skew_pairing_a} We denote by $\langle-\rangle_i$ for $i \in \{0,1,2,3\}$ the pairings $H_8 \times H_8 \to \C$ specified by the following matrices $M_i$ in the basis $B$ of Notation \ref{not:8d_Hopf_algebra}.
\begin{equation*}
M_0 := \left(
\begin{array}{cccccccc}
 1 & 1 & 1 & 1 & 1 & 1 & 1 & 1 \\
 1 & -1 & -1 & 1 & i & -i & -i & i \\
 1 & -1 & -1 & 1 & i & -i & -i & i \\
 1 & 1 & 1 & 1 & 1 & 1 & 1 & 1 \\
 1 & i & i & 1 & -1-i & 0 & 0 & -1-i \\
 1 & -i & -i & 1 & 0 & -1+i & -1+i & 0 \\
 1 & -i & -i & 1 & 0 & -1+i & -1+i & 0 \\
 1 & i & i & 1 & -1-i & 0 & 0 & -1-i \\
\end{array}
\right)
\end{equation*} 

\begin{equation*}
M_1 := \left(
\begin{array}{cccccccc}
 1 & 1 & 1 & 1 & 1 & 1 & 1 & 1 \\
 1 & -1 & -1 & 1 & i & -i & -i & i \\
 1 & -1 & -1 & 1 & -i & i & i & -i \\
 1 & 1 & 1 & 1 & -1 & -1 & -1 & -1 \\
 1 & -i & i & -1 & -\sqrt{2} & 0 & 0 & \sqrt{2} \\
 1 & i & -i & -1 & 0 & i \sqrt{2} & -i \sqrt{2} & 0 \\
 1 & i & -i & -1 & 0 & -i \sqrt{2} & i \sqrt{2} & 0 \\
 1 & -i & i & -1 & \sqrt{2} & 0 & 0 & -\sqrt{2} \\
\end{array}
\right)
\end{equation*}
\begin{equation*}
M_2 := \left(
\begin{array}{cccccccc}
 1 & 1 & 1 & 1 & 1 & 1 & 1 & 1 \\
 1 & -1 & -1 & 1 & i & -i & -i & i \\
 1 & -1 & -1 & 1 & -i & i & i & -i \\
 1 & 1 & 1 & 1 & -1 & -1 & -1 & -1 \\
 1 & -i & i & -1 & \sqrt{2} & 0 & 0 & -\sqrt{2} \\
 1 & i & -i & -1 & 0 & -i \sqrt{2} & i \sqrt{2} & 0 \\
 1 & i & -i & -1 & 0 & i \sqrt{2} & -i \sqrt{2} & 0 \\
 1 & -i & i & -1 & -\sqrt{2} & 0 & 0 & \sqrt{2} \\
\end{array}
\right)
\end{equation*}

\begin{equation*}
M_3 := \left(
\begin{array}{cccccccc}
 1 & 1 & 1 & 1 & 1 & 1 & 1 & 1 \\
 1 & -1 & -1 & 1 & -i & i & i & -i \\
 1 & -1 & -1 & 1 & i & -i & -i & i \\
 1 & 1 & 1 & 1 & -1 & -1 & -1 & -1 \\
 1 & i & -i & -1 & -\sqrt{2} & 0 & 0 & \sqrt{2} \\
 1 & -i & i & -1 & 0 & -i \sqrt{2} & i \sqrt{2} & 0 \\
 1 & -i & i & -1 & 0 & i \sqrt{2} & -i \sqrt{2} & 0 \\
 1 & i & -i & -1 & \sqrt{2} & 0 & 0 & -\sqrt{2} \\
\end{array}
\right)
\end{equation*}
\end{notation}

Finally, we construct three triplets by combining the pairings given above. We emphasize that these triplets are just a selection of examples, and there are many more pairings and pairing combinations that are possible.

\begin{notation}[Triplets] We denote by $\mathcal{H}_*$ for $* \in \{\op{A},\op{B},\op{C}\}$ the Hopf triplet defined as follows. The consistituent $\alpha,\beta$ and $\kappa$ Hopf algebras are, for each triplet, simply equal to the $8$-dimensional algebra $H_8$. The pairings are given as so.
\begin{itemize}
    \item[(a)] For $\mathcal{H}_{\op{A}}$, the pairings are defined to be the pairings $\langle-\rangle_1$. That is
    \[
    \langle-\rangle_{\alpha\beta} = \langle-\rangle_{\beta\kappa} = \langle-\rangle_{\kappa\alpha} = \langle-\rangle_1
    \]
    \item[(b)] For $\mathcal{H}_{\op{B}}$, the pairings are defined to vary as follows.
    \[
    \langle-\rangle_{\alpha\beta} = \langle-\rangle_1 \qquad \langle-\rangle_{\beta\kappa} = \langle-\rangle_2 \qquad \langle-\rangle_{\kappa\alpha} = \langle-\rangle_3
    \]
    \item[(c)] For $\mathcal{H}_{\op{C}}$, the pairings are defined to vary as follows.
    \[
    \langle-\rangle_{\alpha\beta} = \langle-\rangle_0 \qquad \langle-\rangle_{\beta\kappa} = \langle-\rangle_1 \qquad \langle-\rangle_{\kappa\alpha} = \langle-\rangle_1
    \]    
\end{itemize}
\end{notation} 

Of course, one must verify that the pairings and triplets satisfy the necessary properties. We verified this Lemma computationally using a script available at \cite{pythonscript2019}.
\begin{lemma} The tuple $(H_8,H_8,\langle-\rangle_i)$ is a Hopf doublet for each $i \in \{0,1,2,3\}$. Furthermore, the tuple $\mathcal{H}_*$ for each $* \in \{\op{A},\op{B},\op{C}\}$ is an (involutory and trisection admissible) Hopf triplet.
\end{lemma}

We now conclude this section with Table \ref{table:trisection_invariant_with_8d_examples}, where we record the trisection invariants $\tau_{\mathcal{H}_{*}}(-)$ for $* \in \{\op{A},\op{B},\op{C}\}$ and the trisections in Examples \ref{ex:CP2_trisection}-\ref{ex:S1xS3_trisection}. These invariants were calculated using the methods described in \S \ref{subsec:computational_methods_and_scripting}.

\begin{table}[h] \label{table:trisection_invariant_with_8d_examples}
\begin{center}
{\renewcommand{\arraystretch}{1.2} %<- modify value to suit your needs
\begin{tabular}{ |c|c|c|c|c|} 
\hline
$X$ & $\chi(X)$ & $\mathcal{H}_{\op{A}}$ & $\mathcal{H}_{\op{B}}$ & $\mathcal{H}_{\op{C}}$ \\
\hline \hline
$S^4$ & 2 & 1 & 1 & 1 \\ 
\hline
$\mathbb{C}P^2$ & 3 & $\frac{-1 + i}{2\sqrt{2}}$ & $\frac{1 + i}{2\sqrt{2}}$ & $0$ \\ 
\hline
$S^3 \times S^1$ & $0$ & $8$ & $8$ & $2^{8/3}$ \\
\hline
$S^2 \times S^2$ & $4$ & $\frac{1}{4}$ & $\frac{1}{4}$ & $2^{-2/3}$ \\ 
\hline
$S^2 \tilde{\times} S^2$ & $4$ & $\frac{1}{4}$ & $\frac{1}{4}$ & $0$ \\ 
\hline
\end{tabular}}
\end{center}
\caption{Computations of trisection invariant for $\mathcal{H}_{\op{A}}, \mathcal{H}_{\op{B}}$ and $\mathcal{H}_{\op{C}}$.}
\end{table}

\section{Relation to the Dichromatic Invariant}
\label{sec:CY_dichro}

In this section, we review the dichromatic invariant (\S \ref{subsec:dichro}). We then (\S \ref{subsec:trisection_dichro}) formally restate and prove Theorem \ref{thm:trisection_vs_crane_yetter_informal} (as Theorem \ref{thm:tri=dichro}). 

\begin{remark}[Historical] Before beginning, we provide the reader with some historical discussion of the dichromatic invariant and its relation to the Crane-Yetter invariant.

The Crane-Yetter invariant is an invariant of closed oriented 4-manifolds first defined in \cite{crane1993categorical} based on a semisimple quotient of $\Rep(U_q(sl_2))$ at some root of unity, which is a special example of modular tensor categories. The invariant was later generalized to take as input any ribbon fusion category \cite{crane1997state}, not necessarily modular. In both cases, the invariant takes the form of a weighted state-sum on a triangulation. 

Using skein-theoretical methods, Roberts \cite{roberts1995skein} introduced a Broda-type invariant of 4-manifolds again based on the semisimple quotient of $\Rep(U_q(sl_2))$. In \cite{roberts1995skein}, Roberts showed that his invariant is equal to the Crane-Yetter invariant associated to $\Rep(U_q(sl_2))$, up to a factor involving Euler characteristics. He also showed that his invariant can be expressed in terms of the signature of the 4-manifold. In fact, Roberts' definition extends in a straightforward way to take as input any modular tensor category and the resulting invariant has the same  relation with the Crane-Yetter invariant as in the $\Rep(U_q(sl_2))$ case. This implies that the modular Crane-Yetter invariant involves only the signature and the Euler characteristic. 

The existence of a Broda-type reformulation of the Crane-Yetter invariant for premodular categories (i.e., ribbon fusion categories that are not modular) remained open until the recent progress in \cite{barenz2016dichromatic}. Generalizing the work of \cite{petit2008dichromatic, roberts1995skein}, the authors of \cite{barenz2016dichromatic} defined a Broda-type invariant (called the dichromatic invariant) of 4-manifolds based on a pivotal functor $F: \cC \rightarrow \cD$ where $\cC$ is a spherical fusion category and $\cD$ is a ribbon fusion category. 

Among other properties, the authors of \cite{barenz2016dichromatic} showed that if $\cC$ is a ribbon fusion category, $\cD$ is modular, and $F$ is a full inclusion, then the corresponding invariant depends only on $\cC$ and it recovers the Crane-Yetter invariant associated with $\cC$. They also showed that the premodular Crane-Yetter invariant contains strictly more information than the signature and the Euler characteristic combined. \end{remark}

%Given a semisimple quasi-triangular Hopf algebra $H$, recall that we can construct a Hopf triplet $\Hopf{H}$. We show that our invariant based on $\Hopf{H}$ equals the Crane-Yetter invariant based on $\Rep(H)$ up to Euler characteristics. We prove it by relating our invariant to the dichromatic invariant. In Subsection \ref{subsec:dichro}, we give a brief review of the dichromatic invariant and reformulate it in terms of $H$ in the case when $F$ is the embedding of $\Rep(H)$ to a modular tensor category.  For basics on fusion categories (spherical, ribbon, modular), see (?). For a detailed introduction of picture calculus in ribbon categories and  evaluation of ribbon graphs, see (?). In particular, we follow the conventions in (?) for the evaluation of ribbon graphs. For instance, the graphs are evaluated from bottom to top. A strand labeled by an object $V$ is interpreted as $V$ if the it is directed downwards and as $V^*$ if directed upwards.

\subsection{Review of the Dichromatic Invariant}
\label{subsec:dichro}

We now present a brief description of the dichromatic invariant. We refer the reader to \cite{barenz2016dichromatic} for a more detailed treatment. 

\begin{remark}[Background/Conventions] For basics on fusion categories (spherical, ribbon, modular), see for instance \cite{bojko2001lectures, Wang2010topological}. For a detailed introduction of picture calculus in ribbon categories see \cite{turaev1994quantum}. In particular, we follow the conventions in \cite{turaev1994quantum} for the evaluation of ribbon graphs. The graphs are evaluated from bottom to top. A strand labeled by an object $V$ is interpreted as $\text{Id}_V$ if it is directed downwards and as $\text{Id}_{V^*}$ if directed upwards. A positive crossing denotes the braiding and a negative crossing denotes the inverse of the braiding. See Figure \ref{fig:ribbongraph} for examples of the evaluation of some ribbon graphs.\end{remark} 
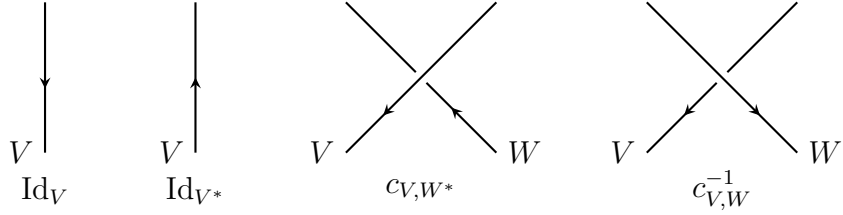
\begin{figure}
\centering
\begin{tikzpicture}[thick, >=stealth]
\begin{scope}[xshift = 0cm]
\draw[myarrow = {0.5}{<}] (0,0)node[left]{$V$} -- (0,2);
\draw (0,-0.5) node{$\text{Id}_V$};

\draw[myarrow = {0.5}{>}] (2,0)node[left]{$V$} -- (2,2);
\draw (2,-0.5) node{$\text{Id}_{V^*}$};
\end{scope}

\begin{scope}[xshift = 4cm]
\draw[myarrow = {0.3}{>}] (2,0)node[right]{$W$} -- (0,2);
\draw[myarrow = {0.3}{<}, line width = 2mm, white] (0,0) -- (2,2);
\draw[myarrow = {0.3}{<}] (0,0)node[left]{$V$} -- (2,2);
\draw (1,-0.5) node{$c_{V, W^*}$};
\end{scope}

\begin{scope}[xshift = 8cm]
\draw[myarrow = {0.3}{<}] (0,0)node[left]{$V$} -- (2,2);
\draw[myarrow = {0.3}{<}, line width = 2mm, white] (2,0)node[right]{$W$} -- (0,2);
\draw[myarrow = {0.3}{<}] (2,0)node[right]{$W$} -- (0,2);
\draw (1,-0.5) node{$c_{V, W}^{-1}$};
\end{scope}

\end{tikzpicture}
\caption{Interpretation for some  ribbon graphs}\label{fig:ribbongraph}
\end{figure}

Let $\cC$ be spherical fusion category and $S(\cC)$ be a complete set of representatives of simple objects of $\cC$, namely, $S(\cC)$ contains a representative for each isomorphism class of simple objects. For an object $a \in \cC$, let $d(a):= d_a$ be the quantum dimension of $a$.  Next we introduce the formal object, sometimes called the Kirby color, $\omega_{\cC}:= \sum_{a \in S(\cC)} d_a \, a$\,, and call $d(\omega_{\cC}) := \sum_{a \in S(\cC)}d_a^2$ the dimension of $\cC$. Let $\cD$ be a ribbon fusion category and denote by $\cD'$ the symmetric center (or Muger center) of $\cD$. If $A \in \cD$ is an object or a formal object, denote by $A'$ the subobject of $A$ which lies in $\cD'$.

Let $F: \cC \rightarrow \cD$ be a pivotal functor  such that all objects in $\cD'$ have trivial twists. For a closed 4-manifold $X$, choose for $X$ a surgery link  $L =  L_1 \sqcup L_2$, where $L_1$ is a 0-framed unlink (of dotted circles) representing the 1-handles and $L_2$ is a framed link for attaching 2-handles. Label each component of $L_1$ by $\omega_{\cD}$ and each component of $L_2$ by $F\omega_{\cC}$. Then evaluate $L$ with the above labels to a complex number $L(\omega_{\cD}, F\omega_{\cC})$. Lastly, denote by $|L_i|$ the number of components of $L_i$. The dichromatic invariant $I_F(X)$ is defined by \cite{barenz2016dichromatic}:
\begin{equation}
I_F(X) := \frac{1}{d(\omega_{\cC})^{|L_2|-|L_1|} \left(d(\omega_{\cD})d((F\omega_{\cC})')\right)^{|L_1|}}\, L(\omega_{\cD}, F\omega_{\cC}).
\end{equation} 

If $\cC$ is ribbon fusion, $\cD$ is modular, and $F$ is a full braided inclusion, then $I_{\cC}:=I_F$ turns out to depend only on $\cC$ and furthermore, it recovers the Crane-Yetter invariant:
\begin{equation}
\label{equ:dichro=CY}
I_{\cC}(X) = \text{CY}_{\cC}(X) d(\omega_{\cC})^{1-\chi(X)}.
\end{equation}

Given two semisimple Hopf algebras  $H$ and $ K$, and a Hopf algebra morphism $\phi: D(H) \to K$, there is an induced pivotal functor $\Rep(\phi): \Rep(K) \to \Rep(D(H))$, where, for a representation $V$ of $K$, $\Rep(\phi)(V)$ is the same as $V$ but viewed as a representation of $D(H)$ via $\phi$. It is well known that $\Rep(D(H))$ is equivalent to the Drinfeld center of $\Rep(H)$ and thus is modular.   In the following we give a description of the invariant $I_{\Rep(\phi)}$ in terms of Hopf algebras.

For any semisimple $H$, choose a two-sided integral $e \in H, \mu \in H^*$ such that $\mu(e) = 1,\ \epsilon(e) = 1$. Then $e$ is central, $e^2 = e$ is a projector, and $\mu(1)=\dim(H)$. In fact, for $a \in H, f \in H^*$, $\mu(a) = \Tr (L_a)$ and  $f(e) = \frac{1}{\dim(H)}\Tr(L_{f})$, where $L_a$ (respectfully $L_f$) denotes the left multiplication by $a$ (respectfully by $f$).  

\begin{lemma}
\label{lem:proj_trivial}
Let $e \in H$ be the two-sided integral as above and $V$ be a representation of $H$ with the action given by $\rho: H \rightarrow \End(V)$.  Then $V$ has a decomposition $V = \Ima \rho_e \oplus \Ker \rho_e$ as representations, and $\rho_e$ is a projection onto the trivial subrepresentation of $V$. 
\begin{proof}
This is straightforward  by noting that the the trivial irreducible representation of $H$ is $\C$ with the action given by $\epsilon$ and that $e$ is a central projector.
\end{proof}
\end{lemma}

\begin{lemma}
\label{lem:encircling}
Let $\cC = \Rep(H)$ and $F: \cC \rightarrow \cD$ be a full braided inclusion of $\cC$ into a modular tensor category $\cD$. Let $e \in H$ be the two-sided integral as above. For any object $V \in \cC$ with the action $\rho: H \rightarrow \End(V)$, the following equality holds: 
\begin{center}
\begin{tikzpicture}[thick,decoration={
    markings,
    mark=at position 0.5 with {\arrow{>}}}]
\begin{scope}
  \draw (1,0) arc(0:180: 1cm and 0.4cm);
  \draw[line width = 3mm, white] (0,0) -- (0,1);
  \draw (0,0) -- (0,1);
  
  \draw (0,0) -- (0,-1)node[left]{$F(V)$};  
  \draw[line width = 3mm, white] (1,0) arc(0:-180: 1cm and 0.4cm);
  \draw (1,0) arc(0:-180: 1cm and 0.4cm) node[left, pos = 1]{$\omega_{\mathcal{D}}$};
\end{scope}
\draw (2,0) node{$=$};
\begin{scope}[xshift = 4.3cm]
  \draw (0,-1) -- (0,1);
  \draw[fill = white] (-0.5, -0.3) rectangle (0.5, 0.3) node[pos = .5]{$F(\rho_e)$};
  \draw (-1.2,0) node{$d(\omega_{\mathcal{D}})$};
\end{scope}
\end{tikzpicture}
\end{center}
\begin{proof}
Decompose $V$ as $V = V_1 \oplus V_2$, where $V_1$ contains all copies of the unit object in $V$. Choose $\pi_i: V \to V_i, \ \iota_i: V_i \to V$ such that $\pi_j \circ \iota_i = \delta_{ij} \text{Id}_{V_i}$ and $\iota_1\circ \pi_1+ \iota_2\circ \pi_2 = \text{Id}_{V}$. Since $F$ is a full inclusion, $F(V)$ decomposes as $F(V) = F(V_1) \oplus F(V_2)$ where $F(V_1)$ corresponds to the subobject of $F(V)$ containing all copies of the unit object, and $F(\pi_i)$ and $F(\iota_i)$ satisfy similar relations as above.

It is well known (see for instance \cite{lickorish1993skein, bojko2001lectures}) that the left-hand side of the equality in this lemma is equal to $d(\omega_{\cD})$ times $ F(\iota_1) \circ F(\pi_1)$ which is the identity on $F(V_1)$ and the zero map on $F(V_2)$.  By Lemma \ref{lem:proj_trivial}, $\iota_1 \circ \pi_1 = \rho_e$\,, and hence the equality follows.
\end{proof}
\end{lemma}

For simplicity, we start with the special case of $H$ being  quasi-triangular, $K = H,$ and $ \phi : D(H) \to H$ is given by $\phi:= \phi_H:= M \circ (f_R \otimes \text{Id})$, where $f_R(q):= (q \otimes \text{Id})R$. See Example \ref{ex:basic_examples_of_triplets} from Section \ref{subsubsec:doublets_and_triplets}.  Given a surgery link $L = L_1 \sqcup L_2$ for some 4-manifold, present $L$ as a planar link diagram with respect to a height function such that the crossings are not critical points and that the framing of $L$ is given by the blackboard framing. We associate to $L$ the following tensors. Orient the components of $L_2$ arbitrarily. For a non-critical point $p \in L_2$, let $c(p) = 0$ if the orientation of $L_2$ near $p$ is downwards and $c(p) = 1$ otherwise. 

For each dotted circle in $L_1$, there can be some strands of $L_2$ intersecting the bounding disk of it. Denote the intersection points from left to right by $p_1,..., p_n$. Then assign to the dotted circle the tensor $(S^{c(p_1)} \otimes \cdots \otimes S^{c(p_n)})\Delta^n(e)$, where $\Delta^n(x) = \sum x^{(1)} \otimes \cdots \otimes x^{(n)}$. See Figure \ref{fig:dotted_circle_tensor} (Left). The $i$-th outgoing leg of the tensor is associated with $p_i$. For each crossing of $L_2$, pick a point $q_1$ (respectfully $q_2$) on the overcrossing (respectfully undercrossing) strand near the crossing. If the crossing is positive (ignoring the orientation), then assign to it the tensor $(S^{c(q_1)} \otimes S^{c(q_2)})R$ where the first outgoing leg is associated with $q_1$ and the second with $q_2$. If the crossing is negative, replace $R$ by $R^{-1}$ in the above tensor. See Figure \ref{fig:crossing_tensor}. Call all the previously chosen points ``labeled points''. Finally, for each component of $L_2$, assume there are $m$ labeled points on it. Then assign to it the tensor $\mu \circ M^m$, where the incoming legs are arranged according to the orientation of the link component and each of them corresponds to a labeled point. No base point is needed since $\mu \circ M^m$ is cyclically invariant. See Figure \ref{fig:dotted_circle_tensor} (Right). We define $L(H)$ to be the contraction of all the tensors assigned above.

\begin{figure}
\centering
\begin{tikzpicture}
\begin{scope}[thick, >=stealth]
  \draw (2,0) arc(0:180: 2cm and 0.4cm);
     
    \draw[line width = 3mm, white] (-1,0) -- (-1,1);
  \draw[line width = 3mm, white] (0,0) -- (0,1);
  \draw[line width = 3mm, white] (1,0) -- (1,1);
  \draw[myarrow = {0.8}{>}]  (-1,0) -- (-1,1);
  \draw[myarrow = {0.8}{<}]  (0,0) -- (0,1);
  \draw[myarrow = {0.8}{>}]  (1,0) -- (1,1);

  \draw (-1,0) -- (-1,-1);
  \draw (0,0) -- (0,-1);
  \draw (1,0) -- (1,-1);
  \draw[line width = 6mm, white] (2,0) arc(0:-180: 2cm and 0.4cm);
  \draw (2,0) arc(0:-180: 2cm and 0.4cm);
  \draw (0.5, 0.8) node{$\cdots$};
  
  \fill (-1,0) circle(2pt);
  \fill[white] (0,-0.1) circle(2pt);
  \fill (0,0) circle(2pt);
  \fill (1,0) circle(2pt);
  
  \draw (-1,0)node[left]{$p_1$};
  \draw (0,0)node[left]{$p_2$};
  \draw (1,0)node[left]{$p_n$};
\end{scope}

\begin{scope}
   \draw (-3,-2) node(Delta){$\Delta$};
   \draw (-4,-2) node(e) {$e$};
   \draw (-2,-1) node(S1){$S^{c(p_1)}$};   
   \draw (-1.7,-2) node(S2){$S^{c(p_2)}$};  
   \draw (-2,-3) node(S3){$S^{c(p_n)}$};    
   \draw (-1.5,-2.3) node{$\vdots$};
   \draw (-1,0) node(p1){};   
   \draw (0,0) node(p2){};  
   \draw (1,0) node(p3){};  
   
   \draw[->] (e) to (Delta);
   \draw[->] (Delta) to (S1);
   \draw[->] (Delta) to (S2);
   \draw[->] (Delta) to (S3);
  
   \draw[->, line width = 1.9mm, white] (S1) to (p1);
   \draw[->] (S1) to (p1);
   
   \draw[->, line width = 2mm, white] (S2) to[out = 0, in = -120] (p2);
   \draw[->] (S2) to[out = 0, in = -120] (p2);
   
   \draw[->, line width = 2mm, white] (S3) to[out = 0, in = -120] (p3);
   \draw[->] (S3) to[out = 0, in = -120] (p3);
\end{scope}
\end{tikzpicture}
\hspace*{2cm}
\begin{tikzpicture}
\draw[>=stealth, thick, myarrow = {0.25}{>}] (0,0) circle(1.5cm);
\fill ({1.5*cos(135)}, {1.5*sin(135)}) circle(2pt);
\fill ({1.5*cos(165)}, {1.5*sin(165)}) circle(2pt);
\fill ({1.5*cos(210)}, {1.5*sin(210)}) circle(2pt);

\draw ({1.5*cos(135)}, {1.5*sin(135)}) node(p1){};
\draw ({1.5*cos(165)}, {1.5*sin(165)}) node(p2){};
\draw ({1.5*cos(210)}, {1.5*sin(210)}) node(p3){};

\draw (0,0) node(M){$M$};
\draw (1,0) node(mu){$\mu$};

\draw[->] (p1) to (M);
\draw[->] (p2) to (M);
\draw[->] (p3) to (M);
\draw[->] (M) to (mu);

\draw (-1, 0) node{$\vdots$};
\end{tikzpicture}
\caption{(Left) Tensors associated with a dotted circle; (Right) Tensors associated with a link component}\label{fig:dotted_circle_tensor}
\end{figure}
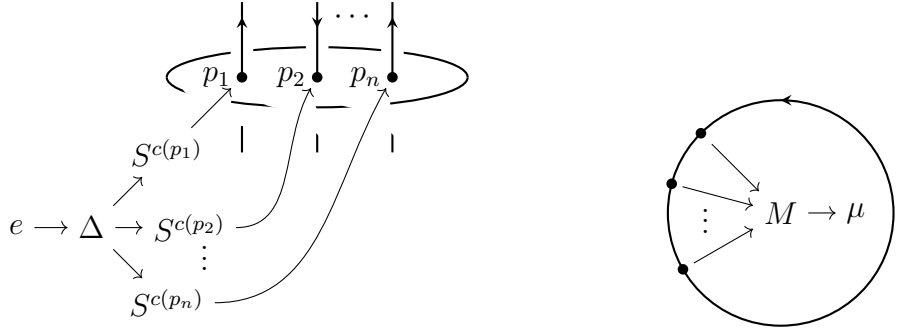
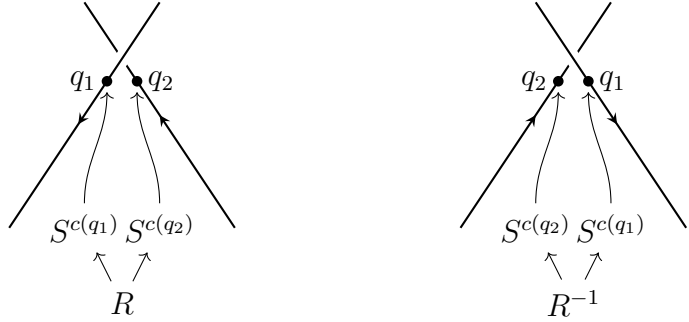
\begin{figure}
\centering
\begin{tikzpicture}
\begin{scope}[thick, >=stealth]
     \draw[myarrow = {0.5}{>}]  (3,0) -- (1,3);
     \draw[myarrow = {0.5}{<}, line width = 2mm, white]  (0,0) -- (2,3);
     \draw[myarrow = {0.5}{<}]  (0,0) -- (2,3);
\end{scope}

\begin{scope}
     \draw (1.5, -1) node(R){$R$};
     \draw (1, 0) node(S1){$S^{c(q_1)}$};
     \draw (2, 0) node(S2){$S^{c(q_2)}$};
     
     \fill (1.3, {1.3*1.5}) circle(2pt);
     \fill (1.7, {1.3*1.5}) circle(2pt);
     
     \draw (1.3, {1.3*1.5})  node[left] {$q_1$};
     \draw (1.7, {1.3*1.5})  node[right] {$q_2$};
     
     \draw (1.3, {1.3*1.5})  node(q1){};
     \draw (1.7, {1.3*1.5})  node(q2) {};
     
     \draw[->] (R) to (S1);
     \draw[->] (R) to (S2);
     \draw[->] (S1) to[out = 90, in = -90] (q1);
     \draw[->] (S2) to[out = 90, in = -90] (q2);
\end{scope}

\begin{scope}[xshift = 6cm]
\begin{scope}[thick, >=stealth]
    \draw[myarrow = {0.5}{>}]  (0,0) -- (2,3);
     \draw[myarrow = {0.5}{<}, line width = 2mm, white]  (3,0) -- (1,3);
     \draw[myarrow = {0.5}{<}]  (3,0) -- (1,3);
\end{scope}

\begin{scope}
     \draw (1.5, -1) node(R){$R^{-1}$};
     \draw (2, 0) node(S1){$S^{c(q_1)}$};
     \draw (1, 0) node(S2){$S^{c(q_2)}$};
     
     \fill (1.3, {1.3*1.5}) circle(2pt);
     \fill (1.7, {1.3*1.5}) circle(2pt);
     
     \draw (1.3, {1.3*1.5})  node[left] {$q_2$};
     \draw (1.7, {1.3*1.5})  node[right] {$q_1$};
     
     \draw (1.3, {1.3*1.5})  node(q2){};
     \draw (1.7, {1.3*1.5})  node(q1) {};
     
     \draw[->] (R) to (S1);
     \draw[->] (R) to (S2);
     \draw[->] (S1) to[out = 90, in = -90] (q1);
     \draw[->] (S2) to[out = 90, in = -90] (q2);
\end{scope}
\end{scope}
\end{tikzpicture}
\caption{(Left) Tensors associated with a positive crossing and (Right) a negative crossing. Note that the orientations are not taken into account when deciding the sign of the crossing.}\label{fig:crossing_tensor}
\end{figure}
%\begin{figure}
%\centering
%%\includegraphics[width = \textwith]{Fig4}
%\begin{tikzpicture}
%\draw[>=stealth, thick, myarrow = {0.25}{>}] (0,0) circle(1.5cm);
%\fill ({1.5*cos(135)}, {1.5*sin(135)}) circle(2pt);
%\fill ({1.5*cos(165)}, {1.5*sin(165)}) circle(2pt);
%\fill ({1.5*cos(210)}, {1.5*sin(210)}) circle(2pt);
%
%\draw ({1.5*cos(135)}, {1.5*sin(135)}) node(p1){};
%\draw ({1.5*cos(165)}, {1.5*sin(165)}) node(p2){};
%\draw ({1.5*cos(210)}, {1.5*sin(210)}) node(p3){};
%
%\draw (0,0) node(M){$M$};
%\draw (1,0) node(mu){$\mu$};
%
%\draw[->] (p1) to (M);
%\draw[->] (p2) to (M);
%\draw[->] (p3) to (M);
%\draw[->] (M) to (mu);
%
%\draw (-1, 0) node{$\vdots$};
%\end{tikzpicture}
%\caption{Tensors associated with a link component}\label{fig:link_tensor}
%\end{figure}

By the properties $\mu \circ S = \mu$ and using that $S$ is an anti-algebra morphism, it is direct to check that $L(H)$ is independent of the choice of orientations of $L$. 
\begin{remark}
When $H$ is factorizable, $L(H)$ is essentially  the link invariant introduced in \cite{kauffman1995invariants} in reformulating the Hennings invariant \cite{hennings1996invariants}. However, when $H$ is not factorizable, these two are different. See also \cite{chang2019two} for an exposition of the link invariant.
\end{remark}

Let $\cC = \Rep(H)$. Then $\omega_{\cC} = H$ where $H$ is viewed as a representation of $H$ by left multiplication, and $d(\omega_{\cC}) = \dim(H)$. Choose any modular tensor category $\cD$ such that there is a full braided inclusion $F: \cC \to \cD$. For instance, one can take $\cD = \Rep(D(H))$ and $F = \Rep(\phi_H)$ as above. Then $(F\omega_{\cC})'$ contains only the unit object and hence $d((F\omega_{\cC})') = 1$.

\begin{proposition}
\label{thm:dichro_reform}
Let $L, H, \cC, \cD, F$ be as above, then $L(\omega_{\cD}, F\omega_{\cC}) = d(\omega_{\cD})^{|L_1|} L(H)$. Therefore for a 4-manifold $X$ with surgery link $L$, 
\begin{equation}
I_{\Rep(H)}(X) = \frac{1}{\dim (H)^{|L_2|-|L_1|}}\,L(H)\,.
\end{equation}
\begin{proof}

Label each component of $L_1$ by $\omega_{\cD}$ and each component of $L_2$ by $F\omega_{\cC} = F(H)$.  Since both $\omega_{\cD}$ and $F\omega_{\cC}$ are self-dual, the orientation of $L$ does not change its evaluation. Hence, we orient each component of $L$ arbitrarily. Also, present $L$ as a planar link diagram as before.

Recall that a strand of $L_2$ directed downwards represents $F(H)$ and a strand directed upwards represents $F(H)^* = F(H^*)$. To compute $L(\omega_{\cD}, F\omega_{\cC})$, by Lemma \ref{lem:encircling}, we can replace each dotted circle by $d(\omega_{\cD}) F(\rho_e)$, where $\rho_e$ denotes the action of $e \in H$ on the strands intersecting the disk bounded by the dotted circle. Specifically, arrange all the intersecting strands from left to right and denote them by $K_1, K_2, \cdots, K_n$. See Figure \ref{fig:remove_dotted_circle}. Assume $\Delta^{n}(e) = \sum e^{(1)} \otimes \cdots \otimes e^{(n)} $ and denote the left multiplication of $e$ on $H$ by $L_e$. Then,
\begin{equation}
\rho_e = \sum_{e^{(i)}} \tilde{L}_{e^{(1)}} \otimes \cdots \otimes \tilde{L}_{e^{(n)}},
\end{equation}
where $\tilde{L}_{e^{(i)}} = L_{e^{(i)}}$ if $K_i$  is directed downwards near the disk, and otherwise $\tilde{L}_{e^{(i)}} = L_{S(e^{(i)})}^*$, the linear dual of $L_{S(e^{(i)})}$. 
\begin{figure}
\centering
\begin{tikzpicture}[thick, >=stealth]
\begin{scope}
  \draw (2,0) arc(0:180: 2cm and 0.4cm);
  \draw[line width = 3mm, white] (-1.5,0) -- (-1.5,1);
  \draw[line width = 3mm, white] (0,0) -- (0,1);
  \draw[line width = 3mm, white] (1.5,0) -- (1.5,1);
  \draw (-1.5,0) -- (-1.5,1);
  \draw (0,0) -- (0,1);
  \draw (1.5,0) -- (1.5,1);

  \draw[myarrow = {0.8}{>}] (-1.5,0) -- (-1.5,-1)node[below]{$K_1$};
  \draw[myarrow = {0.8}{<}] (0,0) -- (0,-1)node[below]{$K_2$};
  \draw[myarrow = {0.8}{>}] (1.5,0) -- (1.5,-1)node[below]{$K_n$};
  \draw[line width = 3mm, white] (2,0) arc(0:-180: 2cm and 0.4cm);
  \draw (2,0) arc(0:-180: 2cm and 0.4cm) node[left, pos = 1]{$\omega_{\mathcal{D}}$};
  \draw (0.75, -0.8) node{$\cdots$};
\end{scope}

\draw[->,decorate,decoration={snake,amplitude=.4mm,segment length=2mm,post length=1mm}] (3,0) -- (4,0);

\begin{scope}[xshift = 8.3cm]
  \draw (-1.5,0) -- (-1.5,1);
  \draw (0,0) -- (0,1);
  \draw (1.5,0) -- (1.5,1);
  \draw[myarrow = {0.8}{>}] (-1.5,0) -- (-1.5,-1)node[below]{$K_1$};
  \draw[myarrow = {0.8}{<}] (0,0) -- (0,-1)node[below]{$K_2$};
  \draw[myarrow = {0.8}{>}] (1.5,0) -- (1.5,-1)node[below]{$K_n$};
  \draw (0.75, -0.8) node{$\cdots$};

  \draw[fill = white] (-0.5-1.5, -0.3) rectangle (0.5-1.5, 0.3) node[pos = .5]{$L_{e^{(1)}}$};
  \draw[fill = white] (-0.5, -0.3) rectangle (0.5, 0.3) node[pos = .5]{$L^*_{e^{(2)}}$};
  \draw[fill = white] (-0.5+1.5, -0.3) rectangle (0.5+1.5, 0.3) node[pos = .5]{$L_{e^{(n)}}$};
  
  \draw (-2.5,0) node{$\sum\limits_{e^{(i)}}$};
  
  \draw[thin] (-2.7, -1) to[round left paren] (-2.7, 1);
  \draw[thin] (2.2, -1) to[round right paren] (2.2, 1);
  
  \draw (-3.3,0) node{$F$};
\end{scope}

\end{tikzpicture}
\caption{Evaluation around a dotted circle}\label{fig:remove_dotted_circle}
\end{figure}
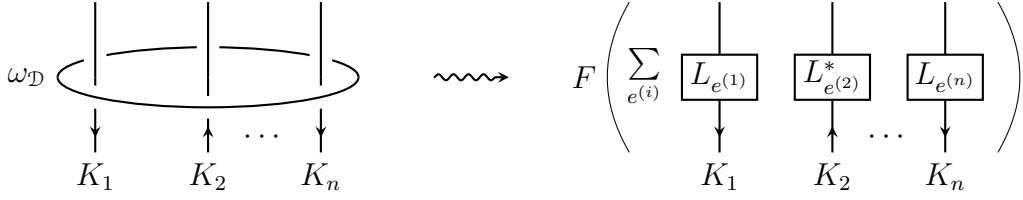

Away from the dotted circles, the constituents of the link consist of crossings, caps, and cups, all from $L_2$. Since $F$ is a braided inclusion, we can hence forget about $F$, label all components of $L_2$ by $H$, and replace each dotted circle by $d(\omega_{\cD}) \rho_e$ as discussed in the previous paragraph. We can then evaluate $L_2$ entirely within $\cC$. Note that after removing all the dotting circles, we obtain an extra scalar factor $d(\omega_{\cD})^{|L_1|}$. (This also shows that $L(\omega_{\cD}, F\omega_{\cC}) $ depends on $\cD$ only for the factor $d(\omega_{\cD})^{|L_1|}$ which is canceled later in the renormalization of $I_F$ and hence $I_F$ depends only on $\cC$.)

At each crossing of $L_2$, denote the over-crossing strand by $K_1$ and the under-crossing strand by $K_2$. If the crossing is positive, then the morphism it represents is $\sum \tilde{L}_{R_1} \otimes \tilde{L}_{R_2}$ where $R = \sum R_1 \otimes R_2$ and $\tilde{L}_{R_i}$ acts on the strand $K_i$ with the same convention as before. If the crossing is negative, then replace $R$ by $R^{-1}$.

Now for each component of $L_2$, the maps on it (fixing a summation term) have the form of either $L_{a}$ or $L_{S(b)}^*$ depending on the orientation. Also note that $\mu(a) = \Tr L_{a}$\,. Then it follows that the evaluation of that component is obtained by  multiplying the $a\,'$s and $S(b)\,'$s  along the orientation followed by the action of $\mu$. In other words, we apply the tensor as in Figure \ref{fig:dotted_circle_tensor} (Right) to get the evaluation. Hence we have $L(\omega_{\cD}, F\omega_{\cC}) = d(\omega_{\cD})^{|L_1|} L(H)$. The second part in the proposition follows immediately.

\end{proof}
\end{proposition}

In the general case $\phi: D(H) \to K$ for two Hopf algebras $H, K$ not necessarily quasi-triangular, the description of $I_{\Rep(\phi)}$ for the induced functor $\Rep(\phi)$ can be given analogously. Choose two-sided integrals $e_H \in H, \mu_H \in H^*$ as above. To avoid confusion, we have introduced subscripts to indicate which Hopf algebra we are working with. Similarly choose integrals $e_K \in K, \mu_K \in K^*, e_D \in D(H), \mu_D \in D(H)^*$. It can be checked that $e_D = \frac{1}{\dim (H)}\, \mu_H \otimes e_H, \ \mu_D = \dim (H) \,e_H \otimes \mu_H$. Let $\phi_1$ (respectfully $\phi_2$) be the restriction of $\phi$ on $H^{*,\Cop}$ (respectfully $H$), then $\phi = M_K \circ (\phi_1 \otimes \phi_2)$. Also the universal $R$-matrix for $D(H)$ is given by $R_D = \sum_{i} (\epsilon \otimes v_i) \otimes (v_i^* \otimes 1)$, where $\{v_1, v_2,...\}$ is a basis of $H$.

Given a surgery link $L = L_1 \sqcup L_2$ representing a manifold $X$,  we introduce $L(H, K; \phi)$ which is similar to $L(H)$, and we only point out the relevant modifications based on $L(H)$. Firstly, for the tensor assigned to each dotted circle, replace $e$ by $\tilde{e}:= \phi(e_D) = \phi(\frac{1}{\dim(H)} \mu_H \otimes e_H)$. Secondly, for each crossing of $L_2$, replace the $R$-matrix by $\tilde{R}:= (\phi \otimes \phi)R_D = \sum_{i} \phi_2(v_i) \otimes \phi_1(v_i^*)$. It can be checked that $\tilde{R}^{-1} =\sum_{i} \phi_2(S(v_i)) \otimes \phi_1(v_i^*) $. Lastly, due to the use of notations, replace all operations taking in $H$ with the relevant operations in $K$. Letting $\cC = \Rep(K), \cD = \Rep(D(H))$, then we have $L(\omega_{\cD}, F\omega_{\cC}) = d(\omega_{\cD})^{|L_1|} L(H,K;\phi)$. Note that $\omega_{\cC} = K, \ d(\omega_{\cC}) = \dim(K), $ and $ d((\Rep(\phi) \omega_{\cC})')$ is the rank of map $L_{\tilde{e}}$. Since $L_{\tilde{e}}$ is a projector, its rank is equal to its trace. Hence $d((\Rep(\phi) \omega_{\cC})') = \Tr(L_{\tilde{e}}) = \mu_K(\tilde{e})$. We have
\begin{equation}
I_{\Rep(\phi)}(X) = \frac{1}{\dim (K)^{|L_2|-|L_1|}} \frac{1}{(\mu_K(\tilde{e}))^{|L_1|}}\,L(H,K;\phi).
\end{equation}

\subsection{Proof of Theorem \ref{thm:trisection_vs_crane_yetter_informal}/\ref{thm:tri=dichro}}
\label{subsec:trisection_dichro}

Let $H, K$ be any semisimple Hopf algebras over a field of characteristic zero, and $\phi: D(H) \to K$ be a Hopf algebra morphism. Denote the restriction of $\phi$ on $H^{*,\Cop}$ by $\phi_1: H^{*,\Cop} \to K$ and that on $H$ by $\phi_2: H \to K$. Then both $\phi_1 $ and $\phi_2$ are Hopf algebra morphisms and $\phi = M \circ (\phi_1 \otimes \phi_2)$.  According to Corollary \ref{cor:map_version_of_triplet}, we can construct a Hopf triplet $\Hopf{H, K; \phi} = (H_{\alpha}, H_{\beta}, H_{\kappa}; \langle\;, \;\rangle)$ by letting  $H_{\alpha} = H^{*, \Cop, \Op},\  H_{\beta} = H^{\Cop}$ and  $H_{\kappa} = K^*$ with the pairings defined as follows. The pairing $\pair_{\alpha,\beta}$ on $H_{\alpha} \otimes H_{\beta}$ is the canonical one, $\pair_{\beta,\kappa}$ on $H_{\beta} \otimes H_{\kappa}$ is given by $\langle h, p\rangle_{\beta,\kappa} = p \circ \phi_2(h)$, and  $\pair_{\kappa,\alpha}$ on $H_{\kappa} \otimes H_{\alpha}$ is given by $\langle p, q\rangle_{\kappa,\alpha} = p \circ \phi_1(q)$. 

In particular, if $(H, R)$ is a quasi-triangular Hopf algebra with the universal $R$-matrix $R \in H \otimes H$. Then $\phi_H: D(H) \to H$ defined by $\phi_H = M \circ (f_R \otimes \text{Id})$ is a Hopf algebra morphism, where $f_R(q) := (q \otimes Id)R$. Hence one can obtain the Hopf triplet $\mathcal{H}_H := \Hopf{H, H; \phi_H}$ (see also Example \ref{ex:basic_examples_of_triplets}).

Let $\Rep(\phi): \Rep(K) \to \Rep(D(H))$ be the induced functor. Recall that $\tilde{e} = \phi(e_D) = \phi(\frac{1}{\dim(H)} \,\mu_H \otimes e_H)$, where $e_H, \mu_H$ are integrals defined in \S\ref{subsec:dichro}.   In this subsection, we prove the following theorem, which is the promised refinement of Theorem~\ref{thm:trisection_vs_crane_yetter_informal} which we stated in the Introduction.
\begin{theorem}
\label{thm:tri=dichro}
For the above setup and  for any closed $4$-manifold $X$, the following equality holds,
\begin{equation}
\tau_{\Hopf{H,K;\phi}}(X) = \left[\frac{\dim(H)\mu_K(\tilde{e})}{\dim(K)}\right]^{\frac{2-\chi(X)}{3}}\,I_{\text{Rep}(\phi)}(X),
\end{equation}
where $\chi(X)$ is the Euler characteristic of $X$.
\end{theorem}
If $H$ is quasi-triangular and we take $K = H, \phi = M \circ (f_R \otimes \text{Id})$, it is direct to check that $\mu_K(\tilde{e}) = 1$, and hence $\tau_{\Hopf{H}}(X) = I_{\Rep(H)}(X)$.

The following corollary follows immediately from Theorem \ref{thm:tri=dichro} and Equation  \ref{equ:dichro=CY}.
\begin{corollary}
\label{cor:trisection=cy}
Let $H$ be any quasi-triangular semisimple Hopf algebra.  Then for any closed $4$-manifold $X$, $\tau_{\Hopf{H}}(X)  = \text{CY}_{\text{Rep}(H)}(X) \dim(H)^{1-\chi(X)}$.
\end{corollary}

\begin{remark}
Given a generalized Drinfeld double $D_{\varphi}(H_{\alpha}, H_{\beta})$ where $\varphi: H_{\alpha} \to H_{\beta}^{*,\Cop}$ is a Hopf algebra morphism, in general $\Rep(D_{\varphi}(H_{\alpha}, H_{\beta}))$ is not a braided fusion category. By \cite{burciu2012irreducible}, it is  equivalent to $\mathcal{Z}_{\Rep(\varphi^{*,\Op})}(\Rep(H_{\beta}))$, the relative center of $\Rep(\varphi^{*,\Op})$, where $\Rep(\varphi^{*,\Op}): \Rep(H_{\alpha}^{*,\Op}) \to \Rep(H_{\beta})$ is the induced functor of $\varphi^{*,\Op}: H_{\beta} \to H_{\alpha}^{*,\Op} $. Therefore, given a Hopf algebra morphism $\phi: D_{\varphi}(H_{\alpha}, H_{\beta}) \to H_{\kappa}$\,, the induced functor $\Rep(\phi): \Rep(H_{\kappa}) \to \Rep(D_{\varphi}(H_{\alpha}, H_{\beta})) $ can not be used to define the dichromatic invariant. This suggests that the trisection invariant arising from the most general Hopf triplet $(H_{\alpha}, H_{\beta}, H_{\kappa}; \langle-\rangle)$ is different from the dichromatic invariant.
\end{remark}

The rest of the subsection is devoted to the proof of Theorem \ref{thm:tri=dichro}. For readers who are familiar with quasi-triangular Hopf algebras, they may find it easier to follow the proof first for the special case of $H$ being quasi-triangular and $\phi = M \circ (f_R \otimes \text{Id})$.

The trisection invariant is defined in terms of trisection diagrams while the dichromatic invariant is defined in terms of surgery links. So we need a translation between the trisection diagrams and surgery links. We provide a concrete translation.

Think of $\Sp^3 = \R^3 \cup \{\infty\}$, $\Sp^2 = (\R^2 \times \{0\}) \cup \{\infty\}$. We identify $\R^2$ with $\R^2 \times \{0\}$. Denote by $D_{(x,y)}(r)$ the $2$-ball centered at $(x,y)$ with radius $r$. Remove the balls $D_{(\pm 1, i)}(\epsilon) \subset \Sp^2$, $i = 1,...,g$, for some small $\epsilon \ll 1$. Let $\Sp_{(\pm 1,i)}(\epsilon) = \partial D_{(\pm 1, i)}(\epsilon) $. We identify $\Sp_{( 1,i)}(\epsilon)$ with $\Sp_{( -1,i)}(\epsilon)$ by the reflection about the $y$-axis. The resulting surface is clearly $\Sigma_g$\,, a closed surface of genus $g$. Equivalently, $\Sigma_g$ can be thought as being obtained by removing the $D_{(\pm 1, i)}(\epsilon)\,'$s from $\Sp^2$ and then for each $i$, gluing a \lq bridge' $h_i = \Sp^1 \times [0,1]$ such that $\Sp^1 \times \{0\}$ is identified with $S_{(-1,i)}(\epsilon)$ and $\Sp^1 \times \{1\}$  with $S_{(1,i)}(\epsilon)$. To be explicit, we embed the interior of the bridges $h_i$ in $\R^2 \times (-\infty, 0)$ so that each of them is unlinked with the rest. 

Let $0 \leq k \leq g$. For $1 \leq i \leq k$, let $\alpha_i = \Sp_{(1,i)}(2\epsilon),\ \beta_i = \Sp_{(1,i)}(3\epsilon)$. For $k+1 \leq i \leq g$, let $\alpha_i = \Sp_{(1,i)}(2\epsilon)$ and $\beta_i$ be the longitude of $\Sigma_g$ passing through the $i$-th bridge $h_i$. Let $\alpha = \{\alpha_1, \cdots, \alpha_g\}$ and $\beta = \{\beta_1, \cdots, \beta_g\}$.  Then $(\Sigma_g, \alpha,\beta)$ is the standard Heegaard diagram of $\#^k\; \Sp^1 \times \Sp^2$. See Figure \ref{fig:special_trisection} (Left). 

For any closed 4-manifold $X$, we can always, for some $g,k$,  choose a $(g,k)$ trisection diagram $T = (\Sigma_g, \alpha,\beta, \kappa)$  such that  $(\Sigma_g, \alpha,\beta)$ is given as above. We can arrange the $\kappa$ curves so that they travel through the bridges $h_i\,'$s in parallel with the longitude. Moreover, every time a $\kappa$ curve travels through $h_i$, it crosses $\alpha_i$ once and if $i \leq k$ it also crosses $\beta_i$ once.  The $\kappa$ curves are otherwise away from the $D_{(\pm 1, i)}(3\epsilon)$ disks. 

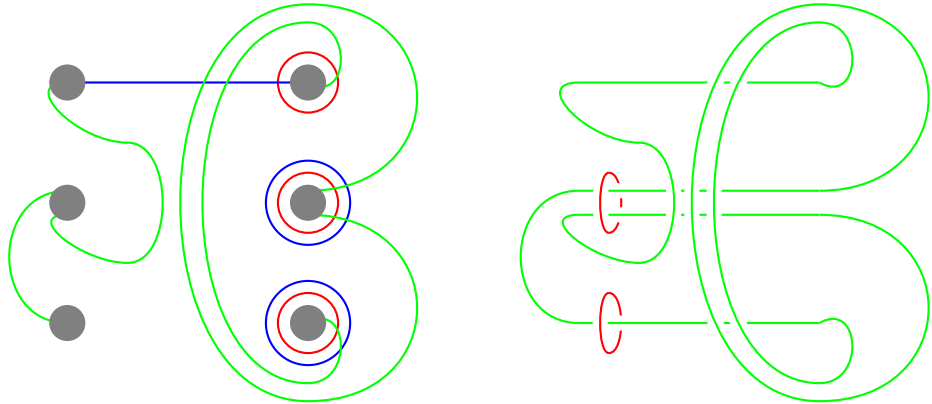
\begin{figure}
\centering
\begin{tikzpicture}[thick, scale=0.8, every node/.style={scale=0.8}]
\draw[red] (4,0) circle(0.5cm);
\draw[red] (4,2) circle(0.5cm);
\draw[red] (4,4) circle(0.5cm);

\draw[blue] (4,0) circle(0.7cm);
\draw[blue] (4,2) circle(0.7cm);
\draw[blue] (0,4) to (4,4);

\draw[green] (0,0) to[out = 180, in = 180, looseness = 1.5] (0,2+0.2);
\draw[green] (0,2-0.2) to[out = 180, in = 180, looseness = 1.5] (1,1) to[out = 0,in = 0] (1,3) to[out = 180, in = 180,looseness = 1.5] (0,4);
\draw[green] (4,0) to[out = 30, in = 0, looseness = 2] (4, -1) to[out = 180, in = 180] (4,5) to[out = 0, in = -30, looseness = 2] (4,4);
\draw[green] (4,2-0.2) to[out = 0, in = 0, looseness = 2] (4, -1.3) to[out = 180, in = 180,looseness = 1.1] (4,5.3) to[out = 0, in = 0, looseness = 2] (4,2+0.2);

\fill[gray] (0,0) circle(0.3cm);
\fill[gray] (4,0) circle(0.3cm);
\fill[gray] (0,2) circle(0.3cm);
\fill[gray] (4,2) circle(0.3cm);
\fill[gray] (0,4) circle(0.3cm);
\fill[gray] (4,4) circle(0.3cm);

\begin{scope}[thick, xshift = 8.5cm]

\draw [red] (0.5,0.5) to [out = 0, in = 0,looseness = 0.7] (0.5,-0.5);
\draw [red] (0.5,2.5) to [out = 0, in = 0,looseness = 0.7] (0.5,1.5);

\draw[white, line width  = 2mm]  (0,0) to (4,0);
\draw[white, line width  = 2mm]  (0,2-0.2) to (4,2-0.2);
\draw[white, line width  = 2mm]  (0,2+0.2) to (4,2+0.2);
\draw[white, line width  = 2mm]  (0,4) to (4,4);

\draw[green] (0,0) to (4,0);
\draw[green] (0,2-0.2) to (4,2-0.2);
\draw[green] (0,2+0.2) to (4,2+0.2);
\draw[green] (0,4) to (4,4);

\draw[white, line width  = 2mm] (0,0) to[out = 180, in = 180, looseness = 1.5] (0,2+0.2);
\draw[white, line width  = 2mm]  (0,2-0.2) to[out = 180, in = 180, looseness = 1.5] (1,1) to[out = 0,in = 0] (1,3) to[out = 180, in = 180,looseness = 1.5] (0,4);
\draw[white, line width  = 2mm]  (4,0) to[out = 30, in = 0, looseness = 2] (4, -1) to[out = 180, in = 180] (4,5) to[out = 0, in = -30, looseness = 2] (4,4);
\draw[white, line width  = 2mm]  (4,2-0.2) to[out = 0, in = 0, looseness = 2] (4, -1.3) to[out = 180, in = 180,looseness = 1.1] (4,5.3) to[out = 0, in = 0, looseness = 2] (4,2+0.2);

\draw[green] (0,0) to[out = 180, in = 180, looseness = 1.5] (0,2+0.2);
\draw[green] (0,2-0.2) to[out = 180, in = 180, looseness = 1.5] (1,1) to[out = 0,in = 0] (1,3) to[out = 180, in = 180,looseness = 1.5] (0,4);
\draw[green] (4,0) to[out = 30, in = 0, looseness = 2] (4, -1) to[out = 180, in = 180] (4,5) to[out = 0, in = -30, looseness = 2] (4,4);
\draw[green] (4,2-0.2) to[out = 0, in = 0, looseness = 2] (4, -1.3) to[out = 180, in = 180,looseness = 1.1] (4,5.3) to[out = 0, in = 0, looseness = 2] (4,2+0.2);

\draw [white, line width  = 2mm]  (0.5,0.5) to [out = 180, in = 180, looseness = 0.5] (0.5,-0.5);
\draw[white, line width  = 2mm] (0.5,2.5) to [out = 180, in = 180, looseness = 0.5] (0.5,1.5);

\draw [red] (0.5,0.5) to [out = 180, in = 180, looseness = 0.5] (0.5,-0.5);
\draw [red] (0.5,2.5) to [out = 180, in = 180, looseness = 0.5] (0.5,1.5);

\end{scope}
\end{tikzpicture}
\caption{(Left) A special $(g,k)$ trisection diagram for $g = 3, k = 2$, where the $\alpha$ curves are in red, the $\beta$ curves in blue, and the $\kappa$ curves (only one is shown) in green; (Right) The surgery link resulted from the trisection diagram on the left, where the dotted circles corresponding to 1-handles are in red.}\label{fig:special_trisection}
\end{figure}

According to \cite{gk2016}, a surgery link $L^T = L_1 \sqcup L_2$ for $X$ can be obtained as follows. For each $1 \leq i\leq k$, push the $\alpha_i\,'$s slightly off $\Sigma_g$ and designate those deformed curves as the dotted circles $L_1$. The attaching maps of 2-handles are simply given by $L_2 := \kappa$ with the framing determined by the surface $\Sigma_g$. We project $L^T$ to the plane $\R^2 \times \{0\}$ to get a link diagram. See Figure \ref{fig:special_trisection} (Right).

Let $\Hopf{} = \Hopf{H,K;\phi}$. We now compare $\langle T \rangle_{\Hopf{}} $ with $L^T(H,K;\phi)$ defined in \S\ref{subsec:dichro}. Note that each $\kappa$ curve, viewed as either a trisection component or a link component, is assigned the same tensor, $\mu\circ M^{m}$, in their respective evaluations. Also note that the $\alpha$ and $\beta$ curves are fixed and their intersections with the $\kappa$ curves can be easily characterized. In particular, the intersections of the $\alpha$ curves and the $\kappa$ curves are all located near the $D_{(1, i)}(2\epsilon)$ disks. Therefore, in evaluating $\langle T \rangle_{\Hopf{}} $, we first contract the tensors along $\alpha$ and $\beta$ curves resulting in some tensors only associated with $\kappa$ curves, and then we compare those tensors with the tensors in the definition of  $L^T(H,K;\phi)$. This process is divided into two cases depending on whether  $i \leq k$ or $i > k$ for $\alpha_i$\,.

For $i \leq k$, see Figure \ref{fig:trisection_link_1} (Left) for a configuration of the trisection $T$ involving $\alpha_i$ and $\beta_i$, and the same figure (Right) for the corresponding configuration in the surgery link $L^T$. Contracting the tensors along $\alpha_i$ and $\beta_i$  is carried out in Figure \ref{fig:trisection_contr_1}. Respectively, the tensors associated with the corresponding configuration in the surgery link are also contracted, as shown in Figure~\ref{fig:link_contr_1}. As we can see, the two calculations result in the same tensor except the first has an extra factor $\dim (H)$. It should be noted that although the tensor contractions in Figures~\ref{fig:trisection_contr_1} and~\ref{fig:link_contr_1} are for the specific trisection/link configuration in Figure~\ref{fig:trisection_link_1}, the contractions for general configurations are exactly analogous. In particular, one can try to reverse some of the orientations in Figure \ref{fig:trisection_link_1} and compare the calculations.

\begin{figure}
\centering
\begin{tikzpicture}[thick, >= stealth]
\begin{scope}
\draw[red, myarrow = {0.25}{>}] (4,0) circle(0.5cm);
\draw[blue, myarrow = {0.2}{<}] (4,0) circle(0.7cm);

\fill (4, -0.5) circle(2pt);
\fill (4, 0.7) circle(2pt);

\draw[green, myarrow = {0.5}{>}] (-1,0.7) to [out = 0, in =150] (0,0);
\draw[green, myarrow = {0.5}{>}]  (-1,0) to [out = 0, in =180] (0,0);
\draw[green, myarrow = {0.5}{>}]  (-1,-0.7) to [out = 0, in =-150] (0,0);

\draw[green, myarrow = {0.8}{>}] (4,0) to [out = 30, in =180] (5,0.7) to (5.3,0.7);
\draw[green, myarrow = {0.8}{>}] (4,0) to [out = 0, in =180] (5,0) to (5.3,0);
\draw[green, myarrow = {0.8}{>}] (4,0) to [out = -30, in =180] (5,-0.7) to (5.3,-0.7);

\draw[green, myarrow = {0.7}{>}] (1.5,1) to (1.5, -1);
\draw[green, myarrow = {0.7}{>}] (2.5,1) to (2.5, -1);

\fill[gray] (0,0) circle(0.3cm);
\fill[gray] (4,0) circle(0.3cm);
\end{scope}

\begin{scope}[xshift = 9cm]
\draw [red] (0.5,0.5) to [out = 0, in = 0,looseness = 0.7] (0.5,-0.5);

\draw[myarrow = {0.8}{>},white, line width  = 2mm] (-1,0.7) to [out = 0, in =180] (0,0.3) to (3,0.3) to[out = 0, in = 180] (4, 0.7);
\draw[myarrow = {0.8}{>},white, line width  = 2mm] (-1,0) to (4,0);
\draw[myarrow = {0.8}{>},white, line width  = 2mm] (-1,-0.7) to [out = 0, in =180] (0,-0.3) to (3,-0.3) to[out = 0, in = 180] (4, -0.7);

\draw[green, myarrow = {0.8}{>}] (-1,0.7) to [out = 0, in =180] (0,0.3) to (3,0.3) to[out = 0, in = 180] (4, 0.7);
\draw[green, myarrow = {0.8}{>}] (-1,0) to (4,0);
\draw[green, myarrow = {0.8}{>}] (-1,-0.7) to [out = 0, in =180] (0,-0.3) to (3,-0.3) to[out = 0, in = 180] (4, -0.7);

\draw[myarrow = {0.8}{>},white, line width  = 2mm] (1.5,1) to (1.5, -1);
\draw[myarrow = {0.8}{>},white, line width  = 2mm] (2.5,1) to (2.5, -1);

\draw[green, myarrow = {0.8}{>}] (1.5,1) to (1.5, -1);
\draw[green, myarrow = {0.8}{>}] (2.5,1) to (2.5, -1);

\draw [red,white, line width  = 2mm] (0.5,0.5) to [out = 180, in = 180,looseness = 0.5] (0.5,-0.5);
\draw [red] (0.5,0.5) to [out = 180, in = 180, looseness = 0.5] (0.5,-0.5);
\end{scope}
\end{tikzpicture}
\caption{A configuration of trisections (Left) near the $i$-th bridge for $i \leq k$ and (Right) the corresponding configuration of surgery links. }\label{fig:trisection_link_1}
\end{figure}
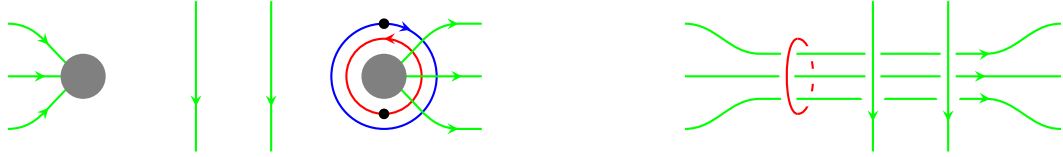

\begin{figure}
\centering
\begin{tikzpicture}
\begin{scope}
\draw (0,0) node(eH){$e_H$};
\draw (0,2) node(muH){$\mu_H$};

\draw (1,0) node(DeltaH){$\Delta_H$};
\draw (1,2) node(MH){$M_H$};

\draw (2,-0.5) node(phi2bot){$\phi_2$};
\draw (2,0) node(phi2mid){$\phi_2$};
\draw (2,0.5) node(phi2top){$\phi_2$};

\draw (2,2-0.5) node(phi1bot){$\phi_1$};
\draw (2,2-0) node(phi1mid){$\phi_1$};
\draw (2,2+0.5) node(phi1top){$\phi_1$};

\draw (4,-0.5) node(MKbot){$M_K$};
\draw (4,1) node(MKmid){$M_K$};
\draw (4,2.5) node(MKtop){$M_K$};

\draw (5,-0.5) node(Nullbot){};
\draw (5,1) node(Nullmid){};
\draw (5,2.5) node(Nulltop){};

\draw[->] (eH) to (DeltaH);
\draw[->] (MH) to (muH);
\draw[->] (DeltaH) to (phi2bot);
\draw[->] (DeltaH) to (phi2mid);
\draw[->] (DeltaH) to (phi2top);
\draw[<-] (MH) to (phi1bot);
\draw[<-] (MH) to (phi1mid);
\draw[<-] (MH) to (phi1top);
\draw[->] (phi2bot) to (MKbot);
\draw[->] (phi2mid) to (MKmid);
\draw[->] (phi2top) to (MKtop);
\draw[->] (phi1bot) to (MKbot);
\draw[->] (phi1mid) to (MKmid);
\draw[->] (phi1top) to (MKtop);
\draw[->] (MKbot) to (Nullbot);
\draw[->] (MKmid) to (Nullmid);
\draw[->] (MKtop) to (Nulltop);
\end{scope}

\draw (6, 1) node{$=$};

\begin{scope}[xshift = 7cm]
\draw (0,0) node(eH){$e_H$};
\draw (0,2) node(muH){$\mu_H$};

\draw (1,0) node(phi2){$\phi_2$};
\draw (1,2) node(phi1){$\phi_1$};

\draw (2,0) node(DeltaKbot){$\Delta_K$};
\draw (2,2) node(DeltaKtop){$\Delta_K$};

\draw (4,-0.5) node(MKbot){$M_K$};
\draw (4,1) node(MKmid){$M_K$};
\draw (4,2.5) node(MKtop){$M_K$};

\draw (5,-0.5) node(Nullbot){};
\draw (5,1) node(Nullmid){};
\draw (5,2.5) node(Nulltop){};

\draw[->] (eH) to (phi2);
\draw[<-] (muH) to (phi1);
\draw[->] (phi2) to (DeltaKbot);
\draw[->] (phi1) to (DeltaKtop);
\draw[->] (DeltaKbot) to (MKbot);
\draw[->] (DeltaKbot) to (MKmid);
\draw[->] (DeltaKbot) to (MKtop);
\draw[->] (DeltaKtop) to (MKbot);
\draw[->] (DeltaKtop) to (MKmid);
\draw[->] (DeltaKtop) to (MKtop);
\draw[->] (MKbot) to (Nullbot);
\draw[->] (MKmid) to (Nullmid);
\draw[->] (MKtop) to (Nulltop);
\end{scope}

\begin{scope}[yshift = -4cm]
\draw (-1, 1) node{$=$};

\draw (0,0) node(eH){$e_H$};
\draw (0,2) node(muH){$\mu_H$};

\draw (1,0) node(phi2){$\phi_2$};
\draw (1,2) node(phi1){$\phi_1$};

\draw (2,1) node(MK){$M_K$};

\draw (4,1) node(DeltaK){$\Delta_K$};

\draw (5,0) node(Nullbot){};
\draw (5,1) node(Nullmid){};
\draw (5,2) node(Nulltop){};

\draw[->] (eH) to (phi2);
\draw[<-] (muH) to (phi1);
\draw[->] (phi2) to (MK);
\draw[->] (phi1) to (MK);
\draw[->] (MK) to (DeltaK);
\draw[->] (DeltaK) to[out = -30, in = 180] (Nullbot);
\draw[->] (DeltaK) to (Nullmid);
\draw[->] (DeltaK) to[out = 30 ,in = 180] (Nulltop);

\draw (6, 1) node{$=$};
\end{scope}

\begin{scope}[yshift = -4cm, xshift = 7cm]
\draw (0,1) node{$\text{dim}(H)$};

\draw (1,1) node(e){$\tilde{e}$};

\draw (2,1) node(DeltaK){$\Delta_K$};

\draw (3,0) node(Nullbot){};
\draw (3,1) node(Nullmid){};
\draw (3,2) node(Nulltop){};

\draw[->] (e) to (DeltaK);
\draw[->] (DeltaK) to[out = -30, in = 180] (Nullbot);
\draw[->] (DeltaK) to (Nullmid);
\draw[->] (DeltaK) to[out = 30 ,in = 180] (Nulltop);
\end{scope}

\end{tikzpicture}
\caption{Contracting tensors of the trisection corresponding to the configuration in Figure \ref{fig:trisection_link_1} (Left) }\label{fig:trisection_contr_1}
\end{figure}
\begin{figure}
\centering
\begin{tikzpicture}
\begin{scope}
\draw (0,0) node(phi1){$\phi_1$};
\draw (0,1) node(phi2){$\phi_2$};

\draw (1,2) node(DeltaK){$\Delta_K$};

\draw (1,1) node(e){$\tilde{e}$};
\draw (2,2.3) node(Nulltop1){};
\draw (2,1.7) node(Nullbot1){};

\draw (2,0) node(DeltaKbot){$\Delta_K$};
\draw (2,1) node(DeltaKtop){$\Delta_K$};

\draw (4,0) node(MKbot){$M_K$};
\draw (4,1) node(MKmid){$M_K$};
\draw (4,2) node(MKtop){$M_K$};

\draw (5,0) node(Nullbot){};
\draw (5,1) node(Nullmid){};
\draw (5,2) node(Nulltop){};

\draw[->] (phi1) to[out = 180, in = 180, looseness = 2] (phi2);
\draw[->] (phi1) to (DeltaKbot);
\draw[->] (phi2) to (DeltaK);
\draw[->] (DeltaK) to (Nullbot1);
\draw[->] (DeltaK) to (Nulltop1);
\draw[->] (e) to (DeltaKtop);
\draw[->] (DeltaKbot) to (MKbot);
\draw[->] (DeltaKbot) to (MKmid);
\draw[->] (DeltaKbot) to (MKtop);
\draw [->](DeltaKtop) to (MKbot);
\draw [->](DeltaKtop) to (MKmid);
\draw[->] (DeltaKtop) to (MKtop);
\draw [->](MKbot) to (Nullbot);
\draw[->] (MKmid) to (Nullmid);
\draw[->] (MKtop) to (Nulltop);
\end{scope}

\begin{scope}[yshift = -3cm]
\draw (-2,1) node{$=$};

\draw (0,0) node(phi1){$\phi_1$};
\draw (0,1) node(phi2){$\phi_2$};

\draw (1,2) node(DeltaK){$\Delta_K$};

\draw (1,1) node(e){$\tilde{e}$};
\draw (2,2.3) node(Nulltop1){};
\draw (2,1.7) node(Nullbot1){};

\draw (2,1) node(MK){$M_K$};

\draw (3,1) node(DeltaK2){$\Delta_K$};

\draw (4,0) node(Nullbot){};
\draw (4,1) node(Nullmid){};
\draw (4,2) node(Nulltop){};

\draw[->] (phi1) to[out = 180, in = 180, looseness = 2] (phi2);
\draw[->] (phi1) to[out = 0, in = -135] (MK);
\draw[->] (phi2) to (DeltaK);
\draw[->] (DeltaK) to (Nullbot1);
\draw[->] (DeltaK) to (Nulltop1);
\draw[->] (e) to (MK);
\draw[->] (MK) to (DeltaK2);
\draw [->](DeltaK2) to[out = -45, in = 180] (Nullbot);
\draw[->] (DeltaK2) to (Nullmid);
\draw[->] (DeltaK2) to[out = 45, in = 180] (Nulltop);

\draw (5,1) node{$=$};
\end{scope}

\begin{scope}[yshift = -3cm, xshift = 6cm]

\draw (1,1.7) node(unitbot){$1_K$};
\draw (1,2.3) node(unittop){$1_K$};

\draw (1,1) node(e){$\tilde{e}$};
\draw (2,2.3) node(Nulltop1){};
\draw (2,1.7) node(Nullbot1){};

\draw (2,1) node(DeltaK){$\Delta_K$};

\draw (3,0) node(Nullbot){};
\draw (3,1) node(Nullmid){};
\draw (3,2) node(Nulltop){};

\draw[->] (unitbot) to (Nullbot1);
\draw[->] (unittop) to (Nulltop1);
\draw[->] (e) to (DeltaK);
\draw [->](DeltaK) to[out = -45, in = 180] (Nullbot);
\draw[->] (DeltaK) to (Nullmid);
\draw[->] (DeltaK) to[out = 45, in = 180] (Nulltop);
\end{scope}

\end{tikzpicture}
\caption{Contracting tensors of the surgery link corresponding to the configuration in Figure \ref{fig:trisection_link_1} (Right).}\label{fig:link_contr_1}
\end{figure}

Similarly, see Figures~\ref{fig:trisection_link_2} and~\ref{fig:trisection_contr_2} for the case of $i>k$. Note that the last term in Figure~\ref{fig:trisection_contr_2} is precisely the tensor associated with the link configuration in Figure~\ref{fig:trisection_link_2}  (Right). In this case, the two configurations in the trisection and in the surgery link result in the same tensor.

\begin{figure}
\centering
\begin{tikzpicture}[thick, >= stealth]
\begin{scope}
\draw[red, myarrow = {0.4}{>}] (4,0) circle(0.7cm);
\draw[blue, myarrow = {0.3}{>}] (4,0) to (0,0);

\fill (4, 0.7) circle(2pt);
\fill (3.5, 0) circle(2pt);

\draw[green, myarrow = {0.5}{>}] (-1,0.7) to [out = 0, in =150] (0,0);
\draw[green, myarrow = {0.5}{>}]  (-1,0) to [out = 0, in =180] (0,0);
\draw[green, myarrow = {0.5}{>}]  (-1,-0.7) to [out = 0, in =-150] (0,0);

\draw[green, myarrow = {0.8}{>}] (4,0) to [out = 30, in =180] (5,0.7) to (5.3,0.7);
\draw[green, myarrow = {0.8}{>}] (4,0) to [out = 0, in =180] (5,0) to (5.3,0);
\draw[green, myarrow = {0.8}{>}] (4,0) to [out = -30, in =180] (5,-0.7) to (5.3,-0.7);

\draw[green, myarrow = {0.7}{>}] (1.5,1) to (1.5, -1);
\draw[green, myarrow = {0.7}{>}] (2.5,1) to (2.5, -1);

\fill[gray] (0,0) circle(0.3cm);
\fill[gray] (4,0) circle(0.3cm);
\end{scope}

\begin{scope}[xshift = 9cm]

\draw[green, myarrow = {0.8}{>}] (-1,0.7) to [out = 0, in =180] (0,0.3) to (3,0.3) to[out = 0, in = 180] (4, 0.7);
\draw[green, myarrow = {0.8}{>}] (-1,0) to (4,0);
\draw[green, myarrow = {0.8}{>}] (-1,-0.7) to [out = 0, in =180] (0,-0.3) to (3,-0.3) to[out = 0, in = 180] (4, -0.7);

\draw[myarrow = {0.8}{>},white, line width  = 2mm] (1.5,1) to (1.5, -1);
\draw[myarrow = {0.8}{>},white, line width  = 2mm] (2.5,1) to (2.5, -1);

\draw[green, myarrow = {0.8}{>}] (1.5,1) to (1.5, -1);
\draw[green, myarrow = {0.8}{>}] (2.5,1) to (2.5, -1);

\end{scope}
\end{tikzpicture}
\caption{A configuration of trisections (Left) near the $i$-th bridge for $i > k$ and the corresponding configuration of surgery links (Right).}\label{fig:trisection_link_2}
\end{figure}

\begin{figure}
\centering
\begin{tikzpicture}
\begin{scope}
\draw (0,0) node(muH) {$\mu_H$};
\draw (0,2) node(eH) {$e_H$};

\draw (1,0) node(MH) {$M_H$};
\draw (1,2) node(DeltaH) {$\Delta_H$};

\draw (2,1.5) node(S) {$S$};

\draw (2.5,-0.5) node(phi1bot) {$\phi_1$};
\draw (2.5,0) node(phi1mid) {$\phi_1$};
\draw (2.5,0.5) node(phi1top) {$\phi_1$};
\draw (2.5,2) node(phi2bot) {$\phi_2$};
\draw (2.5,2.5) node(phi2top) {$\phi_2$};

\draw (3.5, -0.5) node(Null1){};
\draw (3.5, 0) node(Null2){};
\draw (3.5, 0.5) node(Null3){};
\draw (3.5, 2) node(Null4){};
\draw (3.5, 2.5) node(Null5){};

\draw[->] (muH) to (MH);
\draw[->] (eH) to (DeltaH);
\draw[<-] (MH) to (phi1bot);
\draw[<-] (MH) to (phi1mid);
\draw[<-] (MH) to (phi1top);
\draw[->] (DeltaH) to (phi2bot);
\draw[->] (DeltaH) to (phi2top);
\draw[->] (DeltaH) to (S);
\draw[->] (S) to[out = -45, in = 45] (MH);
\draw[->] (phi1bot) to (Null1);
\draw[->] (phi1mid) to (Null2);
\draw[->] (phi1top) to (Null3);
\draw[->] (phi2bot) to (Null4);
\draw[->] (phi2top) to (Null5);
\end{scope}

\draw (4.5,1) node{$=$};

\begin{scope}[xshift = 5.5cm]

\draw (1,0) node(MH) {$M_H$};
\draw (1,2) node(DeltaH) {$\Delta_H$};

\draw (2.5,-0.5) node(phi1bot) {$\phi_1$};
\draw (2.5,0) node(phi1mid) {$\phi_1$};
\draw (2.5,0.5) node(phi1top) {$\phi_1$};
\draw (2.5,2) node(phi2bot) {$\phi_2$};
\draw (2.5,2.5) node(phi2top) {$\phi_2$};

\draw (3.5, -0.5) node(Null1){};
\draw (3.5, 0) node(Null2){};
\draw (3.5, 0.5) node(Null3){};
\draw (3.5, 2) node(Null4){};
\draw (3.5, 2.5) node(Null5){};

\draw[->] (MH) to[out = 180, in = 180, looseness = 1] (DeltaH);
\draw[<-] (MH) to (phi1bot);
\draw[<-] (MH) to (phi1mid);
\draw[<-] (MH) to (phi1top);
\draw[->] (DeltaH) to (phi2bot);
\draw[->] (DeltaH) to (phi2top);
\draw[->] (phi1bot) to (Null1);
\draw[->] (phi1mid) to (Null2);
\draw[->] (phi1top) to (Null3);
\draw[->] (phi2bot) to (Null4);
\draw[->] (phi2top) to (Null5);
\end{scope}

\draw (10,1) node{$=$};

\begin{scope}[xshift = 12cm]

\draw (0,0) node(phi1) {$\phi_1$};
\draw (0,2) node(phi2) {$\phi_2$};

\draw (1,0) node(DeltaKbot) {$\Delta_K$};
\draw (1,2) node(DeltaKtop) {$\Delta_K$};

\draw (2, -0.5) node(Null1){};
\draw (2, 0) node(Null2){};
\draw (2, 0.5) node(Null3){};
\draw (2, 2) node(Null4){};
\draw (2, 2.5) node(Null5){};

\draw[->] (phi1) to[out = 180, in = 180, looseness = 1] (phi2);
\draw[->] (phi1) to (DeltaKbot);
\draw[->] (phi2) to (DeltaKtop);
\draw[->] (DeltaKbot) to (Null1);
\draw[->] (DeltaKbot) to (Null2);
\draw[->] (DeltaKbot) to (Null3);
\draw[->] (DeltaKtop) to (Null4);
\draw[->] (DeltaKtop) to (Null5);
\end{scope}
\end{tikzpicture}
\caption{Contracting tensors of the trisection corresponding to the configuration in Figure \ref{fig:trisection_link_2} (Left) which results in a tensor the same as the one associated with the link in Figure \ref{fig:trisection_link_2} (Right).}\label{fig:trisection_contr_2}
\end{figure}

Therefore, we have $\langle T \rangle_{\Hopf{}} = \dim (H)^k L^T(H,K;\phi)$. By direct calculations,  the standard trisection diagram of $S^4$ has the invariant $\langle T_{\op{st}} \rangle_{\Hopf{}} = \dim(K)^2 \dim(H) \mu_K(\tilde{e})$. So, 
\begin{align*}
\langle T \rangle_{\Hopf{}} &= \dim (K)^g \left[\frac{\dim(H)\mu_K(\tilde{e})}{\dim(K)}\right]^{\frac{g}{3}}\tau_{\Hopf{}}(X).
\end{align*}
On the the other hand, 
\begin{align*}
\dim (H)^k L^T(H,K;\phi) &= \dim (H)^k \dim(K)^{g-k} \mu_K(\tilde{e})^{k} I_{\Rep(\phi)}(X) \\
                                        &= \dim (K)^g \left[\frac{\dim(H)\mu_K(\tilde{e})}{\dim(K)}\right]^{k}I_{\Rep(\phi)}(X).
\end{align*}
The theorem follows by combining the above two equations and noting that $\chi(X) = 2 + g - 3k$. 

\begin{remark}
We make a final remark. In this paper, we have focused on Hopf algebras in constructing invariants. There is also a weaker notion of Hopf algebras, namely, {\it weak} Hopf algebras. One can similarly formulate the notion of weak Hopf triplets using weak Hopf algebras. We believe that it is possible to define invariants of 4-manifolds more generally from weak Hopf triplets. It is known that every braided fusion category is equivalent to the category of representations of a weak Hopf algebra. Hence, it is expected that the resulting invariants from weak Hopf triplets would recover all cases of the Crane-Yetter invariants and a larger class of the dichromatic invariant than what is stated in Theorem \ref{thm:tri=dichro}.
\end{remark}
\newpage
\bibliographystyle{plain}
\bibliography{trisection_inv_final}

\end{document}